\documentclass[a4paper,twoside,11pt]{amsart}

\usepackage{amssymb,amsbsy,amsmath,amsfonts,eucal,amsthm,
amssymb,graphicx,times,xypic,color}

\usepackage[T1]{fontenc} 
\definecolor{blue}{cmyk}{1.,1.,0.,0.53}
\definecolor{red}{cmyk}{0.,1.,1.,0.53}
\definecolor{green}{cmyk}{1.,0.,1.,0.53}
\newcommand{\blue}{\textcolor{blue}}
\newcommand{\green}{\textcolor{green}}
\newcommand{\red}{\textcolor{red}}

\newtheorem{theorem}{Theorem}[section]
\newtheorem{proposition}{Proposition}[section]
\newtheorem{corollary}{Corollary}[section]
\newtheorem{lemma}{Lemma}[section]
\theoremstyle{definition}
\newtheorem{definition}{Definition}[section]
\theoremstyle{remark}

\newcommand{\explain}[1]{\text{\scriptsize\sf [#1]}}

\makeatletter
\renewcommand{\@fnsymbol}[1]
{\ensuremath{\ifcase#1\or $*$\or $**$\or $***$\or $****$\or $*****$
\else\@ctrerr\fi}}
\makeatother

\begin{document}

\title{Effective algebraic degeneracy\textsuperscript{$\ast$}}

\author{Simone Diverio \and Jo\"el Merker \and Erwan Rousseau}

\address{Simone Diverio --- Istituto \lq\lq Guido Castelnuovo\rq\rq{},
Sapienza Universit\`a di Roma}
\curraddr{Institut de Math\'ematiques de Jussieu, 
Universit\'e Pierre et Marie Curie, Paris}
\email{diverio@math.jussieu.fr}

\address{Jo\"el Merker --- D\'epartement de math\'ematiques 
et applications, \'Ecole Normale Sup\'erieure, Paris}
\email{merker@dma.ens.fr} 

\address{Erwan Rousseau --- D\'epartement de math\'ematiques, 
Universit\'e Louis Pasteur, Strasbourg}
\email{rousseau@math.unistra.fr}

\begin{abstract}
We show that for every smooth projective hypersurface $X \subset
\mathbb{ P}^{ n+1}$ of degree $d$ and of arbitrary dimension $n
\geqslant 2$, if $X$ is generic, then there exists a proper algebraic
subvariety $Y \subsetneqq X$ such that every nonconstant entire
holomorphic curve $f \colon \mathbb{ C} \to X$ has image $f (\mathbb{
C})$ which lies in $Y$, as soon as its degree satisfies the effective
lower bound $d \geqslant 2^{ n^5}$.
\end{abstract}

\maketitle

\thispagestyle{empty}


\begin{center}
\begin{minipage}[t]{11cm}
\baselineskip =0.35cm
{\scriptsize

\centerline{\bf Table of contents}

\smallskip

{\bf 1.~Introduction\dotfill 
\pageref{Section-1}.}

{\bf 2.~Preliminaries\dotfill 
\pageref{Section-2}.}

{\bf 3.~Algebraic degeneracy of entire curves 
\dotfill \pageref{Section-3}.}

{\bf 4.~Effectiveness of the degree lower bound
\dotfill \pageref{Section-4}.}

{\bf 5.~Estimations of the quantities ${\sf D}_k (n)$ and ${\sf D}_k'
(n)$
\dotfill \pageref{Section-5}.}

{\bf 6.~Effective bounds in dimensions 2, 3 and 4 through the invariant
theory approach
\dotfill \pageref{Section-6}.}

{\bf 7.~Effective algebraic degeneracy in dimensions 5 and 6
\dotfill \pageref{Section-7}.}

}\end{minipage}
\end{center}

\renewcommand{\thefootnote}{\fnsymbol{footnote}}

\footnotetext[1]{\,
The final publication is available at
{\scriptsize\sf www.springerlink.com}, with 
DOI: 10.1007/s00222-010-0232-4 at the
{\sl Inventiones Mathematic{\ae}}.
}

\renewcommand{\thefootnote}{\arabic{footnote}}

\section{Introduction}
\label{Section-1}

In 1979, Green and Griffiths~\cite{gg1980} conjectured that every
projective algebraic variety $X$ of general type contains a certain
{\em proper} algebraic {\em sub}variety $Y \subsetneqq X$ inside which
all nonconstant entire holomorphic curves $f \colon \mathbb{ C} \to X$
must necessarily lie.

A positive answer to this conjecture has been given for surfaces by
McQuillan~\cite{ mq1998} under the assumption that the second Segre
number $c_1^2-c_2$ is positive. In the survey article~\cite{ siu2004}
({\em cf.}  also~\cite{ siu2002}), Siu provided a beautiful strategy
to establish algebraic degeneracy of entire holomorphic curves in generic
hypersurfaces $X \subset \mathbb{ P}^{ n+1}$ of high degree larger
than a certain $d_n \gg 1$, and also {\em Kobayashi-hyperbolicity} of
such $X$'s if $d_n$ is even much higher.

Siu's strategy is based on two key steps: 1) the explicit
construction, in projective coordinates, of global holomorphic jet
differentials; 2) the deformation of such jet differentials by means
of slanted vector fields having low pole order. The explicit
construction of jet differentials can be seen as a replacement of the
argument using Riemann-Roch which is known to be difficult to realize
since it involves a control of the cohomology. The reason to perform
explicit constructions is also a better access to the base-point set,
in order to provide hyperbolicity instead of just algebraic
degeneracy. Complete up-to-date survey considerations may further be
found in~\cite{siuyeu1996, dem1997, mq1999, deg2000, kob1988,
voi2003}.

In this paper, we overcome the difficulty of the Riemann-Roch argument
thanks to an alternative approach for Siu's first key step based on
Demailly's bundle of invariant jets~\cite{ dem1997}. The advantage of
this method is also that it usually yields better bounds on the
degree. Indeed, after performing in Sections~\ref{Section-4} 
and~\ref{Section-5} below some
explicit, delicate elimination computations, we finally obtain a lower
bound on the degree $d_n = d ( n)$ as an explicit function of $n$, for
generic projective hypersurfaces of arbitrary dimension $n \geqslant
2$.

\begin{theorem}\label{main}
Let $X\subset \mathbb{ P}^{ n+1}$ be a smooth projective hypersurface
of degree $d$ and of arbitrary dimension $n \geqslant
2$. If $X$ is generic and if its degree satisfies the
\emph{effective} lower bound:
\[
d
\geqslant
2^{n^5},
\]
then there exists a \emph{proper} algebraic subvariety $Y \subsetneqq
X$ such that every nonconstant entire holomorphic curve $f \colon
\mathbb{C} \to X$ has image $f (\mathbb{ C })$ contained in $Y$.
\end{theorem}

As in~\cite{ siu2002, siu2004}, we thereby confirm, for generic projective
hypersurfaces of high degree, the Green-Griffiths-Lang conjecture. Even if
our lower bound is far from the one $\deg X \geqslant n+3$ insuring general
type, to our knowledge, Theorem~\ref{main} is, in this direction, the first
$n$-dimensional result with, moreover, an explicit degree lower bound. In
addition, as a byproduct of our constructions, the subvarieties absorbing
the images of nonconstant entire curves vary as a holomorphic family
with the generic projective hypersurface.

Two main ingredients enter our proof: 1) the existence of invariant
jet differentials vanishing on an ample divisor in projective
hypersurfaces of high degree, following~\cite{ dem1997, div2008a}; and
Siu's second key step: 2) the global generation of a sufficiently high
twisting of the tangent bundle to the so-called {\sl manifold of
vertical $n$-jets}, which is canonically associated to the universal
family of projective hypersurfaces, following~\cite{ siu2004,
mer2008a}.

The first ingredient dates back to the seminal work of Bloch
\cite{blo1926}, revisited by Green-Griffiths in~\cite{gg1980}, by
Siu in~\cite{ siu1995, siuyeu1996,
siu2004} and by Demailly in~\cite{ dem1997}. Bloch's main
philosophical idea is that global jet differentials vanishing on an
ample divisor provide some algebraic differential equations that every
entire holomorphic curve $f \colon \mathbb{ C} \to X$ must
satisfy. Five decades later, Green and Griffiths~\cite{gg1980}
modernized Bloch's concepts and established several
results\,\,---\,\,still fundamental nowadays\,\,---\,\,about the
geometry of entire curves.

Later on, Demailly~\cite{dem1997} refined and enlarged the whole
theory by defining jet differentials that are invariant under
reparametrization of the source $\mathbb{ C}$. Through this
geometrically adequate, new point of view, one looks only at the
conformal class of all entire curves. In~\cite{ div2008a, div2008b},
the first-named author combined Demailly's approach with Trapani's
\cite{tra1995} algebraic version of the holomorphic Morse
inequalities, so as to construct global invariant jet differentials in
{\em any} dimension $n \geqslant 2$. The first effective aspect of our
proof is to make somewhat explicit such a construction.

Indeed, by following~\cite{ div2008a, div2008b}, we consider a certain
intersection product ({\em see}~\thetag{\ref{morseintersection}}
and~\thetag{\ref{pi-delta}} below), the positivity of which
yields\,\,---\,\,thanks to a suitable application of the holomorphic
Morse inequalities\,\,---\,\,a lower bound for the (asymptotic)
dimension of the space of global sections of a certain {\em weighted
subbundle} of Demailly's full bundle $E_{ n, m} T_X^*$ of invariant
$n$-jet differentials. This intersection product lives in the
cohomology algebra of the $n$-th projectivized jet bundle over $X$, a
polynomial algebra in $n$ indeterminates $u_1, u_2, \dots, u_n$
equipped with canonical, geometrically significant relations (\cite{
dem1997, div2008a}). The $u_i$ here are the first Chern classes of the
successive (anti)tautological line bundles which arise during the
projectivization process. The task of reducing the mentioned
intersection product in terms of the Chern classes of
$T_X$\,\,---\,\,after eliminating {\em all} the Chern classes living
at each level of Demailly's tower\,\,---\,\,happens to be of high
algebraic complexity, because four combinatorics are intertwined
there: 1) presence of several relations shared by all the Chern
classes of the lifted horizontal distributions; 2) Newton expansion of
large $n^2$-powers; 3) differences of various binomial coefficients;
4) emergence of many Jacobi-Trudy determinants.

The second ingredient, {\em viz.} the {\sl vertical jets}, comes from
ideas developed for 1-jets by Voisin [24] in order to generalize works 
of Clemens~\cite{ cle1986} and Ein on the positivity of the canonical 
bundles of subvarieties of generic projective hypersurfaces of high 
degree. In~\cite{ siu2004}, Siu showed how the corresponding
{\sl global generation property} for $1$-jets devised by Voisin
generalizes to the bundle of tangents to the space of vertical
$n$-jets. Siu then established that one may use the available tangential
generators, which are meromorphic vector fields with a certain {\sl
pole order} $c_n \geqslant 1$, so as to produce, by plain
differentiation, many new algebraically independent invariant jet
differentials when starting from just a single {\em nonzero} jet
differential. At the end, one obtains in this way sufficiently many
independent jet differentials, and this then forces entire curves to
lie in a positive-codimensional subvariety $Y \subsetneqq X$.

This strategy was realized in details for 2-jets in dimension 2 by P\u
aun~\cite{pau2008} with pole order $c_2 = 7$, and similarly, for
3-jets in dimension 3 by the third-named author in~\cite{ rou2007}
with $c_3 = 12$. In both works, global generation holds outside a
certain exceptional set. The general case of $n$-jets in dimension $n$
was performed recently by the second-named author in~\cite{ mer2008a}
with $c_n = \frac{ n^2 + 5 n}{ 2}$ and with a quite similar
exceptional set. It then became clear, when ~\cite{ mer2008a}
appeared, that Demailly's invariant jets combined with Siu's second
key step could yield {\em weak} algebraic degeneracy (nonexistence of
Zariski-dense entire curves) in {\em any} dimension $n \geqslant
2$. But to reach effectivity, it yet remained to perform what the
present article is aimed at: taming somehow the complicated
combinatorics of Demailly's tower. Furthermore, at the cost of
increasing the pole order up to $c_n' = n^2 + 2n$, the exceptional set
is shrunk to be just the set of singular jets (\cite{ mer2008a}), and
then {\sl strong effective} algebraic degeneracy is gained. This is
Theorem~\ref{main}.

These brief words summarize how we combine {\em several ideas}, both
of {\em conceptual} and of {\em technical} nature which stem from
Algebra, from Analysis and from Geometry; deep conjectures always
confirm the unity of mathematics.

As the effective lower bound $\deg X \geqslant 2^{ n^5}$ of the main
theorem above is not optimal, Sections~\ref{Section-6}
and~\ref{Section-7} of the paper are intended to provide numerically
better estimates in small dimensions. For surfaces, the best known
effective lower bound for the degree is $d \geqslant 18$
(\cite{pau2008}), after $d \geqslant 21$ (\cite{ deg2000}) and $d
\geqslant 36$ (\cite{ mq1999}). In \cite{ rou2007}, the third-named
author obtained the first effective result for weak algebraic
degeneracy of entire curves inside threefolds $X$ of $\mathbb{ P}^4$,
whenever $\deg X \geqslant 593$.

\begin{theorem}
\label{lowdim}
Let $X \subset \mathbb{ P }^{ n+1 }$ be a smooth projective
hypersurface of degree $d$.  If $X$ is generic, then there exists a
proper closed subvariety $Y \subsetneqq X$ such that every nonconstant
entire holomorphic curve $f \colon \mathbb{ C } \to X$ has image $f
(\mathbb{ C})$ contained in $Y$

\begin{itemize}
\item for $\dim X=3$, whenever $\deg X \geqslant 593${\rm ;} 
\item for $\dim X=4$, whenever $\deg X \geqslant 3203${\rm ;}
\item for $\dim X=5$, whenever $\deg X \geqslant 35355${\rm ;}
\item for $\dim X=6$, whenever $\deg X \geqslant 172925$.
\end{itemize}
\end{theorem}

The last three effective lower bounds in dimensions 4, 5 and 6 are
entirely new. In dimension $3$, our bound 593 is the same as in~\cite{
rou2007}. Indeed, an inspection of the exceptional set in~\cite{
rou2007} shows that the part of the degeneracy locus which may depend
on $f$ is in fact of codimension~2 ({\em cf.}~\cite{ mer2008a}), and
therefore is empty, thanks to Clemens' result~\cite{ cle1986} which
excludes elliptic and rational curves. Using $c_4 = 18$ and $c_5 =
25$ instead of $c_4 ' = 24$ and $c_5 ' = 35$, we would have obtained
the two lower bounds $\deg X \geqslant 2432$ and $\deg X \geqslant
25586$ which were announced in our first {\tt arxiv.org} preprint and
which insured only {\em weak} algebraic degeneracy ({\em cf.}~\cite{
mer2008a}; using $c_6 = 33$ instead of $c_6 ' = 48$, the bound would
be $\deg X \geqslant 120176$).

For dimensions 5 and 6, our strategy of proof is the same as for
Theorem~\ref{main}, except that we choose a numerically better weighted
subbundle of Demailly's bundle of invariant jet differentials, exactly
as in~\cite{ div2008a}.

Quite differently, for dimensions $3$ and $4$, the construction of
nonzero jet differentials is based on a {\em complete} algebraic
description of the full Demailly bundles $E_{n,m} T^*_X$, $n =3, 4$,
due respectively to the third-named author (\cite{ rou2006a}) and to
the second-named author (\cite{ mer2008b}), after Demailly~\cite{
dem1997} and Demailly-El Goul~\cite{ deg2000} for $n = 2$. The
invariant theory approach requires finding the composition series of
the $E_{ n, m} T_X^*$, but this is understood only in dimensions 2, 3
and 4, because of the proliferation of secondary
invariants\,\,---\,\,a well known phenomenon, {\em cf.}~\cite{
mer2008b} and the references therein. Then by appropriately summing
the Euler characteristics of the composing Schur bundles~\cite{
rou2006a}, taking account of the numerous syzygies shared by a
collection of fundamental bi-invariants~\cite{ mer2008b}, one
establishes the positivity of the Euler characteristics $\chi \big(
E_{n,m} T^*_X \big)$ for $n = 3, 4$, at least asymptotically as $m$
goes to infinity. Furthermore, realizing also in dimension 4 the
strategy finalized in dimension 3 by the third-named author~\cite{
rou2006b}, we estimate from above the contribution of the even
cohomology dimensions $h^{ 2i} \big( X, E_{ n,m} T_X^* \big)$, thereby
gaining a suitable lower bound for the dimension of the space $h^0
\big( X, E_{ n,m} T_X^* \big)$ of global sections. Such estimates are
done by means of Demailly's~\cite{dem1997} generalization of
a vanishing theorem due to Bogomolov for the top cohomology,
and also by means of the algebraic version of the weak holomorphic
Morse inequalities for the intermediate cohomologies~\cite{ rou2006b}.

Even if the numerical bounds obtained in this way in dimensions 3 and
4 are better than the ones we obtained in all dimensions, the extreme
intricacy of the algebras of invariants by reparametrization ({\em
cf.}~\cite{ mer2008b}) is the main obstacle to run the process in the
higher dimensions $n \geqslant 5$. This was our central motivation to
follow the strategy of~\cite{ div2008a, div2008b}.

\subsection*{ Acknowledgments} The first-named author warmly 
thanks Stefano Trapani for patiently listening all the details of the
proof of the main theorem. 

\section{Preliminaries}
\label{Section-2}

\subsection{Jet differentials}
We briefly present here useful geometric concepts selected from the
theory of Green-Griffiths' and Demailly's jets~\cite{gg1980, dem1997}
({\em cf.} also~\cite{ rou2006a, div2008a}). Let $(X,V)$ be a {\sl
directed manifold}, {\em i.e.} a pair consisting of a complex manifold
$X$ together with a (not necessarily integrable) holomorphic subbundle
$V \subset T_X$ of the tangent bundle to $X$. This category will be
very useful later on, when we will consider the situation where $X$ is
the universal family of projective hypersurfaces of fixed degree and
$V$ the relative tangent bundle to the family. The bundle $J_kV$ is
the bundle of $k$-jets of germs of holomorphic curves $f \colon(
\mathbb{ C},0) \to X$ which are tangent to $V$, {\em i.e.}, such that
$f'(t)\in V_{ f(t) }$ for all $t$ near $0$, together
with the projection map $f \mapsto f(0)$ onto $X$.

Let $\mathbb{ G}_k$ be the group of germs of $k$-jets of biholomorphisms
of $(\mathbb{ C},0)$, that is, the group of germs of biholomorphic maps
\[
t\mapsto\varphi(t)=
a_1\,t+a_2\,t^2
+\cdots+
a_k\,t^k,\quad 
a_1\in\mathbb{C}^*,\,\,
a_j\in\mathbb{C},\,\,j\geqslant 2
\]
of $(\mathbb{ C}, 0)$, the composition law being taken modulo terms
$t^j$ of degree $j>k$. Then $\mathbb{ G}_k$ admits a natural fiberwise
right action on $J_kV$ which consists in reparametrizing $k$-jets of
curves by such changes $\varphi$ of parameters. In~\cite{ mer2008a},
one finds the multivariate Faà di Bruno formulae yielding explicit
reparametrization for the so-called
{\sl absolute case} $V = T_X$. Moreover the
subgroup $\mathbb{ H} \simeq \mathbb{ C}^*$ of homotheties $\varphi(t)
= \lambda\, t$ is a (non-normal) subgroup of $\mathbb{ G}_k$ and we have
a semidirect decomposition $\mathbb{ G}_k = \mathbb{ G}_k' \ltimes
\mathbb{ H}$, where $\mathbb{ G}_k'$ is the group of $k$-jets of
biholomorphisms tangent to the identity, {\em i.e.} with $a_1 =
1$. The corresponding action on $k$-jets is described in coordinates
by
\begin{equation}\label{omothety}
\lambda\cdot
\big(f',f'',\dots,f^{(k)}\big)
=
\big(
\lambda f',\lambda^2 f'',\dots,\lambda^k f^{(k)}
\big).
\end{equation}
As in~\cite{gg1980}, we introduce the {\sl Green-Griffiths vector
bundle $E_{k, m}^{ GG }V^*\to X$}, the fibers of which are
complex-valued polynomials $Q(f', f'', \dots, f^{ (k)})$ in the fibers
of $J_kV$ having weighted degree $m$ with respect to the $\mathbb{
C}^*$ action, namely such that:
\[
Q
\big(
\lambda f',\lambda^2 f'',\dots,\lambda^k f^{(k)}
\big)
=
\lambda^mQ
\big(f',f'',\dots,f^{(k)}\big),
\]
for all $\lambda \in \mathbb{ C}^*$ and all $\big(
f',f'',\dots,f^{(k)} \big) \in J_k V$. Demailly
refined this concept.

\begin{definition}[\cite{dem1997}]
\label{invariant-jets}
The {\sl bundle of invariant jet differentials of order $k$ and
weighted degree $m$} is the subbundle $E_{k,m}V^*\subset
E_{k,m}^{GG}V^*$ of polynomial differential operators
$Q(f',f'',\dots,f^{(k)})$ which are invariant under {\em arbitrary}
changes of parametrization, {\em i.e.} which, for every $\varphi \in
\mathbb{ G}_k$, satisfy:
\[
Q\big(
(f\circ\varphi)',(f\circ\varphi)'',\dots,(f\circ\varphi)^{(k)}
\big)
=
\varphi'(0)^m\,
Q\big(
f',f'',\dots,f^{(k)}
\big).
\]
Alternatively, $E_{k,m}V^* = \big( E_{k,m}^{ GG}V^* \big)^{ \mathbb{
G}_k'}$ is the set of invariants of $E_{k,m}^{ GG}V^*$ under the
action of $\mathbb{ G}'_k$.
\end{definition}

We now define a filtration on $E_{k,m}^{GG}V^*$. A coordinate change
$f\mapsto \Psi \circ f$ transforms every monomial $( f^{ (\bullet) }
)^\ell = (f')^{\ell_1} (f'')^{\ell_2} \cdots ( f^{(k)} )^{ \ell_k}$
having, for any $s$ with $1 \leqslant s\leqslant k$, the partial
weighted degrees $\vert\ell\vert_s := 
\vert \ell_1 \vert + 2 \vert \ell_2 \vert + \cdots +
s \vert \ell_s \vert$, 
into a new polynomial $\big( ( \Psi \circ f)^{( \bullet)}
\big)^\ell$ in $(f', f'', \dots, f^{ ( k) })$, which has the same
partial weighted degree of order $s$ when $\ell_{ s+1 } = \cdots =
\ell_k=0$, and a larger or equal partial degree of order $s$ otherwise
(use the chain rule). Hence, for each $s=1, \dots, k$, we get a well
defined decreasing filtration $F_s^\bullet$ on $E_{ k,m }^{ GG}V^*$ as
follows:
\[
F_s^p\big(E_{k,m}^{GG}V^*\big)
=
\left\{
\begin{matrix}
\text{$Q(f',f'',\dots,f^{(k)})\in E_{k,m}^{GG}V^*$ involving} 
\\
\text{only monomials $(f^{(\bullet)})^\ell$ 
with $\vert\ell\vert_s\geqslant p$}
\end{matrix}
\right\},
\quad\forall\,p\in\mathbb{N}.
\]
The graded terms $\operatorname{ Gr}^p_{ k-1} \big( E_{ k,m}^{ GG}V^*
\big)$ associated with the $(k-1)$-filtration $F_{ k-1 }^p( E_{k,m }^{
GG }V^*)$ are the homogeneous polynomials $Q(f',f'',\dots,f^{(k)})$
all the monomials $( f^{( \bullet )})^\ell$ of which have partial
weighted degree $\vert\ell\vert_{k-1} = p$; hence, their degree
$\ell_k$ in $f^{(k)}$ is such that $m-p = k\ell_k$ and $\operatorname{
Gr}^p_{ k-1}( E_{k,m }^{ GG}V^*)=0$ unless $k\vert m - p$. Looking at
the transition automorphisms of the graded bundle induced by the
coordinate change $f\mapsto \Psi \circ f$, it turns out that $f^{(k)}$
transforms as an element of $V\subset T_X$ and, by means of a simple
computation, one finds
\[
\operatorname{Gr}^{m-k\ell_k}_{k-1}
\big(E_{k,m}^{GG}V^*\big)
=
E_{k-1,m-k\ell_k}^{GG}V^*\otimes S^{\ell_k}V^*.
\]
Combining all filtrations $F^\bullet_s$ together, we find inductively
a filtration $F^\bullet$ on $E_{ k,m }^{ GG}V^*$ the graded terms of
which are
\[
\operatorname{Gr}^{\ell}
\big(
E_{k,m}^{GG}V^*
\big)
=
S^{\ell_1}V^*\otimes S^{\ell_2}V^*
\otimes\cdots\otimes 
S^{\ell_k}V^*,\quad\ell
\in\mathbb{N}^k,\ \
\vert\ell\vert_k=m.
\]
Moreover (\cite{ dem1997}), invariant jet differentials enjoy the
natural induced filtration:
\[
F^p_s(E_{k,m}V^*)
=
E_{k,m}V^*\cap F_{s}^p
\big(
E_{k,m}^{GG}V^*
\big),
\]
the associated graded bundle being, if we employ $( \bullet)^{
\mathbb{ G}_k'}$ to denote $\mathbb{ G}_k'$-invariance:
\[
\operatorname{Gr}^\bullet(E_{k,m}V^*)
=
\bigg(
\bigoplus_{\vert\ell\vert_k=m}
S^{\ell_1}V^*\otimes S^{\ell_2}V^*
\otimes\cdots\otimes S^{\ell_k}V^*
\bigg)^{\mathbb{G}_k'}.
\]

\subsection{Projectivized $k$-jet bundles}
Next, we recall briefly Demailly's construction~\cite{ dem1997} of the
tower of projectivized bundles providing a (relative) smooth
compactification of $J^{\text{ reg}}_kV / \mathbb{ G}_k$, where $J^{
\text{ reg}}_kV$ is the bundle of {\sl regular $k$-jets tangent to
$V$}, that is, $k$-jets such that $f'(0)\neq 0$.

Let $(X,V)$ be a directed manifold, with $\dim X = n$ and $\text{\rm
rank} \, V = r$. With $(X, V)$, we associate another directed manifold
$( \widetilde X, \widetilde V)$ where $\widetilde X = P(V)$ is the
projectivized bundle of lines of $V$, $\pi \colon \widetilde X \to X$
is the natural projection and $\widetilde V$ is the subbundle of
$T_{\widetilde X}$ defined fiberwise as
\[
\widetilde V_{(x_0,[v_0])}\overset{\text{def}}
=
\big\{
\xi\in T_{\widetilde X,(x_0,[v_0])}\mid
\pi_*\xi\in\mathbb{C}\cdot v_0
\big\},
\]
for any $x_0 \in X$ and $v_0\in T_{ X, x_0} \setminus\{ 0 \}$. We also
have a \lq\lq lifting\rq\rq{} operator which assigns to a germ of
holomorphic curve $f\colon( \mathbb{ C}, 0)\to X$ tangent to $V$ a
germ of holomorphic curve $\widetilde f \colon( \mathbb{ C}, 0) \to
\widetilde X$ tangent to $\widetilde V$ in such a way that
$\widetilde{ f}(t) = (f(t), [f'(t)])$.

To construct the projectivized $k$-jet bundle we simply set
inductively $(X_0,V_0) = (X,V)$ and $(X_k, V_k) = (\widetilde X_{k-1},
\widetilde V_{k-1})$. Clearly $\text{\rm rank} \, V_k = r$ and $\dim
X_k = n+k(r-1)$. Of course, we have for each $k > 0$ a tautological
line bundle $\mathcal{ O}_{X_k}(-1)\to X_k$ and a natural projection
$\pi_k \colon X_k\to X_{ k-1}$. We call $\pi_{ j,k}$ the composition
of the projections $\pi_{ j+1} \circ \cdots \circ \pi_{k}$, so that
the total projection is given by $\pi_{ 0, k} \colon X_k \to X$. We
have, for each $k>0$, two short exact sequences
\begin{equation}\label{ses1}
0\to T_{X_k/X_{k-1}}\to V_k\to\mathcal{O}_{X_k}(-1)\to 0,
\end{equation}
\begin{equation}\label{ses2}
0\to\mathcal{O}_{X_k}\to\pi_k^*V_{k-1}
\otimes\mathcal{O}_{X_k}(1)\to T_{X_k/X_{k-1}}\to 0.
\end{equation}
Here, we also have an inductively defined $k$-lifting for germs of
holomorphic curves such that $f_{[k]} \colon( \mathbb{ C}, 0) \to X_k$
is obtained as $f_{[k]} = \widetilde f_{ [k-1] }$.

\begin{theorem}[\cite{dem1997}]
Suppose that $\text{\rm rank} \, V\geqslant 2$. The quotient $J_k^{
\text{\rm reg}} V \big/ \mathbb{ G}_k$ has the structure of a locally
trivial bundle over $X$, and there is a holomorphic embedding $J_k^{
\text{\rm reg}} V \big/ \mathbb{ G}_k \hookrightarrow X_k$ over $X$,
which identifies $J_k^{ \text{\rm reg}} V \big/ \mathbb{ G}_k$ with
$X_k^{\text{\rm reg}}$, that is the set of points in $X_k$ of the form
$f_{ [k]}(0)$ for some non singular $k$-jet $f$. In other words $X_k$
is a relative compactification of $J_k^{ \text{\rm reg}} V/ \mathbb
G_k$ over $X$. Moreover, one has the direct image formula:
\[
(\pi_{0,k})_*\mathcal{ O}_{X_k}(m)
=
\mathcal{ O}\big(
E_{k,m}V^*
\big).
\]
\end{theorem}

Next, we are in position to recall the fundamental application of jet
differentials to Kobayashi-hyperbolicity and to Green-Griffiths
algebraic degeneracy.

\begin{theorem}[\cite{gg1980, siuyeu1996, dem1997}]\label{A}
Assume that there exist integers $k,m>0$ and an ample line bundle
$A\to X$ such that
\[
H^0\big(
X_k,\mathcal{ O}_{X_k}(m)
\otimes
\pi_{0,k}^*A^{-1}
\big)
\simeq 
H^0\big(
X,E_{k,m}V^*\otimes A^{-1}
\big)
\]
has non zero sections $\sigma_1, \dots, \sigma_N$ and let $Z \subset
X_k$ be the base locus of these sections. Then every entire
holomorphic curve $f \colon \mathbb{ C}\to X$ tangent to $V$
necessarily satisfies $f_{ [k] }( \mathbb{ C }) \subset Z$. In other
words, for every global $\mathbb{ G}_k$-invariant differential
equation $P$ vanishing on an ample divisor, every entire holomorphic
curve $f$ must satisfy the algebraic differential equation 
$P\big( j^k \! f (t) \big) \equiv 0$.
Furthermore, the same result also holds true for the bundle $E_{ k,m
}^{ GG } T^*_X$.
\end{theorem}

\subsection{Existence of invariant jet differentials}
Now, we recall some results obtained by the first-named author in
\cite{div2008b}, concerning the existence of invariant jet differentials
on projective hypersurfaces which generalized to all dimensions $n$
previous works by Demailly~\cite{ dem1997} and of the third-named
author~\cite{rou2006b}.

Denote by $c_\bullet(E)$ the total Chern class of a vector bundle
$E$. The two short exact sequences (\ref{ses1}) and (\ref{ses2}) give,
for each $k>0$, the following two formulae:
\[
\aligned
c_\bullet(V_k)
&
=
c_\bullet\big(T_{X_k/X_{k-1}}\big)\,
c_\bullet\big(\mathcal{ O}_{X_k}(-1)\big)
\\
c_\bullet
\big(\pi_k^*V_{k-1}
\otimes
\mathcal{ O}_{X_k}(1)\big)
&
=
c_\bullet\big(
T_{X_k/X_{k-1}}\big),
\endaligned
\]
so that by a plain substitution:
\begin{equation}\label{chern1}
c_\bullet(V_k)
=
c_\bullet\big(\mathcal{ O}_{X_k}(-1)\big)\,
c_\bullet\big(\pi_k^*V_{k-1}
\otimes\mathcal{ O}_{X_k}(1)
\big).
\end{equation}
Let us call $u_j = c_1 \big( \mathcal{ O}_{ X_j}(1) \big)$ and
$c_l^{[j]} = c_l(V_j)$. With these notations, \thetag{\ref{chern1}}
becomes:
\begin{equation}\label{c-l-k}
c_l^{[k]}
=
\sum_{s=0}^l
\big[
{\textstyle{\binom{n-s}{l-s}}}
-
{\textstyle{\binom{n-s}{l-s-1}}}
\big]
u_k^{l-s}\cdot\pi_k^*c_s^{[k-1]},\quad 1\leqslant l\leqslant r.
\end{equation}
Since $X_j$ is the projectivized bundle of line of $V_{j-1}$, we also
have the polynomial relations
\begin{equation}\label{u-n}
u_j^r+\pi_j^*c_1^{[j-1]}
\cdot u_j^{r-1}+\cdots+
\pi_j^*c_{r-1}^{[j-1]}\cdot u_j+\pi_j^*c_{r}^{[j-1]}
=
0,
\quad 1\leqslant j\leqslant k.
\end{equation}
After all, the cohomology ring of $X_k$ is defined in terms of
generators and relations as the polynomial algebra $H^\bullet( X)[
u_1, \dots, u_k]$ with the relations (\ref{u-n}) in which, using
inductively~\thetag{\ref{c-l-k}}, one may express in advance all the
$c_l^{ [j]}$ as certain polynomials with integral coefficients in the
variables $u_1, \dots, u_j$ and $c_1(V), \dots, c_l(V)$. In
particular, for the first Chern class of $V_k$, a simple explicit
formula is available:
\begin{equation}\label{c1}
c_1^{[k]}
=
\pi_{0,k}^*c_1(V)+(r-1)
\sum_{s=1}^k\,\pi_{s,k}^*\,u_s.
\end{equation}
Also, it is classically known that the Chern classes $c_j ( X)$ of a
smooth projective hypersurface $X \subset \mathbb{ P}^{ n+1}$ are
polynomials in $d := \deg X$ and the hyperplane class $h := c_1 \big(
\mathcal{ O}_{ \mathbb{ P }^{ n+1}} ( 1) \big)$, {\em viz.} for $1
\leqslant j \leqslant n$:
\begin{equation}
\label{c-d}
c_j(X)
=
c_j(T_X)
=
(-1)^j\,h^j\,
\sum_{i=0}^j\,(-1)^i\,
{\textstyle{\binom{n+2}{i}}}\,d^{j-i}.
\end{equation} 

Now, let $X \subset \mathbb{ P}^{ n+1}$ be a smooth projective
hypersurface of degree $\deg X = d$ and consider, for all what follows
in the sequel, the absolute case $V = T_X$ with jet order $k = n$
equal to the dimension. Given any $\bold a = (a_1, \dots,a_n) \in
\mathbb{ Z}^n$, we define ({\em cf.}~\cite{ dem1997, div2008a}) the
following line bundle $\mathcal O_{X_n} (\bold a)$ on $X_n$:
\[
\mathcal{ O}_{X_n}(\bold a)
=
\pi_{1,n}^*\mathcal{ O}_{X_1}(a_1)
\otimes\pi_{2,n}^*\mathcal{ O}_{X_2}(a_2)
\otimes\cdots\otimes
\mathcal{ O}_{X_n}(a_n).
\] 
Using the algebraic version\,\,---\,\,first appeared in Trapani's
article~\cite{tra1995}\,\,---\,\,of Demailly's holomorphic Morse
inequalities, the first-named author showed in~\cite{div2008b} that,
in order to check the {\em bigness} of $\mathcal{ O}_{ X_n}(1)$, it
suffices to show the {\em positivity}, for some $\mathbf{ a} = (a_1,
\dots, a_n) \in \mathbb{ N}^n$ lying arbitrarily in the cone defined
by:
\begin{equation}\label{rel-nef-cone}
a_1\geqslant 3a_2,\dots,a_{n-2}
\geqslant 
3a_{n-1}\quad\text{\rm and}
\quad a_{n-1}
\geqslant 
2a_n\geqslant 1,
\end{equation}
of the following intersection product:
\[
F^N-N\,F^{N-1}
\cdot G,
\]
where $N = \dim X_n = n^2$, and where the two bundles $F := \mathcal{
O}_{ X_n} ( {\bf a}) \otimes \pi_{ 0, n}^* \mathcal{ O}_X ( 2 \vert
{\bf a} \vert)$ and $G := \pi_{ 0, n}^* \mathcal{ O}_X ( 2 \vert {\bf
a} \vert)$ are both globally nef on $X_n$ 
(\cite{div2008b}, Proposition~2); here,
$\mathcal O_{ X}(1)$ is the hyperplane bundle over $X$ and we
abbreviate $\vert \mathbf{ a} \vert := a_1 +\cdots+ a_n$.  In other
words, we express $\mathcal{ O}_{ X_n}( \mathbf{ a})$ as a
``difference'' $F \otimes G^{ -1}$
between two nef line bundles over $X_n$:
\[
\mathcal{ O}_{X_n}(\mathbf{a})
=
\big(\mathcal{ O}_{X_n}(\mathbf{a})\otimes\pi^*_{0,n}
\mathcal{ O}_X(2\vert\mathbf{a}\vert)\big)
\otimes
\big(\pi^*_{0,n}\mathcal{ O}_X(2\vert\mathbf{a}\vert)\big)^{-1}.
\] 
Thus in sum, we have to find some ${\bf a} \in \mathbb{ Z}^n$ lying in
the cone~\thetag{ \ref{rel-nef-cone}} for which the concerned
intersection product written in length:
\begin{multline}
\label{morseintersection}
\aligned
\big(\mathcal{ O}_{X_n}
(\mathbf{a})
&
\otimes\pi^*_{0,n}
\mathcal{ O}_X(2\vert\mathbf{a}\vert)\big)^{n^2}
-
\\
&
-n^2\big(\mathcal{O}_{X_n}(\mathbf{a})
\otimes\pi^*_{0,n}\mathcal{O}_X(2\vert\mathbf{a}\vert)\big)^{n^2-1}
\cdot\pi^*_{0,n}\mathcal{O}_X
(2\vert\mathbf{a}\vert)
\endaligned
\end{multline}
is positive.  This was done by the first-named author, and an
application of the mentioned Morse inequalities yielded
the following.

\begin{theorem}[\cite{div2008b}]\label{existence}
Let $X \subset \mathbb{ P}^{n + 1}$ by a smooth complex hypersurface
of degree $\deg X = d$ and fix any ample line bundle $A\to X$. Then,
for jet order $k = n$ equal to the dimension, there exists a positive
integer $d_n$ such that the two isomorphic spaces of sections:
\[
H^0\big(
X_n,\mathcal{ O}_{X_n}(m)
\otimes\pi^*_{0,n}A^{-1}\big)
\simeq
H^0\big(
X,E_{n,m}T^*_X\otimes A^{-1}
\big)\neq 0,
\]
are {\em nonzero}, whenever $d \geqslant d_n$ provided that $m$ is
large enough.
\end{theorem}

It is also proved in~\cite{div2008a} that for any jet order $k < n$
{\em smaller} than the dimension, no nonzero sections, though, are
available: $H^0 \big( X_k, \, \mathcal{ O}_{ X_k} ( m) \otimes \pi_{
0, k}^* A^{ -1} \big) = 0$; in fact, this vanishing property is
used as a technical tool in the proof of 
Theorem~\ref{existence}.

\smallskip

In our applications, it will be crucial to be able to control in a
more precise way the order of vanishing of these differential
operators along the ample divisor. Thus, we shall need here a slightly
different theorem, inspired from~\cite{ siu2004, pau2008, rou2007}.
Recall at first that for $X$ a smooth projective hypersurface of
degree $d$ in $\mathbb{ P }^{ n+1 }$, the canonical bundle has the
following expression in terms of the hyperplane bundle:
\[
K_X\simeq \mathcal{ O}_{X}(d-n-2),
\]
whence it is ample as soon as $d \geqslant n+3$. 

\begin{theorem}\label{existenceKX}
Let $X \subset \mathbb{ P}^{ n+1}$ by a smooth complex hypersurface of
degree $\deg X = d$. Then, for all positive rational numbers $\delta$
small enough, there exists a positive integer $d_n$ such that the
space of twisted jet differentials:
\[
H^0\big(X_n,\mathcal{ O}_{X_n}(m)
\otimes
\pi^*_{0,n}K_X^{-\delta m}\big)
\simeq
H^0\big(X,E_{n,m}T^*_X\otimes 
K_X^{-\delta m}\big)\neq 0,
\]
is nonzero, whenever $d \geqslant d_{ n, \delta}$ provided again that
$m$ is large enough and that $\delta m$ is
an integer.
\end{theorem}

Observe that all nonzero sections $\sigma \in H^0 \big(X, E_{ n,m }
T^*_X \otimes K_X^{-\delta m })$ then have vanishing order at least
equal to $\delta m (d - n - 2)$, when viewed as sections
of $E_{ n,m} T_X^*$. 

\begin{proof}[Proof of Theorem \ref{existenceKX}]
For each weight $\mathbf{a} \in \mathbb{N}^n$ 
satisfying~\thetag{ \ref{rel-nef-cone}}, we first of
all express $\mathcal{ O}_{X_n }( \mathbf{ a}) \otimes \pi^{*}_{ 0, n}
K_X^{ -\delta\vert \mathbf{ a} \vert}$ as the following difference of two
nef line bundles:
\[
\big(
\mathcal{O}_{X_n}(\mathbf{a})\otimes\pi^*_{0,n}
\mathcal{ O}_X(2\vert\mathbf{a}\vert)\big)
\otimes
\big(\pi^*_{0,n}
\mathcal{ O}_X(2\vert\mathbf{a}\vert)
\otimes
\pi^{*}_{0,n}
K_X^{\delta\vert\mathbf{a}\vert}\big)^{-1}.
\]
In order to apply the algebraic holomorphic Morse inequalities to obtain
the existence of sections for high powers, we are thus led to compute the
following intersection product:
\begin{equation}
\label{morse-intersection-delta}
\aligned
&
\big(\mathcal{ O}_{X_n}(\mathbf{a})
\otimes
\pi^*_{0,n}\mathcal{ O}_X
(2\vert\mathbf{a}\vert)\big)^{n^2}
-
\\
&
\ \ \ \ \
-
n^2\big(\mathcal{ O}_{X_n}
(\mathbf{a})
\otimes\pi^*_{0,n}
\mathcal{ O}_X
(2\vert\mathbf{a}\vert)\big)^{n^2-1}
\cdot\big(\pi^*_{0,n}
\mathcal{O}_X(2\vert\mathbf{a}\vert)
\otimes\pi^{*}_{0,n}
K_X^{\delta\vert\mathbf{a}\vert}\big),
\endaligned
\end{equation}
and to decide whether it is positive. After reducing it in terms of the
Chern classes of $X$, and then in terms of $d = \deg X$ using~\thetag{
\ref{c-d}}, this intersection product becomes a
polynomial\,\,---\,\,difficult to compute explicitly, but effective aspects
will start in Section~4\,\,---\,\,in $d$ of degree less than or equal to $n
+ 1$, having coefficients which are polynomials in $(\mathbf{ a}, \delta)$
of bidegree $(n^2, 1)$, homogeneous in $\mathbf{ a}$ or identically
zero. Notice that for $\delta = 0$, the intersection product identifies
with~\thetag{ \ref{morseintersection}}; we claim that there exists a weight
$\mathbf a'$ such that~\thetag{ \ref{morseintersection}} is positive. Thus
by continuity, with the same choice of weight, for all $\delta > 0$ small
enough, the leading coefficient still remains positive. So the polynomial
in question again takes only positive values when $d \geqslant d_n$, for
some (noneffective) $d_n$. Holomorphic Morse inequalities then insure the
existence of nonzero sections.

Coming back to our claim, the argument is as follow. First of all, the three
intersection products:~\thetag{ \ref{morseintersection}}, $\mathcal
O_{X_n}(\mathbf a)^{n^2}$ and
$\big(\mathcal{O}_{X_n}(\mathbf{a})\otimes\pi^*_{0,n}\mathcal{
O}_X(2\vert\mathbf{a}\vert)\big)^{n^2}$, once evaluated with respect to the
degree $d$ of the hypersurface, are all polynomials in the variable $d$
with coefficients in $\mathbb Z[a_1,\dots,a_n]$ of degree at most $n+1$ and
the coefficients of $d^{n+1}$ of the three expressions are the same ({\em
cf.}  Proposition~3 in~\cite{ div2008b}). Next, by Proposition 2 in~\cite{
div2008b}, $\mathcal{O}_{X_n}(\mathbf{a})\otimes\pi^*_{0,n}\mathcal{
O}_X(2\vert\mathbf{a}\vert)$ is nef if $\mathbf a$ satisfies~\thetag{
\ref{rel-nef-cone}}; therefore the coefficient of $d^{n+1}$ of its top
self-intersection must be non-negative. Thus, by Lemma 1 in~\cite{
div2008b}, in order to find a weight $\mathbf a'$ in the cone defined
by~\thetag{ \ref{rel-nef-cone}} as in the claim, it suffices to show that
this coefficient is not an identically zero polynomial in $\mathbb
Z[a_1,\dots,a_n]$. So, we have to prove that it contains at least one
non-zero monomial: but by Lemma 3 in~\cite{ div2008b}, the coefficient of
its monomial $a_1^n\cdot a_2^n\cdots a_n^n$ is $(n^2)!/(n!)^n$ and we are
done ({\em cf.} also Subsection~4.4).
\end{proof}

\subsection{Global generation of the tangent bundle
to the variety of vertical jets} We now briefly present the second
ingredient, as said in the Introduction. Let $\mathcal{ X} \subset
\mathbb{ P}^{n+1} \times \mathbb{ P}^{ N_d^n}$ be the universal family
of projective $n$-dimensional hypersurfaces of degree $d$ in $\mathbb{
P}^{ n+1}$; its parameter space is the projectivization $\mathbb{ P}
\big( H^0( \mathbb{ P}^{ n+1}, \mathcal{ O}(d))\big) = \mathbb{ P}^{
N_d^n}$, where $N_d^n = \binom{ n+d+1}{ d}-1$. We have two canonical
projections:
\[
\xymatrix{
& 
\mathcal{ X} \ar[dl]_{\text{pr}_1} \ar[dr]^{\text{pr}_2} 
& 
\\
\mathbb{P}^{n+1} 
& & 
\mathbb{P}^{N_d^n}.}
\]
Consider the relative tangent bundle $\mathcal{ V} \subset T_{\mathcal
X}$ with respect to the second projection $\mathcal{ V} := \ker(
\text{ pr}_2)_*$, and form the corresponding directed manifold
$(\mathcal{ X}, \mathcal{ V})$. It is clear that $\mathcal{ V}$ is
integrable and that any entire holomorphic curve from $\mathbb{ C}$ to
$\mathcal{ X}$ tangent to $\mathcal{ V}$ has its image entirely
contained in some fiber $\text{ pr}_2^{-1}(s) = X_s$, $s\in \mathbb{
P}^{ N_d^n}$.

Now, let $p \colon J_n\mathcal{ V} \to \mathcal{ X}$ be the bundle of
$n$-jets of germs of holomorphic curves in $\mathcal{ X}$ tangent to
$\mathcal{ V}$, the so-called {\sl vertical jets}, and consider the
subbundle $J_n^{ \text{\rm reg }} \mathcal{ V}$ of {\sl regular
$n$-jets} of maps $f \colon( \mathbb{ C},0)\to \mathcal{ X}$ tangent
to $\mathcal{ V}$ such that $f'(0)\neq 0$.

\begin{theorem}[\cite{mer2008a}]\label{globgen}
The twisted tangent bundle to vertical $n$-jets:
\[
T_{J_n\mathcal{V}}
\otimes p^*\text{\rm pr}_1^*\,
\mathcal{O}_{\mathbb{P}^{n+1}}
(n^2+2n)\otimes p^*
\text{\rm pr}_2^*\,
\mathcal{O}_{\mathbb{P}^{N_d^n}}(1)
\]
is generated over $J_n^{ \text{\rm reg }} \mathcal{ V}$ by its global
holomorphic sections. Moreover, one may choose such global generating
vector fields to be invariant with respect to the reparametrization
action of $\mathbb{ G}_n$ on $J_n \mathcal{ V}$.
\end{theorem}

This means that we have enough independent, global, invariant vector
fields having {\em meromorphic} coefficients over $J_n \mathcal{ V}$
in order to linearly generate the tangent space $T_{ J_n \mathcal{ V},
j^n}$ at every arbitrary fixed regular jet $j^n \in J_n^{ \text{\rm
reg }} \mathcal{ V}$. The poles of these vector fields occur only in
the base variables of $\mathcal{ X}$, but not in the vertical jet
variables of positive differentiation order. {\em Most importantly},
the maximal pole order here is $\leqslant n^2 + 2n$, hence it is
compensated by the first twisting $( \bullet) \otimes p^* \text{\rm
pr}_1^* \, \mathcal{ O}_{ \mathbb{ P }^{n+1 }} ( n^2+2n)$.

\section{Algebraic degeneracy of entire curves}
\label{Section-3}

Now, we are fully in position to establish the {\em noneffective}
version of Theorem~\ref{main}. The proof ({\em cf.} the Introduction)
incorporates two main ingredients: 1) the existence, already
established by Theorem~\ref{existenceKX}, of at least {\em
one} nonzero global invariant jet differential vanishing on an ample
divisor; 2) Theorem~\ref{globgen} just above to produce sufficiently
many {\em new algebraically independent} jet differentials.

\begin{theorem}
\label{noneffective-main}
Let $X \subset \mathbb{ P}^{ n+1 }$ be a smooth projective
hypersurface of arbitrary dimension $n \geqslant 2$. Then there
exists a positive integer $d_n$ such that whenever $\deg X \geqslant
d_n$ and $X$ is generic, there exists a \emph{proper} algebraic
subvariety $Y \subsetneqq X$ such that every nonconstant entire
holomorphic curve $f \colon \mathbb{ C} \to X$ has image $f (\mathbb{
C})$ contained in $Y$.
\end{theorem}

\proof
As above, consider the universal projective hypersurface $\mathbb{
P}^{ n+1} \overset{ \text{ \rm pr}_1}{ \longleftarrow } \mathcal{ X}
\overset{ \text{\rm pr}_2}{ \longrightarrow } \mathbb{ P}^{ N_d^n}$ of
degree $d$ in $\mathbb{ P }^{ n+1}$. Observe that
$X_s = \text{\rm
pr}_2^{-1} (s)$ is a smooth projective hypersurface of $\mathbb{ P}^{
n+1}$ for generic $s \in \mathbb{ P}^{ N_d^n}$ and that $\mathcal{ V} =
\ker( \text{\rm pr}_2)_*$ restricted to $X_s$ coincides with the
tangent bundle to $X_s$. We infer therefore that:
\[
\small
\aligned
H^0
\big(X_s,\,
E_{n,m}\mathcal{V}^*
\otimes
\text{pr}_1^*
\mathcal{O}_{\mathbb{P}^{n+1}}
\big(-\delta m(d-n-2)\big)
\big\vert_{X_s}
\big) 
\simeq 
H^0
\big(
X_s,\,
E_{n,m}T^*_{X_s}
\otimes
K_{X_{s}}^{-\delta m}
\big).
\endaligned
\] 
Thanks to Theorem~\ref{existenceKX}, the latter space of sections is
nonzero, for small rational $\delta >0$, for $d \geqslant d_{ n,
\delta}$ and for $m$ large enough,
independently of $s$. Fix any $s_0 \in \mathbb{ P}^{ N_d^n}$ and pick
a nonzero jet differential $P_0 \in H^0 \big( X_{ s_0},\,
E_{n,m}T^*_{X_{ s_0}} \otimes K_{X_{s_0}}^{-\delta m} \big)$. In
order to employ the vector fields of Theorem~\ref{globgen}, we must at
first extend $P_0$ as a {\em holomorphic family} of nonzero jet
differentials. Thus, we invoke the following classical extension
result.

\begin{theorem}[\cite{har1977}, p.~288]
Let $\tau \colon \mathcal{ Y} \to S$ be a flat holomorphic family
of compact complex spaces and let $\mathcal{ L} \to \mathcal{ Y}$ be a
holomorphic vector bundle. Then there exists a proper subvariety $Z
\subset S$ such that for each $s_0 \in S \setminus Z$, the restriction
map $H^0 \big( \tau^{-1} (U_{ s_0}), \mathcal{ L} \big) \to H^0 \big(
\tau^{ -1} (s_0), \mathcal{ L} \vert_{ \tau^{-1} (s_0)} \big)$ is onto,
for some Zariski-dense open set $U_{ s_0} \subset S$
containing $s_0$.
\end{theorem} 

We remark that this theorem implies that the weighted degree of the jet
differential constructed above may be chosen to be independent of the
hypersurface $X_s$ of degree $d$. Now, we apply this statement $\tau =
\text{\rm pr}_2$, to $\mathcal{ Y} = \mathcal{ X}$, to $S = \mathbb{ P}^{
N_d^n}$, to $\mathcal{ L} = E_{ n, m} \mathcal{ V}^* \otimes \text{\rm
pr}_1^* \mathcal{ O}_{ \mathbb{ P}^{ n+1}} \big( - \delta m ( d - n - 2)
\big)$ and we similarly denote by $Z \subset \mathbb{ P}^{ N_d^n}$ the
embarrassing proper algebraic subvariety.  The genericity of $X$ assumed in
the two theorems~\ref{main} and~\ref{noneffective-main} will just consist
in requiring that $s_0 \not \in Z$ (notice {\em passim}\, that we do not
have a constructive access to $Z$) and of course also, that $s$ does not
belong to the set for which $X_s$ is singular.

We therefore obtain a holomorphic family of jet differentials:
\[
P
=
\big\{
P\vert_s\in H^0\big(
X_s,\,
E_{n,m}T^*_{X_s}
\otimes 
K_{X_s}^{-\delta m}
\big)
\big\}
\]
parametrized by $s$ with $P \vert_{ s_0} = P_0 \not\equiv 0$ and
vanishing on $K_{ X_s}^{ \delta m}$; for our purposes, it will suffice
that $s$ varies in some neighborhood of $s_0$.

Now, take a {\em nonconstant} entire holomorphic curve $f \colon
\mathbb{ C} \to \mathcal{ X}$ tangent to $\mathcal{ V}$. Since the
distribution $\mathcal{ V}$ has integral manifolds $\text{\rm pr}_2^{
-1} ( s) = X_s$, $f$ maps $\mathbb{ C}$ into some $X_{s_0}$, for some
$s_0 \in \mathbb{ P}^{ N_d^n}$. Of course, we assume 
that $s_0 \not \in
Z$ and that $X_{ s_0}$ is non-singular. 
Consider now the zero-set locus
\[
Y_{s_0} 
:= 
\big\{x\in X_{s_0}
\colon
P\vert_{s_0}(x)
=
0\big\}, 
\]
where $P\vert_{ s_0} \not\equiv 0$ vanishes as a section of the vector
bundle $E_{ n, m} T^*_{ X_{ s_0 }} \otimes K_{ X_{ s_0 }}^{ - \delta
m}$. Then $Y_{ s_0}$ is a {\em proper algebraic subvariety} of $X_{
s_0}$. We then claim that
\[
f(\mathbb{C})\subset Y_{s_0},
\]
which will complete the proof of the theorem. (It will even come out
that we obtain strong algebraic degeneracy of entire curves $f :
\mathbb{ C} \to X_s$ inside a $Y_s \subsetneqq X_s$ defined by $Y_s =
\big\{ x \in X : \, \, P\vert_s ( x) = 0 \big\}$ and parametrized by
$s$ near $s_0$.)

Reasoning by contradiction, suppose that there exists $t_0 \in
\mathbb{ C}$ with $f ( t_0) \not \in Y_{ s_0}$. Consider the $n$-jet
map $j^n \! f \colon \mathbb{ C}\to J_n \mathcal{ V}$ induced by
$f$. If $j^n \! f ( \mathbb{ C })$ would be entirely contained in
$J_n^{ \text{ \rm sing }} \mathcal{ V} \overset{ \text{\rm def}} = J_n
\mathcal{ V} \setminus J_n^{\text{\rm reg }} \mathcal{ V}$, then $f$
would be {\em constant}, since singular $n$-jets satisfy $f' (t) = 0$.
So necessarily $j^n \! f (\mathbb{C}) \not \subset J_n \mathcal{ V}^{
\text{\rm sing}}$, namely $f' \not \equiv 0$. Then by shifting a bit
$t_0$ if necessary, we can assume that we in addition have $f' ( t_0)
\neq 0$, {\em viz.} $j^n \! f ( t_0) \in J_n^{\text{\rm reg}}
\mathcal{ V}$.

Theorem~\ref{A} ensures that $P \vert_{ s_0} \big( j^n \, f ( t) \big)
\equiv 0$. Denote $U := \mathbb{ P}^{ N_d^n} \setminus Z$.

We may now view the family $P = \{ P \vert_s \}$ as being a
holomorphic map
\[
P\colon J_n\mathcal{ V}\big\vert_{\text{pr}_2^{-1}(U)}
\longrightarrow
p^*\text{\rm pr}_1^*
\mathcal{O}_{\mathbb{P}^{n+1}}
\big(-\delta m(d-n-2)\big)
\big\vert_{\text{\rm pr}_2^{-1}(U)}
\]
which is polynomial of weighted degree $m$ in the jet variables. Let
$V$ be any of the global invariant holomorphic vector fields on $J_n
\mathcal{ V}$ with values in $p^*\text{\rm pr}_1^*\mathcal{
O}_{\mathbb{ P}^{ n+1}} (n^2+2n)$ that were provided by Theorem
\ref{globgen}. Then we observe that the Lie derivative $L_V P$
together with the natural duality pairing
\[
\mathcal{O}_{\mathbb{P}^{n+1}}(p)
\times
\mathcal{O}_{\mathbb{P}^{n+1}}(-q)
\to
\mathcal{O}_{\mathbb{P}^{n+1}}(p-q)
\ \ \ \ \ \ \ \ \ \ \ \ \
{\scriptstyle{(p,\,q\,\geqslant\,1)}}
\] 
provides a new holomorphic map (notice the shift by $n^2 + 2n$):
\[
L_VP
\colon 
J_n\mathcal{V}
\big\vert_{\text{pr}_2^{-1}(U)}
\longrightarrow
p^*\text{\rm pr}_1^*
\mathcal{ O}_{\mathbb{P}^{n+1}}
\big(-\delta m(d-n-2)+n^2+2n\big)
\big\vert_{\text{\rm pr}_2^{-1}(U)},
\]
again polynomial of weighted degree $m$ in the jet variables, thus a
new parameterized family of invariant jet differentials. In
particular, the restriction $L_VP \vert_{ s_0}$ of $L_VP$ to $\{ s =
s_0\}$ yields a {\em nonzero} global holomorphic section in
\begin{multline*}
H^0\big(
X_{s_0},\,
E_{n,m}T^*_{X_{s_0}}
\otimes 
K_{X_{s_0}}^{-\delta m}
\otimes
\mathcal{O}_{X_{s_0}}(n^2+2n)
\big)
=
\\
=
H^0\big(X_{s_0},\,
E_{n,m}T^*_{X_{s_0}}
\otimes
\mathcal{ O}_{X_{s_0}}
(-\delta m(d-n-2)+n^2+2n)\big),
\end{multline*}
which is a global invariant jet differential on $X_{ s_0}$ vanishing
on an ample divisor provided that $-\delta m (d-n-2) + n^2+2n$ {\em
still remains negative}; therefore, if we ensure such a negativity
({\em see} below), Theorem~\ref{A} shows that $[L_VP \vert_{s_0}]
\big( j^n\! f (t) \big) \equiv 0$. As a result, the $n$-jet of $f$ now
satisfies {\em two} global algebraic differential equations:
\[
P_{s_0}
\big(j^n\! f(t)\big)
\equiv
\big[L_VP\vert_{s_0}\big]
\big(j^n\! f(t)\big)
\equiv
0.
\]

\begin{center}
\begin{picture}(0,0)%
\includegraphics{further-jet-differential.pstex}%
\end{picture}%
\setlength{\unitlength}{4144sp}%
\begingroup\makeatletter\ifx\SetFigFont\undefined%
\gdef\SetFigFont#1#2#3#4#5{%
  \reset@font\fontsize{#1}{#2pt}%
  \fontfamily{#3}\fontseries{#4}\fontshape{#5}%
  \selectfont}%
\fi\endgroup%
\begin{picture}(5501,2235)(397,-1580)
\put(3698,-175){\makebox(0,0)[lb]{\smash{{\SetFigFont{7}{8.4}{\familydefault}{\mddefault}{\updefault}{\color[rgb]{.69,0,0}\red{$V$}}%
}}}}
\put(1379,-494){\makebox(0,0)[lb]{\smash{{\SetFigFont{8}{9.6}{\familydefault}{\mddefault}{\updefault}{\color[rgb]{0,.69,0}\green{$f(\mathbb{C})$}}%
}}}}
\put(950,-853){\makebox(0,0)[lb]{\smash{{\SetFigFont{6}{7.2}{\familydefault}{\mddefault}{\updefault}{\color[rgb]{0,.69,0}\green{$f(t_0)$}}%
}}}}
\put(2471,-546){\makebox(0,0)[lb]{\smash{{\SetFigFont{9}{10.8}{\familydefault}{\mddefault}{\updefault}{\color[rgb]{0,0,.69}\blue{fiber}}%
}}}}
\put(3269,154){\makebox(0,0)[lb]{\smash{{\SetFigFont{7}{8.4}{\familydefault}{\mddefault}{\updefault}{\color[rgb]{0,0,.69}\blue{$J_n\!\mathcal{V}_{\!\!f(t_0)}$}}%
}}}}
\put(2439,-746){\makebox(0,0)[lb]{\smash{{\SetFigFont{8}{9.6}{\familydefault}{\mddefault}{\updefault}{\color[rgb]{0,0,.69}\blue{$J_n\!\mathcal{V}_{\!\!f(t_0)}$}}%
}}}}
\put(907,-160){\makebox(0,0)[lb]{\smash{{\SetFigFont{9}{10.8}{\familydefault}{\mddefault}{\updefault}{\color[rgb]{0,0,.69}\blue{$X$}}%
}}}}
\put(5388,-303){\makebox(0,0)[lb]{\smash{{\SetFigFont{8}{9.6}{\familydefault}{\mddefault}{\updefault}{\color[rgb]{.69,0,0}{\bf ?}}%
}}}}
\put(1846,-1456){\makebox(0,0)[lb]{\smash{{\SetFigFont{12}{14.4}{\familydefault}{\mddefault}{\updefault}{\color[rgb]{0,0,0} }%
}}}}
\put(4249,-383){\makebox(0,0)[lb]{\smash{{\SetFigFont{9}{10.8}{\familydefault}{\mddefault}{\updefault}{\color[rgb]{0,0,.69}\blue{differential}}%
}}}}
\put(4248,-550){\makebox(0,0)[lb]{\smash{{\SetFigFont{9}{10.8}{\familydefault}{\mddefault}{\updefault}{\color[rgb]{0,0,.69}\blue{$L_VP$}}%
}}}}
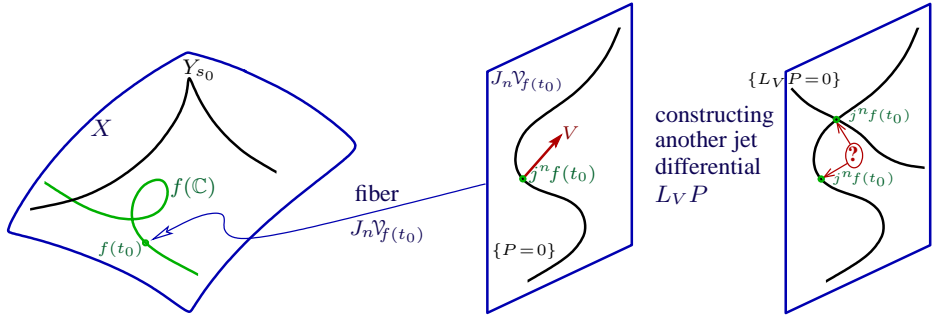
\put(680,-1516){\makebox(0,0)[lb]{\smash{{\SetFigFont{9}{10.8}{\familydefault}{\mddefault}{\updefault}{\color[rgb]{0,0,0}{\bf Fig.~1: Producing from $P$ a new jet differential $L_VP$ having distinct zero locus in $J_n \mathcal{ V}$}}%
}}}}
\put(4249,-223){\makebox(0,0)[lb]{\smash{{\SetFigFont{9}{10.8}{\familydefault}{\mddefault}{\updefault}{\color[rgb]{0,0,.69}\blue{another jet}}%
}}}}
\put(3282,-864){\makebox(0,0)[lb]{\smash{{\SetFigFont{6}{7.2}{\familydefault}{\mddefault}{\updefault}{\color[rgb]{0,0,0}$\{P\!=\!0\}$}%
}}}}
\put(4794,155){\makebox(0,0)[lb]{\smash{{\SetFigFont{6}{7.2}{\familydefault}{\mddefault}{\updefault}{\color[rgb]{0,0,0}$\{L_VP\!=\!0\}$}%
}}}}
\put(4250,-76){\makebox(0,0)[lb]{\smash{{\SetFigFont{9}{10.8}{\familydefault}{\mddefault}{\updefault}{\color[rgb]{0,0,.69}\blue{constructing}}%
}}}}
\put(1450,214){\makebox(0,0)[lb]{\smash{{\SetFigFont{8}{9.6}{\familydefault}{\mddefault}{\updefault}{\color[rgb]{0,0,0}$Y_{s_0}$}%
}}}}
\put(3505,-411){\makebox(0,0)[lb]{\smash{{\SetFigFont{7}{8.4}{\familydefault}{\mddefault}{\updefault}{\color[rgb]{0,.69,0}\green{$j^n\!f(t_0)$}}%
}}}}
\put(5396,-56){\makebox(0,0)[lb]{\smash{{\SetFigFont{5}{6.0}{\familydefault}{\mddefault}{\updefault}{\color[rgb]{0,.69,0}\green{$j^n\!f(t_0)$}}%
}}}}
\put(5300,-434){\makebox(0,0)[lb]{\smash{{\SetFigFont{5}{6.0}{\familydefault}{\mddefault}{\updefault}{\color[rgb]{0,.69,0}\green{$j^n\!f(t_0)$}}%
}}}}
\end{picture}%

\end{center}

Heuristically ({\em cf.} the figure), if the fiber $J_n \mathcal{
V}_{\! f ( t_0)}$ would be, say, 2-dimensional, and if the
intersection of $\{ P_{s_0} = 0 \}$ with $\{ L_V P \vert_{ s_0} =
0\}$, viewed in the fiber $J_n \mathcal{ V}_{\! f ( t_0)}$, would be a
point {\em distinct from the original $j^n \! f ( t_0)$}, we would get
the sought contradiction. Now we realize this idea ({\em cf.}~\cite{
siu2004, pau2008, rou2007}) by producing enough new jet differential
divisors whose intersection becomes {\em empty}.

Indeed, with $t_0$ such that $f ( t_0) \not \in Y_{ s_0}$ and $j^n \!
f (t_0) \in J_n^{\text{\rm reg}} \mathcal{ V}$, and with $W_i$, $V_j$
denoting some global meromorphic vector fields in
\[
H^0
\big(
J_n\mathcal{ V},\,
T_{J_n\mathcal{V}}
\otimes 
p^*\text{\rm pr}_1^*
\mathcal{O}_{\mathbb{P}^{n+1}}(n^2+2n)
\otimes 
p^*\text{\rm pr}_2^*\mathcal{O}_{\mathbb{P}^{N_d^n}}(1)
\big),
\]
that are supplied by Theorem~\ref{globgen}, we claim that the
following two {\em evidently contradictory} conditions can be
satisfied, and this will achieve the proof.

\begin{itemize}

\smallskip\item[{\bf (i)}]
For every $p \leqslant m$ and for arbitrary such fields $W_1, \dots,
W_p$, the restriction $L_{ W_p} \cdots L_{ W_1} P \big\vert_{
s_0}$ yields a nonzero global holomorphic section in
\[
H^0
\big(
X_{s_0},\,
E_{n,m}T_{X_{s_0}}^*\otimes
\mathcal{O}_{X_{s_0}}
(-\delta m(d-n-2)+p(n^2+2n))
\big)
\]
with the property that $\big[ L_{ W_p} \cdots L_{ W_1} P \big] \big(
s_0, \, j^n \! f (t) \big) \equiv 0$.

\smallskip\item[{\bf (ii)}]
there exist some $p \leqslant m$ and some invariant fields $V_1,
\dots, V_p$ such that $\big[ L_{V_p} \cdots L_{V_1} P
\big] \big( s_0, \, j^n\! f(t_0) \big) \neq 0$.

\end{itemize}\smallskip

\noindent

The first condition {\bf (i)} will automatically be ensured by
Theorem~\ref{A} provided the resulting jet differential still vanishes
on an ample divisor, \emph{i.e.} provided that
\[
-\delta m(d-n-2)+p(n^2+2n)
<
0
\]
is still negative. But since $p$ will be $\leqslant m$, it suffices
that $- \delta m ( d - n - 2) + m ( n^2 + 2n) < 0$, and then after
erasing $m$, that:
\begin{equation}
\label{C2}
d
>
{\textstyle{\frac{n^2+2n}{\delta}}}
+
n+2.
\end{equation}
To get {\bf (i)}, we first fix a rational $\delta >0$ so that
Theorem~\ref{existenceKX} gives a {\em nonzero} jet differential for
any $d \geqslant d_{n, \delta}$, we increase (if necessary) this
lower bound by taking account of~\thetag{ \ref{C2}}, we construct
the holomorphic family $P\vert_s$, and {\bf (i)} holds.

To establish {\bf (ii)}, we choose local coordinates:
\[
\big(s,z,z',\dots,z^{(n)}
\big)
\in
\mathbb{C}^{N_d^n}
\times\mathbb{C}^n\times\mathbb{C}^n
\times\cdots\times\mathbb{C}^n
\]
on $J_n \mathcal{ V}$ near $\big( s_0, j^n \! f ( t_0) \big)$, where
$z \in \mathbb{ C}^n$ provides some local coordinates on $X_s$ for any
fixed $s$ near $s_0$, and where $\big( z', \dots, z^{ (n)} \big)$ are
the jet coordinates associated with $z$. We also choose a local
trivialization of the line bundle $K_{ X_s}^{ -
\delta m}$. Then our holomorphic family of jet differentials
$P\vert_s \in H^0 \big( X_s, \, E_{ n, m} T_{ X_s}^* \otimes K_{
X_s}^{ - \delta m} \big)$ writes locally as a weighted $m$-homogeneous
jet-polynomial:
\[
P
=
\sum_{\vert i_1\vert+\cdots+n\vert i_n\vert=m}\,
q_{i_1,\dots,i_n}(s,z)\,
(z')^{i_1}\cdots (z^{(n)})^{i_n},
\]
where $i_1, \dots, i_n \in \mathbb{ N}^n$ and where the $q_{ i_1,
\dots, i_n} ( s, z)$ are holomorphic near $( s_0, f( t_0) )$. Locally,
the proper subvariety $Y_{ s_0} \subset
X$ is represented as the common zero-locus:
\[
Y_{s_0}
=
\big\{
z\in X_{s_0}\colon
q_{i_1,\dots,i_n}(s_0,z)
=
0,\ \
\forall\,\,i_1,\dots,i_n
\big\}.
\]
By our assumption that $f ( t_0) \not \in Y_{ s_0}$, there exist
$i_1^0, \dots, i_n^0 \in \mathbb{ N}^n$ such that $q_{ i_1^0, \dots,
i_n^0} \big( s_0, f( t_0) \big) \neq 0$. If we make the translational
change of jet coordinates $\overline{ z}' := z' - f'(t_0)$, \dots,
$\overline{ z}^{ ( n)} := z^{ ( n)} - f^{ ( n)} ( t_0)$, our
jet-polynomial transfers to:
\[
\overline{P}
=
\sum_{\vert i_1\vert+\cdots+n\vert i_n\vert\leqslant m}\,
\overline{q}_{i_1,\dots,i_n}(s,z)\,
(\overline{z}')^{i_1}\cdots(\overline{z}^{(n)})^{i_n},
\]
(notice ``$\leqslant \! m$'') with new coefficients $\overline{ q}_{
i_1, \dots, i_n} (s, z)$ that depend linearly upon the old ones and
polynomially upon $\big( f'(t_0), \dots, f^{ (n)} ( t_0) \big)$.
Again, there exist $\overline{ i}_1^0, \dots, \overline{ i}_n^0 \in
\mathbb{ N}^n$ such that $\overline{ q}_{ \overline{ i}_1^0, \dots,
\overline{ i}_n^0} \big( s_0, f ( t_0) \big)\neq 0$, because otherwise
the two jet-polynomials $P\big\vert_{ s_0, f (t _0)}$ and $\overline{
P}\big\vert_{ s_0, f ( t_0)}$ would be both identically zero.

Since $j^n \! f ( t_0) \in J_n^{\text{\rm reg}} \mathcal{ V}$, by the
property~\ref{globgen} of generation by global sections, we get that
for every $k$ with $1 \leqslant k \leqslant n$ and for every $i$ with
$1 \leqslant i \leqslant n$, there exists an invariant vector field
$V_i^k$ with
\[
V_i^k\big\vert_{(s_0,\overline{j}^n\!f(t_0))} 
= 
{\textstyle{\frac{\partial}{
\partial\overline{z}_i^{(k)}}}}
\Big\vert_{ ( s_0, \overline{
j}^n\! f ( t_0 ))}, 
\]
where we have denoted the translated central jet by $\overline{ j}^n
\! f ( t_0) := \big( f (t_0), 0, \dots, 0 \big)$.

To achieve the proof of {\bf (ii)}, we may suppose that for every
integer $p$ with $p < \vert \overline{ i}_1^0 \vert + \vert \overline{
i}_2^0 \vert + \cdots + \vert \overline{ i}_n^0 \vert$, whence $p <
\vert \overline{ i}_1^0 \vert + 2 \, \vert \overline{ i}_2^0 \vert +
\cdots + n\, \vert \overline{ i}_n^0 \vert = m$, and for every $p$
invariant vector fields $W_1, \dots, W_p$, one has $\big[ W_1 \cdots
W_p \, \overline{ P} \, \big] \big(s_0, \overline{j}^n\! f ( t_0)\big)
= 0$, since if any such an expression is already $\neq 0$, {\bf (ii)}
would be got gratuitously. Thanks to the global generation
Theorem~2.5, this vanishing property then holds for any vector fields
$W_i$ involving all the possible differentiations $\frac{ \partial}{
\partial s}$, $\frac{ \partial }{ \partial z}$, $\frac{ \partial }{
\partial \overline{ z}'}$, \dots, $\frac{ \partial }{ \partial
\overline{ z}^{ (n)}}$. Then under this assumption, the contribution
of the remainder differentiations present in $V_i^k$ after $\partial
\big/ \partial\overline{z}_i^{(k)} \big\vert_{ ( s_0, \overline{
j}^n\! f ( t_0 ))}$ will vanish at the point $\big( s_0, \overline{
j}^n \! f ( t_0) \big)$ when performing any multi-derivation of length
equal to $\vert \overline{ i}_1^0 \vert + \cdots + \vert \overline{
i}_n^0 \vert$, hence if we write in 
length $\overline{ i}_k^0 = \big( \overline{
i}_{ k, 1}^0, \dots, \overline{ i}_{ k, n}^0 \big) \in \mathbb{ N}^n$
all the multiindices present in the specific coefficient $\overline{
q}_{ \overline{i }_1^0, \dots, \overline{ i}_n^0}$, it follows that:
\[
\aligned
\big[
V_{\overline{i}_{n,n}^0}^n\!\!\cdots\,
&
V_{\overline{i}_{n,1}^0}^n\cdots\cdots\,
V_{\overline{i}_{1,n}^0}^1\!\!\cdots\,
V_{\overline{i}_{1,n}^0}^1\!\!
\overline{P}\,
\big]
\big(s_0,\,\overline{j}^n\!f(t_0)\big)
=
\\
&
=
\Big[
{\textstyle{\frac{\partial}{
\partial\overline{z}_{\overline{i}_{n,n}^0}^{(n)}}}}
\cdots
{\textstyle{\frac{\partial}{
\partial\overline{z}_{\overline{i}_{n,1}^0}^{(n)}}}}
\cdots\cdots
{\textstyle{\frac{\partial}{
\partial\overline{z}_{\overline{i}_{1,n}^0}^{(1)}}}}
\cdots
{\textstyle{\frac{\partial}{
\partial\overline{z}_{\overline{i}_{1,1}^0}^{(1)}}}}
\overline{P}\,
\Big]
\big(s_0,\,f(t_0),0,\dots,0\big)
\\
&
=
\overline{i}_{n,n}^0!\cdots
\overline{i}_{n,1}^0!\cdots\cdots
\overline{i}_{1,n}^0!\cdots
\overline{i}_{1,1}^0!\,\,
\overline{q}_{\overline{i}_1^0,\dots,\overline{i}_n^0}
\big(
s_0,\,f(t_0)
\big)
\neq 0,
\endaligned
\]
which is nonzero. Thus {\bf (ii)} holds and the proof of
Theorem~\ref{noneffective-main} is complete.
Theorem~\ref{noneffective-main} being {\em not} effective regarding
the condition $d \geqslant d_n$, the next two 
Sections~\ref{Section-4} and~\ref{Section-5} are
devoted to the proof of the effective main Theorem~\ref{main}.
\endproof

\section{Effectiveness of the degree lower bound}
\label{Section-4}

It is known ({\em cf.}~\cite{ siu1995, dem1997, voi2003, siu2004,
rou2006a, div2008a, mer2008b}) that reaching an explicit lower bound
degree $\deg X \geqslant d_n$ both for Green-Griffiths algebraic
degeneracy and for Kobayashi hyperbolicity (in nonoptimal degree)
still remained an open question in arbitrary dimension $n$, due to the
existence of {\em substantial algebraic obstacles}. In order to render
somewhat explicit the lower bound $d_n$ of
Theorem~\ref{noneffective-main}, one has to expand the $n^2$-powered
intersection product~\thetag{ \ref{morse-intersection-delta}} and then
to reduce it as an explicit polynomial ${\sf P}_{ \mathbf{ a}, \delta}
(d)$, as was foreseen in the proof of Theorem~\ref{existenceKX}. To
this aim, one should descend Demailly's tower {\em step by step}, each
time using the two relations~\thetag{ \ref{c-l-k}} and~\thetag{
\ref{u-n}}. As a matter of fact, one must perform some numerous,
explicit eliminations and substitutions and thereby tame the
exponential growth of computations. At several places, we shall leave
aside optimality of majorations in order to reach the neat announced
lower bound $2^{ n^5}$.

\subsection{Reduction of the basic intersection product} 
We remind from Theorem~\ref{existenceKX} that, in order to produce a
global invariant jet differential with controlled vanishing order on
hypersurfaces $X$ whose degree $d \geqslant d_n$ would be bounded from
below by an effectively known function $d_n = d ( n)$ of $n$, we
should ensure {\em in an effective way} the positivity of the
intersection product:
\begin{multline*}
\big(\mathcal{ O}_{X_n}(\mathbf{a})
\otimes
\pi^*_{0,n}\mathcal{ O}_X(2\vert\mathbf{a}\vert)\big)^{n^2}
-
\\
-n^2\big(\mathcal{ O}_{X_n}(\mathbf{a})
\otimes
\pi^*_{0,n}\mathcal{ O}_X(2\vert\mathbf{a}\vert)\big)^{n^2-1}
\cdot
\big(\pi^*_{0,n}\mathcal{ O}_X(2\vert\mathbf{a}\vert)
\otimes
\pi^{*}_{0,n}K_X^{\delta\vert\mathbf{a}\vert}\big),
\end{multline*}
for a certain $n$-tuple of integers $\mathbf{ a} = \mathbf{ a} ( n)
\in \mathbb{ N}^n$ belonging to the cone~\thetag{ \ref{rel-nef-cone}}
(with $k = n$)
which would depend {\em effectively} upon $n$, and for a certain
rational number $\delta = \delta ( n) > 0$ which would also depend
{\em effectively} upon $n$.

As in~\cite{div2008b}, denote $u_\ell = c_1 \big( \mathcal{ O}_{
X_\ell} (1) \big)$ for $\ell = 1, \dots, n$, denote $c_k = c_k(T_X)$
for $k = 1, \dots, n$, and $h = c_1 \big( \mathcal{ O}_X(1) \big)$.
With these standard notations, the intersection product we have to
evaluate becomes:
\begin{multline}
\label{pi-delta}
\Pi_\delta
:=
\big(
a_1u_1+\dots+a_nu_n+2\vert\mathbf{a}\vert h
\big)^{n^2}
-
\\
-n^2\big(
a_1u_1+\dots+a_nu_n+2\vert
\mathbf{a}\vert h
\big)^{n^2-1}
\cdot
\big(
2\vert\mathbf{ a}\vert h-\delta\vert\mathbf{a}\vert c_1
\big);
\end{multline}
here and from now on, admitting a slight abuse of notation which will
greatly facilitate the reading of formal computations, {\em we
systematically omit every pull-back symbol $\pi_{ j,k}^* ( \bullet
)$}. After elimination and reduction using the
relations~\thetag{\ref{c-l-k}} and~\thetag{\ref{u-n}} ({\em see}
below), our intersection product gives in principle a polynomial
(difficult to compute, 
{\em see} the end of the paper) 
of degree $\leqslant n+1$ with respect to $d =
\deg X$, which is affine in $\delta$, and all of which coefficients
are homogeneous polynomials in $\mathbf{ a}$ of degree $n^2$. Thus,
let us call it:
\[
{\sf P}_{\mathbf{ a},\delta}(d)
=
{\sf P}_{\mathbf{ a}}(d)+\delta\,{\sf P}'_{\mathbf{ a}}(d)
=
\sum_{k=0}^{n+1}\,{\sf p}_{k,\mathbf{a}}\,d^k
+
\delta\,
\sum_{k=0}^{n+1}\,{\sf p}_{k,\mathbf{a}}'\,d^k.
\] 
Now, suppose in advance that we have an effective control, through
explicit inequalities, of all the coefficients ${\sf p}_{ k, \mathbf{
a}} \in \mathbb{ Z}$ and ${\sf p}_{ k, \mathbf{ a}}' \in \mathbb{ Z}$
of both ${\sf P}_{ \mathbf{ a}}$ and ${\sf P}_{ \mathbf{ a}}'$, and
more precisely, that we already know inequalities of the type: 
\[
\vert{\sf p}_{k,\mathbf{a}}\vert 
\leqslant{\sf E}_k\ \ \
{\scriptstyle{(k\,=\,0,\,\dots,\,n)}},
\ \ \ \ \ \ \
{\sf p}_{n+1, \mathbf{ a}} 
\geqslant
{\sf G}_{n+1},
\ \ \ \ \ \ \
\vert{\sf p}_{k,\mathbf{a}}'\vert 
\leqslant 
{\sf E}_k'\ \ \
{\scriptstyle{(k\,=\,0,\,\dots,\,n,\,n\,+\,1)}}, 
\]
with the ${\sf E}_k \in \mathbb{ N}$, with ${\sf G}_{ n+1} \in
\mathbb{ N} \setminus \{ 0\}$ and with the ${\sf E}_k' \in \mathbb{
N}$ all depending upon $n$ only. According to the proof of
Theorem~\ref{existenceKX}, a good choice of weight $\mathbf{ a}$
indeed makes ${\sf p}_{ n+1, \mathbf{ a}}$ positive; we will see below
that ${\sf p}_{ n+1, \mathbf{ a}}'$ is then necessarily negative.

If we now set $\delta := \frac{ 1}{ 2}\, \frac{ {\sf G}_{n+1} }{ {\sf
E}_{ n+1}'}$ so that $\delta$ also depends {\em a posteriori}
explicitly upon $n$, the leading $d^{ n+1}$-coefficient of ${\sf P}_{
\mathbf{ a}, \delta}$ becomes positive and bounded from below:
\[
{\sf p}_{n+1,\mathbf{a}}
+
\delta\,{\sf p}_{n+1,\mathbf{a}}'
=
{\sf p}_{n+1,\mathbf{a}}
-
\delta\,
\big\vert
{\sf p}_{n+1,\mathbf{a}}'
\big\vert
\geqslant
{\sf G}_{n+1}
-
{\textstyle{\frac{1}{2}}}\,
{\textstyle{\frac{{\sf G}_{n+1}}{{\sf E}_{n+1}'}}}\,
{\sf E}_{n+1}'
=
{\textstyle{\frac{1}{2}}}\,
{\sf G}_{n+1}.
\]
The largest real root of a polynomial ${\sf a}_{ n+1} \, d^{ n+1} +
{\sf a}_n \, d^n + \cdots + {\sf a}_0$ having integer coefficients and
positive leading coefficient ${\sf a}_{ n+1} \geqslant 1$ may be
checked to be less than $1 + ({\sf a}_n + \cdots + {\sf a}_0) \big/
{\sf a}_{ n+1}$; instead of the finer bound $2 \max_{0\leqslant j
\leqslant n} \big( \frac{ \vert {\sf a}_j \vert}{ \vert {\sf a}_{ n+1}
\vert}\big)^{ 1 / n+1-j}$, we use this easier-to-write-down majoration
because at the end of Section~4, this will make no difference in
reaching the bound $\deg X \geqslant 2^{ n^5}$ of Theorem~\ref{main}.
Applied to our situation:

\begin{lemma}
\label{d-1-n}
If one chooses $\delta := \frac{ 1}{ 2}\, \frac{ {\sf G}_{n+1} }{ {\sf
E}_{ n+1}'}$, then the intersection product $\sum_{ k=0 }^{ n+1}\,
\big( {\sf p}_{ k, \mathbf{ a}} + \delta \, {\sf p}_{ k , \mathbf{
a}}' \big) \, d^k$ has positive leading coefficient ${\sf p}_{ n+1,
\mathbf{ a} } + \delta \, {\sf p}_{ n+1, \mathbf{ a}}' \geqslant
\frac{ 1}{ 2}\, {\sf G}_{ n+1}$ and has other coefficients enjoying
the majorations:
\[
\Big\vert{\sf p}_{k,\mathbf{a}} 
+ 
\delta\, 
{\sf p}_{k ,\mathbf{a}}'
\Big\vert 
\leqslant
{\sf E}_k
+
{\textstyle{\frac{1}{2}\, 
\frac{{\sf G}_{ n+1}}{{\sf E}_{n+1}'}}}\,
{\sf E}_k '
\ \ \ \ \ \ \ \ \ \ \ \ \
{\scriptstyle{(k\,=\,0,\,\dots,\,n)}},
\] 
and therefore it takes only positive values for all degrees
\[
\small
\aligned
d
\geqslant
1+
{\textstyle{
\Big(
{\sf E}_n+\cdots+{\sf E}_0
+
\frac{1}{2}\,\frac{{\sf G}_{n+1}}{{\sf E}_{n+1}'}\,
\big\{{\sf E}_n'
+\cdots+
{\sf E}_0'\big\}
\Big)
\Big/
\frac{1}{2}\,
{\sf G}_{n+1}
}}
=:
d_n^1.
\qed
\endaligned
\]
\end{lemma}

Thus, this $d_n^1$ will be effectively known in terms of $n$ when ${\sf
E}_k$, ${\sf G}_{ n+1}$, ${\sf E}_k'$ will be so. In order to have not
only the existence of global invariant jet differentials with
controlled vanishing order, but also algebraic degeneracy, we have
also to take account of condition~\thetag{\ref{C2}}, and 
this condition now reads:
\[
d
\geqslant
1
+
n+2
+
2\,(n^2+2n)\,
{\textstyle{\frac{{\sf E}_{n+1}'}{{\sf G}_{n+1}}}}
=:
d_n^2. 
\]
In conclusion, we would obtain the \emph{effective} estimate of
Theorem~\ref{main} provided we compute the bounds ${\sf E}_k$, ${\sf
G}_{ n+1}$, ${\sf E}_k'$ in terms of $n$ and provided we establish
that:
\begin{equation}
\label{finalestimate}
2^{n^5}
\geqslant
\max
\big\{
d_n^1,d_n^2
\big\} 
=:
d_n.
\end{equation}

\subsection{ Expanding the intersection product}
By expanding the $n^2$- and the $(n^2-1)$-powers, the intersection
product $\Pi_\delta$ in~\thetag{\ref{pi-delta}} writes as a certain
sum, with coefficients being polynomials in $\mathbb{ Z} \big[ a_1,
\dots, a_n, \delta \big]$, of monomials in the present Chern classes
that are of the general form:
\[
h^lu_1^{i_1}\cdots u_n^{i_n}
\ \ \ \ \ 
\text{\rm or}
\ \ \ \ \
h^lc_1u_1^{j_1}\cdots u_n^{j_n},
\]
where $l + i_1 + \cdots + i_n = n^2$ or $l + 1 + j_1 + \cdots + j_n = 
n^2$.

\begin{lemma}[\cite{dem1997,div2008a}]
\label{ineffective-reduction}
After several elimination computations which take account of the
relations~\thetag{\ref{c-l-k}} and~\thetag{\ref{u-n}}, any such
monomial reduces to a certain polynomial in $\mathbb{ Z} \big[ h,
c_1, \dots, c_n \big]$ which is homogeneous of degree $n = \dim X$, if
$h$ is assigned the weight $1$ and each $c_k$ receives the weight $k$.
Furthermore, after a last substitution by means of \thetag{\ref{c-d}}
which uses $h^n \equiv \int_X h^n = d = \deg X$, the polynomial in
question becomes a plain polynomial in $\mathbb{ Z} [ d]$ of degree
$\leqslant n+1$.
\qed
\end{lemma}

We illustrate with $h^l u_1^{i_1} \cdots u_{ n-1}^{ i_{ n-1}} u_n^{
i_n}$ three fundamental processes of reduction that will be
intensively used. Recall that any {\em sub}monomial $h^l u_1^{ i_1}
\cdots u_\ell^{ i_\ell} = \pi_{ 0, \ell}^* (h^l) \pi_{ 1, \ell}^*
(u_1^{ i_1}) \cdots u_\ell^{ i_\ell}$ denotes a differential form
living $X_\ell$ and that $\dim X_\ell = n + \ell ( n-1)$. Such a form
is of bidegree $(p, p)$ where $p = l + i_1 + \cdots + i_\ell$. We
shall allow the (slight) abuse of language to say
that $p$ itself is the {\sl
degree} of a $(p, p)$-form.

At first, if $i_n \leqslant n-2$, then $l + i_1 + \cdots + i_{ n-1}
\geqslant n^2 - n + 2 = 1 + \dim_\mathbb{ C} X_{ n-1}$, whence the
(sub)form $h^l u_1^{ i_1} \cdots u_{ n-1}^{ i_{ n-1}}$ which lives on
$X_{ n-1}$ annihilates, as then does $h^l u_1^{i_1} \cdots u_{ n-1}^{
i_{ n-1}} u_n^{ i_n}$ too. We call this (straightforward) first kind
of reduction process:
\[
{\text{\sl ``vanishing for degree-form reasons''}},
\]
and we symbolically point out the annihilating subform by underlining
it with a small circle appended, {\em viz.}:
\[
\underline{h^lu_1^{i_1}\cdots u_{n-1}^{i_{n-1}}}_\circ
u_n^{i_n}
=
0
\ \ \ \ \ \ \
\text{\rm when}\ \
i_n\leqslant n-2.
\]
This will greatly improve readability of elimination computations
below.

Secondly, in the case where $i_n = n-1$, using an appropriate version
of the Fubini theorem and taking account of the fact that $\int_{\rm
fiber} \, u_n^{ n-1} = \int_{ \mathbb{P}^{ n-1}}\, u_n^{n-1} = 1$,
where all the fibers of $\pi_{ n-1, n} : X_n \to X_{ n-1}$ are $\simeq
\mathbb{P}^{ n-1} ( \mathbb{ C})$ (\cite{ dem1997, rou2007, div2008a,
div2008b}), we may simplify as follows our monomial:
\[
h^lu_1^{i_1}\cdots u_{n-1}^{i_{n-1}}
\underline{u_n^{n-1}}_{\int}
=
h^l
u_1^{i_1}\cdots u_{n-1}^{i_{n-1}}\cdot 1
=
h^l
u_1^{i_1}\cdots u_{n-1}^{i_{n-1}}.
\]
We shall call this second kind of reduction process:
\[
{\text{\sl ``{\sl fiber-integration}''}}.
\]

The third process of course consists in substituting the two
relations~\thetag{\ref{c-l-k}} and~\thetag{\ref{u-n}} as many times as
necessary. With $r = n$ and without any $\pi_{ j, k}^* ( \bullet)$,
they now read:
\begin{equation}
\label{c-j-ell}
c_j^{[\ell]}
=
\sum_{k=0}^j\,\lambda_{j,j-k}\cdot
c_k^{[\ell-1]}
\big(u_\ell\big)^{j-k},
\end{equation}
where $1 \leqslant j, \, \ell \leqslant n$, with the conventions
$c_0^{ [\ell]} = 1$ and $c_j^{ [0]} = c_j$, where we set
\[
\lambda_{j,j-k}
:=
{\textstyle{\binom{n-k}{j-k}}}
-
{\textstyle{\binom{n-k}{j-k-1}}}
=
{\textstyle{\frac{(n-k)!}{(j-k)!\,(n-j)!}}}
-
{\textstyle{\frac{(n-k)!}{(j-k-1)!(n-j+1)!}}},
\]
and also, with upper indices of $u_\ell$ denoting exponents:
\begin{equation}
\label{u-ell-n}
u_\ell^n
=
-
c_1^{[\ell-1]}\,u_\ell^{n-1}
-
c_2^{[\ell-1]}\,u_\ell^{n-2}
-\cdots-
c_{n-1}^{[\ell-1]}\,u_\ell
-
c_n^{[\ell-1]}. 
\end{equation}

\subsection*{ Estimating the coefficient of $d^{ n+1}$}
Our first main task is to reach a lower bound ${\sf G}_{ n+1} - \delta
\, {\sf E}_{ n+1}'$ for the coefficient of $d^{ n+1}$ in $\Pi_\delta$,
and this cannot be straightforward, because there are {\em very
numerous} monomials in the expansion of $\Pi_\delta$. In a first
reading, one might jump directly to Subsection~\ref{highlighting} just
after Corollary~\ref{before-minorating}.
Here is an initial observation.

\begin{lemma}[\cite{ div2008b}]
\label{h-1}
Assume $l + i_1 + \cdots + i_n = n^2$ or $l + 1 + j_1 + \cdots + j_n =
n^2$. Then as soon as $l \geqslant 1$, one has: 
\[
0
=
{\sf coeff}_{d^{n+1}}
\big[
h^lu_1^{i_1}\cdots u_n^{i_n}
\big]
\ \ \ \ \ 
\text{\rm and}
\ \ \ \ \
0
=
{\sf coeff}_{d^{n+1}}
\big[
h^lc_1u_1^{j_1}\cdots u_n^{j_n}
\big].
\]
\end{lemma}

\proof
Indeed, after reduction of either $u$-monomial in terms of the Chern
classes $c_k$ of the base, one obtains a sum with integer coefficients
of terms of the form:
\[
h^lc_1^{\lambda_1}c_2^{\lambda_2}\cdots
c_n^{\lambda_n}
\]
with $l + \lambda_1 + 2 \lambda_2 + \cdots + n \lambda_n = n$. But
then if we replace the Chern classes by their expressions~\thetag{
\ref{c-d}} in terms of $h$ and of the degree, we get:
\[
\aligned
{\sf coeff}_{d^{n+1}}
\big[
h^lc_1^{\lambda_1}c_2^{\lambda_2}
\cdots
c_n^{\lambda_n}
\big]
&
=
{\sf coeff}_{d^{n+1}}
\big[
(-1)^{\lambda_1+\cdots+\lambda_n}\,h^n\cdot
d^{\lambda_1+\lambda_2+\cdots+\lambda_n}
+
{\sf l.o.t}
\big]
\\
&
=
{\sf coeff}_{d^{n+1}}
\big[
(-1)^{\lambda_1+\cdots+\lambda_n}\,d\cdot
d^{\lambda_1+\lambda_2+\cdots+\lambda_n}
+
{\sf l.o.t}
\big]
\\
&
=
0,
\endaligned
\]
since $1 + \lambda_1 + \lambda_2 + \cdots + \lambda_n \leqslant l +
\lambda_1 + 2 \lambda_2 + \cdots + n \lambda_n = n$.
\endproof

As a result, a glance at~\thetag{\ref{pi-delta}} immediately shows that: 
\[
\small
\aligned
\text{\rm coeff}_{d^{n+1}}\!
\big[\Pi_\delta\big]
=
\text{\rm coeff}_{d^{n+1}}
\Big[
\big(a_1u_1+\cdots+a_nu_n\big)^{n^2}
+
\delta\vert\mathbf{a}\vert c_1
\big(a_1u_1+\cdots+a_nu_n\big)^{n^2-1}
\Big].
\endaligned
\]

\subsection{ Reverse lexicographic ordering for the $u$-monomials}
We order the collection of all homogeneous monomials $u_1^{ i_1}
\cdots u_n^{ i_n}$ with $i_1 + \cdots + i_n = n^2$ appearing in the
expansion of $\big( a_1 u_1 + \cdots + a_n u_n \big)^{ n^2}$ above by
declaring that the monomial $u_1^{ i_1} \cdots u_n^{ i_n}$ is {\sl
smaller}, for the {\sl reverse lexicographic ordering}, than another
monomial $u_1^{ j_1} \cdots u_n^{ j_n}$, again of course with $j_1 +
\cdots + j_n = n^2$, if:
\[
\left\{
\aligned
&
i_n>j_n
\\
\text{\rm or if}\ \
&
i_n=j_n\ \
\text{\rm but}\ \
i_{n-1}>j_{n-1}
\\
\cdots\cdots
&
\cdots\cdots\cdots\cdots\cdots\cdots\cdots\cdots\cdot\cdot
\\
\text{\rm or if}\ \
&
i_n=j_n,\,\dots,\,i_3=j_3\ \
\text{\rm but}\ \
i_2>j_2.
\endaligned
\right.
\]
Observe that $i_n = j_n, \dots, i_2 = j_2$ implies $i_1 = j_1$. An
equivalent language says that the multiindices themselves are ordered in 
this way:
\[
(i_1,\dots,i_n)
<_{\sf revlex}
(j_1,\dots,j_n).
\]

\begin{proposition}
\label{bigger-central}
The coefficient of $d^{ n+1}$ in any monomial $u_1^{ i_1} \cdots u_n^{
i_n}$ which is {\em larger} than $u_1^n \cdots u_n^n$ is zero:
\[
{\sf coeff}_{d^{n+1}}
\big[u_1^{i_1}\cdots u_n^{i_n}\big]
=
0
\ \ \ \ \ \
\text{\em for any}
\ \ \ \ \ \
(i_1,\dots,i_n)
>_{\sf revlex}
(n,\dots,n).
\]
\end{proposition}

\proof
Thus, assume $(i_1, \dots, i_n) >_{ \sf revlex} (n, \dots,
n)$. Firstly, if $i_n = n$, the claimed vanishing property is in all
concerned subcases yielded by {\bf (iii)} of the lemma just
below. Secondly, if $i_n = n -1$, an integration on the fiber of
$\pi_{ n-1, n} : X_n \to X_{ n-1}$ replaces $u_n^{ n-1}$ by the
constant $+1$, hence we are left with $u_1^{ i_1} \cdots u_{ n-1}^{
i_{ n-1}}$ and {\bf (i)} of the same lemma then yields the
conclusion. Thirdly and lastly, if $i_n \leqslant n-2$, then the form
$u_1^{ i_1} \cdots u_{ n-1}^{ i_{ n-1}}$ vanishes identically for
degree-form reasons. Thus, granted the lemma, the proposition is
proved.
\endproof

\begin{lemma}
\label{i-ii-iii-iv}
The coefficient of $d^{ n+1}$ in all the following four sorts of
$u$-monomials is equal to zero:

\begin{itemize}

\smallskip\item[{\bf (i)}]
$u_1^{ i_1 } \cdots u_k^{ i_k }$ for any $k \leqslant n-1$ and any
$i_1, \dots, i_k$ with $i_1 + \cdots + i_k = n + k ( n-1)${\rm ;}

\smallskip\item[{\bf (ii)}]
$(c_1)^{ n-k }\,u_1^{ i_1 } \cdots u_k^{ i_k }$ for any $k
\leqslant n-1$, and any $i_1, \dots, i_k$ with $i_k \leqslant n-1$ and
$i_1 + \cdots + i_k = kn${\rm ;}

\smallskip\item[{\bf (iii)}]
$u_1^{i_1} \cdots u_l^{ i_l} \, u_{ l+1}^n \cdots u_n^n$ for any $l
\leqslant n$, any $i_1, \dots, i_l$ with $i_l \leqslant n-1$ and $i_1
+ \cdots + i_l = ln${\rm ;}

\smallskip\item[{\bf (iv)}]
$c_1 u_1^{ i_1} \cdots u_l^{ i_l} u_{ l+1}^n \cdots u_{ n-1}^n$
for any $l \leqslant n-1$, any $i_l \leqslant n-1$, any $i_1, \dots,
i_l$ with $i_1+ \cdots + i_l = ln$.

\end{itemize}
\end{lemma}

\proof
Property {\bf (i)} is established in Section~3 of~\cite{ div2008b}.
So {\bf (i)} holds.

Applying~\thetag{\ref{c-j-ell}} written for $j = 1$, namely 
$c_1^{[ \ell]} = c_1^{[ \ell - 1]} + ( n-1)\, u_\ell$, we get:
\begin{equation}
\label{c-1-ell}
c_1^{[\ell]}
=
c_1
+
(n-1)\,u_1
+\cdots+
(n-1)\,u_\ell.
\end{equation}

To begin with, we start from {\bf (i)} for $k = n-1$, $i_{ n-1} = n$
and $i_1 + \cdots + i_{ n-2} = n + (n-1)(n-1) - i_{ n-1} = n^2 - 2n
+1$ arbitrary, namely:
\[
0
=
{\sf coeff}_{d^{n+1}}
\big[
u_1^{i_1}\cdots u_{n-2}^{i_{n-2}}u_{n-1}^n
\big].
\]
Next, thanks to \thetag{\ref{u-ell-n}}, we may replace in this
equality $u_{ n-1}^n$ by $- c_1^{[ n-2]} u_{ n-1}^{ n-1} -
c_2^{[n-2]}u_{n-1}^{n-2} - \cdots - c_n^{[ n-2]}$:
\[
\footnotesize
\aligned
0
&
=
{\sf coeff}_{d^{n+1}}
\Big[
u_1^{i_1}\cdots u_{n-2}^{i_{n-2}}
\big(
-c_1^{[n-2]}u_{ n-1}^{ n-1}
-
\underline{c_2^{[n-2]}u_{n-1}^{n-2}
-\cdots-
c_n^{[n-2]}}_\circ
\big)
\Big]
\\
&
=
{\sf coeff}_{d^{n+1}}
\big[
u_1^{i_1}\cdots u_{n-2}^{i_{n-2}}
\big(
-c_1^{[n-2]}\,u_{n-1}^{n-1}
\big)
\big]
\ \ \ \ \ \
\explain{degree-form reasons}
\ \ \ \ \
\explain{use \thetag{\ref{c-1-ell}}}
\\
&
=
{\sf coeff}_{d^{n+1}}
\big[
u_1^{i_1}\cdots u_{n-2}^{i_{n-2}}
\big(
-c_1-
\underline{(n-1)u_1-\cdots-(n-1)u_{n-2}}_\circ
\big)\,u_{n-1}^{n-1}
\big]
\\
&
=
{\sf coeff}_{d^{n+1}}
\big[
-c_1u_1^{i_1}\cdots u_{n-2}^{i_{n-2}}u_{n-1}^{n-1}
\big]
\ \ \ \ \ \
\explain{apply {\bf (i)} again},
\endaligned
\]
and we therefore get {\bf (ii)} for $k = n-1$ when $i_{ n-1} =
n-1$. But in all the other remaining cases when $i_{n-1} \leqslant
n-2$, then by the assumption that the sum of the 
indices $i_l$ is equal to $(n-1)n$:
\[
i_1
+\cdots+ 
i_{ n-2}
\geqslant
(n-1)n-(n-2)=n^2-2n+2 
= 
\dim X_{n-2},
\]
and consequently, the degree of the form $c_1 u_1^{ i_1} \cdots
u_{ n-2}^{ i_{ n-2}}$ is $\geqslant 1 + \dim X_{ n-2}$, whence this
form vanishes identically. Thus {\bf (ii)} is proved completely for
$k = n-1$.

Next, consider {\bf (iii)} for $l = n$. If $i_n \leqslant n-2$, then
by degree-form reasons $0 \equiv u_1^{ i_1} \cdots u_{ n-1}^{ i_{
n-1}}$, whence ${\sf coeff}_{ d^{ n+1}} \big[ u_1^{ i_1}
\cdots u_{ n-1}^{ i_{n-1}} u_n^{ i_n} \big] = 0$ gratuitously. So we
assume $i_n = n-1$. But then $i_1 + \cdots + i_{ n-1} = n^2 - n +1$,
hence {\bf (i)} applies to give:
\[
\aligned
0
&
=
{\sf coeff}_{d^{n+1}}
\big[
u_1^{i_1}\cdots u_{n-1}^{i_{n-1}}
\big]
\ \ \ \ \ \ \
\explain{reconstitute hidden integration of $u_n^{n-1}$}
\\
&
=
{\sf coeff}_{d^{n+1}}
\big[
u_1^{i_1}\cdots u_{n-1}^{i_{n-1}}u_n^{n-1}
\big],
\endaligned
\]
and therefore this proves {\bf (iii)} completely for $l = n$. But we
also get at the same time the property {\bf (iii)} for $l = n-1$. 
Indeed, with $i_1 + \cdots + i_{ n-1} = (n-1)n$ and with $i_{
n-1} \leqslant n-1$, we may reduce, using \thetag{ \ref{u-ell-n}}:
\[
\footnotesize
\aligned
u_1^{i_1}\cdots u_{n-1}^{i_{n-1}}u_n^n
&
=
u_1^{i_1}\cdots u_{n-1}^{i_{n-1}}
\big[
-c_1^{[n-1]}u_n^{n-1}
-
\underline{c_2^{[n-1]}u_n^{n-2}
-\cdots-c_n^{[n-1]}}_\circ
\big]
\\
&
=
u_1^{i_1}\cdots u_{n-1}^{i_{n-1}}
\big[
-c_1^{[n-1]}u_n^{n-1}
\big]
\ \ \ \ \ \ \
\explain{degree-form reasons}
\ \ \ \ \
\explain{use \thetag{\ref{c-1-ell}}}
\\
&
=
u_1^{i_1}\cdots u_{n-1}^{i_{n-1}}
\big[
-c_1
-(n-1)u_1
-\cdots
-(n-1)u_{n-1}
\big]
\endaligned
\]
Thanks to {\bf (i)}, after expansion, the pure $u$-monomials give no
contribution to $d^{ n+1}$, and consequently:
\[
{\sf coeff}_{d^{n+1}}
\big[
u_1^{i_1}\cdots u_{n-1}^{i_{n-1}}u_n^n
\big]
=
{\sf coeff}_{d^{n+1}}
\big[
-c_1u_1^{i_1}\cdots u_{n-1}^{i_{n-1}}
\big]
=
0,
\]
where the last equality holds true thanks to the property {\bf (ii)}
already proved for $k = n-1$. Thus {\bf (iii)} is completely proved
for $l = n$ and for $l = n-1$.

Lastly, we just observe that {\bf (iv)} for $l = n-1$ coincides with
{\bf (ii)} for $k = n-1$. In summary, we have completed a first loop
of proofs.

Consider now the second loop. We start from {\bf (ii)} for $k = n-1$
(already got) with $i_{ n-1} = n-1$ and with $i_{ n-2} = n$, so that
$i_1 + \cdots + i_{ n-3} = (n-1)n - i_{ n-2} - i_{ n -1} = n^2 - 3n +
1$, and then we compute:
\[
\small
\aligned
0
&
=
{\sf coeff}_{d^{n+1}}
\big[
c_1u_1^{i_1}\cdots
u_{n-3}^{i_{n-3}}u_{n-2}^n
\underline{u_{n-1}^{n-1}}_{\int}
\big]
\ \ \ \ \ 
\explain{fiber-integration}
\\
&
=
{\sf coeff}_{d^{n+1}}
\big[
c_1u_1^{i_1}\cdots u_{n-3}^{i_{n-3}}
\big(
-c_1^{[n-3]}u_{n-2}^{n-1}
-\underline{c_2^{[n-3]}u_{n-2}^{n-2}
-\cdots
-c_n^{[n-3]}}_\circ
\big)
\big]
\ \ \ \ \ \
\explain{use \thetag{\ref{u-ell-n}}}
\\
&
=
{\sf coeff}_{d^{n+1}}
\big[
c_1u_1^{i_1}\cdots u_{n-3}^{i_{n-3}}
\big(
-c_1^{[n-3]}
\big)u_{n-2}^{n-1}
\big]
\ \ \ \ \ \
\explain{degree-form reasons}
\ \ \ \ \
\explain{use \thetag{\ref{c-1-ell}}}
\\
&
=
{\sf coeff}_{d^{n+1}}
\big[
c_1u_1^{i_1}\cdots u_{n-3}^{i_{n-3}}
\big(
-c_1-
\underline{(n-1)u_1-\cdots-(n-1)u_{n-3}}_\circ
\big)u_{n-2}^{n-1}
\underline{u_{n-1}^{n-1}}_{\int}
\big]
\\
&
=
{\sf coeff}_{d^{n+1}}
\big[
-c_1c_1u_1^{i_1}\cdots u_{n-3}^{i_{n-3}}u_{n-2}^{n-1}
\underline{u_{n-1}^{n-1}}_{\int}
\big]
\ \ \ \ \ \
\explain{apply {\bf (ii)} for $k = n-1$ again}
\\
&
=
{\sf coeff}_{d^{n+1}}
\big[
-c_1c_1u_1^{i_1}\cdots u_{n-3}^{i_{n-3}}u_{n-2}^{n-1}
\big]
\ \ \ \ \ \
\explain{fiber-integration},
\endaligned
\]
where we have reintroduced $u_{ n-1}^{ n-1}$ (artificially) in the
fourth line, so as to apply {\bf (ii)} for $k = n-1$ (got). As a
result of the last obtained equation, we have gained {\bf (ii)} for $k
= n-2$ when $i_{ n-2} = n-1$, but since when $i_{ n-2 } \leqslant
n-2$, the form $c_1 c_1 u_1^{ i_1} \cdots u_{ n-3}^{ i_{ n-3}}$
vanishes identically for degree reasons, we finally have fully
established {\bf (ii)} for $k = n-2$.

Next, we look at {\bf (iii)} for $l = n - 2$. Then $i_1 + \cdots +
i_{ n-2} = (n -2)n$ with $i_{ n-2} \leqslant n - 1$. So we ask
whether the following coefficient vanishes:
\[
\small
\aligned
&
{\sf coeff}_{d^{n+1}}
\big[
u_1^{i_1}\cdots u_{n-2}^{i_{n-2}}u_{n-1}^nu_n^n
\big]
=
\\
&
=
{\sf coeff}_{d^{n+1}}
\big[
u_1^{i_1}\cdots u_{n-2}^{i_{n-2}}u_{n-1}^n
\big(
c_1-
\underline{(n-1)u_1-\cdots-(n-1)u_{n-1}}_\circ
\big)
\big]
\\
&
=
{\sf coeff}_{d^{n+1}}
\big[
-c_1u_1^{i_1}\cdots u_{n-2}^{i_{n-2}}u_{n-1}^n
\big]
\\
&
=
{\sf coeff}_{d^{n+1}}
\big[
-c_1u_1^{i_1}\cdots u_{n-2}^{i_{n-2}}
\big(
-c_1-
\underline{(n-1)u_1-\cdots-(n-1)u_{n-2}}
\big)
u_{n-1}^{n-1}
\big]
\\
&
=
{\sf coeff}_{d^{n+1}}
\big[
c_1c_1u_1^{i_1}\cdots u_{n-2}^{i_{n-2}}
\underline{u_{n-1}^{n-1}}_{\int}
\big]
\\
&
=0,
\endaligned
\]
and in fact, this coefficient vanishes actually, thanks to {\bf (ii)}
for $k = n-2$ seen a moment ago. This therefore proves {\bf (iii)} for
$l = n - 2$ completely.

Finally, consider {\bf (iv)} for $l = n- 2$. Then $i_1 + \cdots + i_{
n-2 } = ( n - 2 ) n$ and $i_{ n-2} \leqslant n - 1$. But coming back
to the third line of the equations just above, where $i_{n-2}
\leqslant n - 1$ too, we have in fact already implicitly proved that:
\[
0 
= 
{\sf coeff}_{d^{n+1}} 
\big[c_1u_1^{i_1}
\cdots
u_{n-2}^{i_{n-2}}u_{n-1}^n\big],
\]
and this is {\bf (iv)} for $l = n-2$. Thus, the second loop is
completed, and the general induction, similar, is now intuitively
clear.
\endproof

\begin{corollary}
\label{before-minorating}
The coefficient of $d^{ n+1}$ in any monomial $c_1 u_1^{ j_1}
\cdots u_{ n-1}^{ j_{ n-1}} u_n^{ j_n}$ with $1 + j_1 + \cdots + j_{ n-1}
+ j_n = n^2$ which is larger than $c_1 u_1^n \cdots u_{ n -
1}^n u_n^{ n-1}$ is zero:
\[
\small
\aligned
{\sf coeff}_{d^{n+1}}
&
\big[
c_1u_1^{j_1}\cdots u_{n-1}^{j_{n-1}}u_n^{j_n}
\big]
=
0,
\\
&
\ \ \ \ \ \ \ \ \ \ \ \ \ \ \ \ \ \
\text{\rm for any}
\ \ \ \ \
(j_1,\dots,j_{n-1},j_n)
>_{\sf revlex}
(n,\dots,n,n-1).
\endaligned
\]
Furthermore:
\[
\aligned
{\sf coeff}_{d^{n+1}}
\big[
u_1^n\cdots u_{n-1}^nu_n^n
\big]
&
=
{\sf coeff}_{d^{n+1}}
\big[(-1)^n(c_1)^n\big]
=
+1.
\\
{\sf coeff}_{d^{n+1}}
\big[
c_1u_1^n\cdots u_{n-1}^nu_n^{n-1}
\big]
&
=
{\sf coeff}_{d^{n+1}}
\big[(-1)^{n-1}(c_1)^n\big]
=
-1.
\endaligned
\]
\end{corollary}

\proof
The first claim is just a rephrasing of the property {\bf (iv)} of the
lemma, after one notices that $c_1 u_1^{j_1} \cdots u_{n-1}^{ j_{n-1}}
u_n^{ j_n}$ vanishes identically for degree reasons when $j_n
\leqslant n-2$, while the term $u_n^{ n-1} = u_n^{ j_n}$ disappears
after fiber integration when $j_n = n - 1$. The identities stated just
after now have obvious proofs.
\endproof

\subsection{ Minorating ${\sf coeff}_{ d^{ n+1}} 
\big[ \Pi \big]$}
\label{highlighting}
Let us decompose the intersection product $\Pi_\delta$
defined by~\thetag{\ref{pi-delta}} as $\Pi + \delta \Pi'$, where:
\[
\small
\aligned
\Pi
&
:=
\big(a_1u_1+\cdots+a_nu_n+2\vert\mathbf{a}\vert h\big)^{n^2}
-
n^2h
\big(a_1u_1+\cdots+a_nu_n+2\vert\mathbf{a}\vert h\big)^{n^2-1}\,
2\vert\mathbf{a}\vert,
\\
\Pi'
&
:=
n^2c_1\big(a_1u_1+\cdots+a_nu_n+2\vert\mathbf{a}\vert h\big)^{n^2-1}\,
\vert\mathbf{a}\vert.
\endaligned
\]
The (ineffective) Lemma~\ref{ineffective-reduction} insures that the
reduction of $\Pi$ in terms of $d = \deg X$ is a certain polynomial:
\[
{\sf P}_{\bf a}(d)
=
\sum_{k=0}^{n+1}\,
{\sf p}_{k,{\bf a}}\,d^k, 
\]
having certain coefficients ${\sf p}_{ k, {\bf a}} \in \mathbb{ Z}
\big[ a_1, \dots, a_n \big]$. Moreover, Lemma~\ref{h-1} showed
that positive powers of $h$ do not contribute to the
leading coefficient, whence:
\[
\aligned
{\sf p}_{n+1,{\bf a}}
=
{\sf coeff}_{d^{n+1}}
\big[\Pi\big]
&
=
{\sf coeff}_{d^{n+1}}
\big[
\big(
a_1u_1+\cdots+a_nu_n
\big)^{n^2}
\big]
\\
&
=
{\sf coeff}_{d^{n+1}}
\big[
\big(
a_1u_1+\cdots+a_nu_n+2\vert{\bf a}\vert h
\big)^{n^2}
\big].
\endaligned
\]
By  Proposition 2 in~\cite{ div2008b}, the bundle: 
\[
\mathcal O_{X_n}(\mathbf a)\otimes
\pi_{0,n}^*\mathcal O_{X}(2|\mathbf a|)
\]
is nef whenever $\mathbf a$ belongs to the cone defined by~\thetag{
\ref{rel-nef-cone}}, therefore its top self-intersection must be
non-negative. Thus, once this top self-intersection is evaluated in term of
the degree $d$ of the hypersurface, its dominating coefficient must be
non-negative, too. In other words we must have:
\[
{\sf p}_{n+1,\mathbf{a}}
\geqslant 
0.
\]
But from the corollary just above, we know that ${\sf p}_{ n+1,
{\bf a}} \in \mathbb{ Z} [ {\bf a} ]$ is not identically
zero, for it incorporates at least the nonzero (central) monomial: 
\[
{\sf coeff}_{ d^{ n+1}}
\big[
{\textstyle{\frac{n^2!}{n!\,\cdots\,n!}}}\,
a_1^n\cdots a_n^n\,
u_1^n\cdots u_n^n
\big]
=
{\textstyle{\frac{n^2!}{n!\,\cdots\,n!}}}\,
a_1^n\cdots a_n^n. 
\]
Then, in order to capture a weight ${\bf a}$ for which ${\sf p}_{ n+1,
{\bf a}} > 0$, we at first observe that the cube of $\mathbb{ N}^n$ having
edges of length $n^2$ which consists of all integers $(a_1, \dots,
a_n)$ satisfying the inequalities:
\[
\small
\aligned
1\leqslant a_n\leqslant 1+n^2,
\ \ \ \ \
3n^2\leqslant a_{n-1}\leqslant (3+1)n^2,
\ \ \ \ \
(3^2+3)n^2\leqslant a_{n-2}\leqslant (3^2+3+1)n^2
\\
\dots,
\ \ \ \ \
(3^{n-1}+\cdots+3)n^2
\leqslant
a_1
\leqslant
(3^{n-1}+\cdots+3+1)n^2
\endaligned
\]
is visibly contained in the cone in question:
\[
a_n\geqslant
1,
\ \ \ \ \
a_{n-1}\geqslant 2a_n,
\ \ \ \ \ 
a_{n-2}
\geqslant 
3a_{n-1},
\dots,
a_1\geqslant 3a_2. 
\] 
We now claim that there exists at least one $n$-tuple of integers
${\bf a}^\ast = (a_1^\ast, \dots, a_n^\ast)$ belonging to this cube
with the property that ${\sf p}_{ n+1, {\bf a}^\ast}$ is nonzero, and
hence:
\[
{\sf p}_{n+1,{\bf a}^\ast}
\geqslant
1
=:
{\sf G}_{n+1},
\]
so that we can take $1$ as the minorant introduced at the beginning.
Indeed, ${\sf p}_{ n+1, {\bf a}}$ is a homogeneous polynomial of
degree $n^2$ to which an elementary lemma applies.

\begin{lemma}
Let ${\sf q} = {\sf q} ( b_1, \dots, b_\nu) \in \mathbb{ Z} \big[ b_1,
\dots, b_\nu \big]$ be a polynomial of degree $c\geqslant 1$.  Then
${\sf q}$ can vanish at all points of a cube of integers having edges
of length equal to its degree $c$ only when it is identically zero.
\end{lemma}

\proof
Expand ${\sf q} = \sum_{ k_1 = 0}^c \, b_1^{ k_1} \, {\sf q}_{ k_1} (
b_2, \dots, b_\nu)$, recognize a $(c+1) \times ( c+1)$ Van der Monde
determinant, deduce that each ${\sf q}_{ k_1} ( b_2, \dots, b_\nu)$
vanishes at all points of a similar cube in a space of dimension $\nu
- 1$, and terminate by induction.
\endproof

\subsection{ Majorating the other coefficients ${\sf coeff}_{
d^k } \big[ \Pi \big]$} Now, for such an ${\bf a}^\ast$ which is not
very precisely located in the cube, we nevertheless have the effective
control, which is useful below:
\[
\max_{1\leqslant i\leqslant n}\,
a_i^\ast
=
a_1^\ast
=
{\textstyle{\frac{3^n-1}{2}}}\,n^2
\leqslant
{\textstyle{\frac{3^n}{2}}}\,n^2. 
\]
From now on, we shall simply denote ${\bf a}^\ast$ by ${\bf a}$. 
At present, for any integer $k$ with $0
\leqslant k \leqslant n$, let us denote by ${\sf D}_k ( n)$ any
available bound ({\em see} in advance Theorem~\ref{D-k-n}) in
terms of $n$ only for the maximal absolute value of the coefficient of
$d^k$ in all monomials $h^l u_1^{ i_1} \cdots u_n^{ i_n}$ with $l +
i_1 + \cdots + i_n = n^2$, namely:
\[
\max_{l+i_1+\cdots+i_n=n^2}
\big\vert
{\sf coeff}_{d^k}
\big[
h^lu_1^{i_1}\cdots u_n^{i_n}
\big]
\big\vert
\leqslant
{\sf D}_k(n).
\]
Then for any $k$ with $0 \leqslant k \leqslant n$, we now aim 
at estimating from above the coefficient of $d^k$ in our
intersection product $\Pi$, using two new lemmas and starting
from its expansion, all terms of which we shall
have to control:
\[
\small
\aligned
&
\big\vert
{\sf coeff}_{d^k}
\big[\Pi\big]
\big\vert
\leqslant
\\
&
\leqslant
\sum_{l+i_1+\cdots+i_n=n^2}
{\textstyle{\frac{n^2!}{l!\,i_1!\,\cdots\,i_n!}}}
\cdot
(2\vert{\bf a}\vert)^l
a_1^{i_1}\cdots a_n^{i_n}
\cdot
\big\vert
{\sf coeff}_{d^k}
\big[h^lu_1^{i_1}\cdots u_n^{i_n}\big]
\big\vert
+
\\
&
+
\sum_{l+j_1+\cdots+j_n=n^2-1}\,
n^2\,
{\textstyle{\frac{(n^2-1)!}{l!\,j_1!\,\cdots\,j_n!}}}
\cdot
2\vert{\bf a}\vert
(2\vert{\bf a}\vert)^l
a_1^{j_1}\cdots a_n^{j_n}
\cdot
\big\vert
{\sf coeff}_{d^k}
\big[hh^lu_1^{j_1}\cdots u_n^{j_n}\big]
\big\vert. 
\endaligned
\]

\begin{lemma}
\label{stirling-2}
Let $l$, $i_1, \dots, i_n\in \mathbb{ N}$ satisfying $l + i_1 + \cdots
+ i_n = n^2$ and let $l$, $j_1, \dots, j_n \in \mathbb{ N}$ satisfying
$l + j_1 + \cdots + j_n = n^2 - 1$. Then:
\[
{\textstyle{\frac{n^2!}{l!\,i_1!\,\cdots\,i_n!}}}
\leqslant
(n+1)^{n^2}
\ \ \ \ \ 
\text{and:}
\ \ \ \ \ 
n^2\,{\textstyle{\frac{(n^2-1)!}{l!\,j_1!\,\cdots\,j_n!}}}
\leqslant
(n+1)^{n^2+1}.
\]
Furthermore, the number of summands in $\sum_{ l + i_1 + \cdots + i_n
= n^2}$ and the number of summands in $\sum_{ l + j_1 + \cdots + j_n =
n^2 - 1}$, which are both plain binomial coefficients, enjoy the
following two elementary majorations:
\[
{\textstyle{\frac{(n^2+n)!}{n^2!\,n!}}}
\leqslant
4\,n^{2n-1}
\ \ \ \ \
\text{and:}
\ \ \ \ \
{\textstyle{\frac{(n^2-1+n)!}{(n^2-1)!\,n!}}}
\leqslant
2\,n^{2n-1}.
\]
\end{lemma}

\proof
Indeed, any multinomial coefficient ${\textstyle{ \frac{ n^2!}{
l!\,i_1! \,\cdots\, i_n! }}}$ is less than or equal to the sum of all
multinomial coefficients $(1 + 1 + \cdots + 1)^{ n^2} = (n+1)^{
n^2}$. At the same time, we deduce: $n^2 \, \frac{ (n^2 - 1)!}{ l!
j_1! \cdots j_n!} = n^2 ( n+1)^{ n^2 -1} \leqslant (n+1)^{ n^2 + 1}$.

For the second claim, we as a preliminary have:
\[
{\textstyle{\frac{(n^2+n-1)!}{n^2!\,(n-1)!}}}
=
{\textstyle{\frac{(n^2 + 1)\cdots(n^2+n-1)}{1\,\,\cdots\,\,(n-1)}}} 
\leqslant 
{\textstyle{\frac{(n^2+n^2)\cdots(n^2+n^2)}{(n-1)!}}} 
=
{\textstyle{\frac{2^{n-1}\,n^{2n-2}}{(n-1)!}}} 
\leqslant 
2\,n^{ 2n-2},
\] 
since $2^{ n-1} \leqslant 2 \, (n-1) !$ for any $n \geqslant 1$.
Consequently, we deduce:
\[
{\textstyle{\frac{(n^2+n)!}{n^2!\,n!}}}
=
{\textstyle{\frac{(n^2+n-1)!}{n^2!\,(n-1)!}}}\cdot
{\textstyle{\frac{(n^2+n)}{n}}}
\leqslant
2\,n^{2n-2}\cdot(n+
{\textstyle{\frac{1}{n}}})
\leqslant
4\,n^{2n-1},
\]
and similarly: ${\textstyle{ \frac{ (n^2-1+n)! }{( n^2-1)! \, n! }}}
\leqslant {\textstyle{ \frac{ (n^2+n-1) !}{ n^2 ! \, (n-1)! }}} \cdot
{\textstyle{ \frac{ n^2}{ n}}} \leqslant 2\, n^{ 2n-2} \cdot n = 2\,
n^{ 2n-1 }$.
\endproof

\begin{lemma}
For any $l, i_1, \dots, i_n\in \mathbb{ N}$ satisfying $l + i_1 +
\cdots + i_n = n^2$, one has:
\[
(2\vert{\bf a}\vert)^l
a_1^{i_1}\cdots a_n^{i_n}
\leqslant
n^{3n^2}\,3^{n^3}. 
\]
\end{lemma}

\proof
Indeed, we majorate each $a_i$ by $\vert {\bf a} \vert$ 
and $\vert {\bf a} \vert = a_1
+ \cdots + a_n$ by $n a_1$, and also $l$ by $n^2$, so that $(2\vert
{\bf a} \vert)^l a_1^{ i_1} \cdots a_n^{i_n} \leqslant 2^{ n^2} \big(
na_1\big)^{ n^2}$ and we apply $a_1 \leqslant \frac{ 3^n}{ 2} \, n^2$.
\endproof

Thanks to these two lemmas, we may perform
majorations:
\[
\small
\aligned
\big\vert
{\sf coeff}_{d^k}
\big[\Pi\big]
\big\vert
&
\leqslant
4\,n^{2n-1}\cdot (n+1)^{n^2}
\cdot 
n^{3n^2}\,3^{n^3}
\cdot{\sf D}_k(n)
+
\\
&\ \ \ \ \ 
+
2\,n^{2n-1}\cdot (n+1)^{n^2+1}
\cdot 
n^{3n^2}\,3^{n^3}
\cdot{\sf D}_k(n)
\\
&
\leqslant
6\,n^{2n-1}\cdot (n+1)^{n^2+1}
\cdot
n^{3n^2}\,3^{n^3}
\cdot{\sf D}_k(n)
\ \ \ \ \ \
{\scriptstyle{(k\,=\,0,\,\dots,\,n)}}.
\endaligned
\]

\begin{lemma}
For any exponent $k$ with $0 \leqslant k \leqslant n$, one has:
\[
\big\vert
{\sf coeff}_{d^k}
\big[\Pi\big]
\big\vert
\leqslant
6\,n^{2n-1}\cdot(n+1)^{n^2}
\cdot
n^{3n^2}\,3^{n^3}
\cdot{\sf D}_k(n).
\qed
\]
\end{lemma}

To conclude these estimates, for any integer $k = 0, 1, \dots, n,
n+1$, let us denote by ${\sf D}_k' ( n)$ any available majorant
for all the monomials appearing in $\Pi'$: 
\[
\max_{1+l+j_1+\cdots+j_n=n^2}
\big\vert
{\sf coeff}_{d^k}\big[c_1h^lu_1^{j_1}\cdots
u_n^{j_n}\big]
\big\vert
\leqslant
{\sf D}_k'(n).
\]

\begin{lemma}
For any exponent $k$ with $0 \leqslant k\leqslant n+1$, one has:
\[
\big\vert
{\sf coeff}_{d^k}
\big[\Pi'\big]
\big\vert
\leqslant
n^{2n-1}\cdot(n+1)^{n^2+1}
\cdot 
n^{3n^2}\,3^{n^3}
\cdot
{\sf D}_k'(n).
\]
\end{lemma}

\proof
Indeed, one performs the similar majorations:
\[
\small
\aligned
&
\big\vert
{\sf coeff}_{d^k}
\big[
\Pi'
\big]
\big\vert
\leqslant
\\
&
\leqslant
\sum_{l+j_1+\cdots+j_n=n^2-1}
n^2\,
{\textstyle{\frac{(n^2-1)!}{l!\,j_1!\,\cdots\,j_n!}}}
\cdot
\vert{\bf a}\vert
(2\vert{\bf a}\vert)^l
a_1^{j_1}\cdots a_n^{j_n}
\cdot
\big\vert
{\sf coeff}_{d^k}
\big[
c_1h^lu_1^{j_1}\cdots u_n^{j_n}
\big]
\big\vert
\\
&
\leqslant
2\,n^{2n-1}
\cdot 
(n+1)^{n^2+1}\cdot
{\textstyle{\frac{1}{2}}}\,
n^{3n^2}\,3^{n^3}
\cdot
{\sf D}_k'(n)
\\
&
\leqslant
n^{2n-1}\cdot(n+1)^{n^2+1}
\cdot
n^{3n^2}\,3^{n^3}
\cdot
{\sf D}_k'(n),
\endaligned
\]
hence the bound we obtain is exactly the same, up to the factor 6.
\endproof

\subsection{ Final effective estimations}
We can now explain how to achieve the proof of Theorem~\ref{main}. At
first, we shall realize in Section~\ref{Section-5} 
that both constant coefficients
${\sf coeff}_{ d^0} \big[ \Pi \big] = {\sf coeff}_{ d^0 } \big[ \Pi'
\big] = 0$ vanish, hence ${\sf D}_0 ( n ) = {\sf D}_0 ' ( n) = 0$
works. Most importantly, we shall establish in Section~\ref{Section-5}
that one may choose:
\[
\small
\aligned
{\sf D}_1(n)
=\cdots=
{\sf D}_n(n)
=
{\sf D}_1'(n)
=\cdots=
{\sf D}_n'(n)
=
{\sf D}_{n+1}'(n)
=
n^{4n^3}
2^{n^4}.
\endaligned
\]
Taking $n^{4n^3}2^{n^4}$ for granted, remind that with the above
choice of weight ${\bf a}^*$ (now
denoted ${\bf a}$), we ensure that:
\[
{\sf coeff}_{d^{n+1}}
\big[\Pi\big]
=
{\sf p}_{n+1,{\bf a}}
\geqslant
1
=:
{\sf G}_{n+1}.
\]
From the preceding two lemmas, we therefore deduce that: 
\[
\aligned
\big\vert
{\sf coeff}_{d^k}
\big[
\Pi
\big]
\big\vert
&
\leqslant
6\,n^{2n-1}
\cdot
(n+1)^{n^2+1}
\cdot
n^{3n^2}\,3^{n^3}
\cdot
n^{4n^3}
2^{n^4}
=:
6\,{\sf H}(n)
\ \ \ \ \ \ \ \ 
{\scriptstyle{(k\,=\,1\,\cdots\,n)}}
\\
\big\vert
{\sf coeff}_{d^k}
\big[
\Pi'
\big]
\big\vert
&
\leqslant
n^{2n-1}
\cdot
(n+1)^{n^2+1}
\cdot
n^{3n^2}\,3^{n^3}
\cdot
n^{4n^3}
2^{n^4}
=:{\sf H}(n)
\ \ \ \ \ \ \ \ 
{\scriptstyle{(k\,=\,1\,\cdots\,n\,+\,1)}}. 
\endaligned
\]
so that, coming back to the beginning of Section~\ref{Section-4}, 
we may choose
${\sf E}_0 = {\sf E}_0 ' = 0$ (since ${\sf D}_0 ( n ) = {\sf D}_0 '
(n) = 0$) and also explicitly in terms of $n$:
\[
\aligned
{\sf E}_1
=\cdots=
{\sf E}_n
&
=
6\,{\sf H}(n)
\\
{\sf E}_1'
=\cdots=
{\sf E}_n'
=
{\sf E}_{n+1}'
&
=
{\sf H}(n). 
\endaligned
\]
Coming back to the definition of $d_n^1, d_n^2$ given at the end of
Lemma~\ref{d-1-n} and just after, we may now majorate:
\[
\small
\aligned
d_n^1
&
\leqslant
1
+
\big(
n\,6\,{\sf H}(n)
+
{\textstyle{\frac{n+1}{2}}}
\big)
\big/
{\textstyle{\frac{1}{2}}}
=:
\widetilde{d}_n^1,
\\
d_n^2
&
\leqslant
1
+
n+2
+
2\,(n^2+2n)\,{\sf H}(n)
=:
\widetilde{d}_n^2.
\endaligned
\] 
Notice that $\widetilde{ d}_n^2 \geqslant \widetilde{ d}_n^1$ as soon
as $n \geqslant 3$. Finally, by comparing the growth of all terms in
${\sf H}(n)$ as $n \to \infty$, one sees that $2^{ n^4}$ dominates and
hence that the following inequality:
\[
\widetilde{d}_n^2
=
1+n+2
+
2\,
(n^2+2n)
\cdot
n^{2n-1}
\cdot
(n+1)^{n^2+1}
\cdot
n^{3n^2}\,3^{n^3}
\cdot
n^{4n^3}
2^{n^4}
\leqslant
2^{n^5},
\] 
holds for all large $n$. However, any symbolic computer shows that for
$n = 2, 3, 4$, one in fact has $\widetilde{ d}_2^2 > 2^{ 2^5}$,
$\widetilde{ d}_3^2 > 2^{ 3^5}$, $\widetilde{ d}_4^2 > 2^{ 4^5}$,
while $\widetilde{ d}_5^2 < 2^{ n^5}$ and $\widetilde{ d}_n^2 \ll 2^{
n^5}$ for $n = 6, 7, 8, 9$ so that $\widetilde{ d}_n^2 < 2^{ n^5}$
holds for any $n \geqslant 5$ by an elementary inspection of the
function $n \mapsto \widetilde{ d}_n^2$. Fortunately, the three left
cases $n = 2$, $n = 3$ and $n = 4$ of Theorem~\ref{main} are covered,
firstly for the classical surface case $n = 2$ by, say~\cite{ deg2000}
in which $\deg X \geqslant 21$ with $21 \ll 2^{ 2^5}$, and secondly
for $n = 3$ and $n = 4$ by our second Theorem~\ref{lowdim}, because
$2^{ 3^5} \gg 593$ and $2^{ 4^5} \gg 3203$.  So we conclude that if we
take for granted: 1) that one may take all the ${\sf D}_k (n)$ and all
the ${\sf D}_k' ( n)$ equal to $n^{ 4n^3} 2^{n^4}$, a technical and
crucial statement to which Section~\ref{Section-5} below is entirely
devoted; and 2) that Theorem~\ref{lowdim} is got, an effective
statement to which the two Sections~\ref{Section-6}
and~\ref{Section-7} below are devoted, then the proof of our main
Theorem~\ref{main} is to be considered as complete, and finally, the
neat uniform degree bound $\deg X \geqslant 2^{ n^5}$ works in all
dimensions $n \geqslant 2$.
\qed

\section{Estimations of the quantities 
${\sf D}_k (n)$ and ${\sf D}_k' (n)$}
\label{Section-5}

To complete our program, it now remains only to capture somewhat
effective upper bounds ${\sf D}_k ( n)$, $0
\leqslant k \leqslant n$ and ${\sf D}_k' ( n)$, $0 \leqslant k
\leqslant n+1$.

\begin{theorem}\label{D-k-n}
With $n \geqslant 2$, for any $l, i_1, \dots, i_n \in \mathbb{ N}$
with $l+ i_1 + \cdots + i_n = n^2$ and any $l, j_1, \dots, j_n \in
\mathbb{ N}$ with $1 + l+ j_1 + \cdots + j_n = n^2$, one has:
\[
0
=
{\sf coeff}_{d^0}
\big[h^lu_1^{i_1}\cdots u_n^{i_n}\big]
=
{\sf coeff}_{d^0}
\big[c_1h^lu_1^{j_1}\cdots u_n^{j_n}\big].
\]
Moreover and above all, for every $k = 1, \dots, n+1$,
the following uniform effective upper bound holds:
\[
\aligned
\big\vert{\sf coeff}_{d^k}
\big[h^lu_1^{i_1}\cdots u_n^{i_n}\big]
\big\vert
\leqslant
n^{4n^3}2^{n^4},
\\
\big\vert{\sf coeff}_{d^k}
\big[c_1h^lu_1^{j_1}\cdots u_n^{j_n}\big]
\big\vert
\leqslant
n^{4n^3}2^{n^4}.
\endaligned
\]
\end{theorem}

In other words, in the above notations, one may choose ${\sf D}_0 ( n)
= {\sf D}_0 ' (n) = 0$ and ${\sf D}_k (n) = {\sf
D}_k ' (n) = n^{4n^3}2^{n^4}$ for $k = 1, \dots, n+1$.

\subsection{Jacobi-Trudy determinants}
One key observation towards these estimations is that the reduction
process from one level to the lower level in Demailly's tower involves
Jacobi-Trudy determinants in the Chern classes of the lower level
in question.

\begin{definition}
At any level $\ell$ with $0 \leqslant \ell \leqslant n- 1$ and for any
$J$ with $0 \leqslant J \leqslant n + \ell ( n-1) = \dim X_\ell$, we
define the corresponding {\sl Jacobi-Trudy determinant}:
\begin{equation*}
\mathcal{C}_J^\ell
:=
\left\vert
\begin{array}{ccccc}
c_1^{[\ell]} & c_2^{[\ell]} & c_3^{[\ell]} & \cdots & c_J^{[\ell]}
\\
1 & c_1^{[\ell]} & c_2^{[\ell]} & \cdots & c_{J-1}^{[\ell]}
\\
0 & 1 & c_1^{[\ell]} & \cdots & c_{J-1}^{[\ell]}
\\
\cdot & \cdot & \cdot & \cdots & \cdot
\\
0 & 0 & 0 & \cdots & c_1^{[\ell]}
\end{array}
\right\vert,
\end{equation*}
where, again by convention, we set 
any $c_k^{ [ \ell]} := 0$ as soon as $k
\geqslant n+1$; by convention also, $\mathcal{ C }_J^\ell := 0$ 
is set to zero when $J > \dim X_\ell$ and when $J < 0$; lastly, we
set $\mathcal{ C}_0^\ell := 1$.
\end{definition}

Expanding the determinant $\mathcal{ C}_J^\ell$
along its first line, and expanding again
the obtained block-determinants, one easily convinces oneself 
of the induction formulae:
\begin{equation}
\label{recurrence-jacobi-trudy}
\mathcal{C}_J^\ell
=
c_1^{[\ell]}\,\mathcal{C}_{J-1}^\ell
-
c_2^{[\ell]}\,\mathcal{C}_{J-2}^\ell
+
c_3^{[\ell]}\,\mathcal{C}_{J-3}^\ell
-\cdots,
\end{equation}
the last term in this expansion being
either $(-1)^{ n - 1} \, c_n^{ [\ell]}
\, \mathcal{ C}_{ J-n}^\ell$ when $J \geqslant n$ or else $(-1)^{ J-1} \,
c_J^{ [ \ell]} \, \mathcal{ C}_0^\ell$ when $J < n$.

In the proof of Theorem~\ref{D-k-n}, the study of the monomials $u_1^{
i_1} \cdots u_n^{ i_n}$ will appear {\em a posteriori} to be exactly
the same as the study of the monomials $h^l u_1^{ i_1} \cdots u_n^{
i_n}$ and $c_1 h^l u_1^{ j_1} \cdots u_n^{ j_n}$.

Generally speaking, fixing $\ell$ with $1 \leqslant \ell \leqslant n$ and
exponents $i_1, \dots, i_\ell \in \mathbb{ N}$ satisfying $i_1 +
\cdots + i_\ell = n + \ell  (n-1)  = \dim X_\ell$, let us therefore
study the reduction, in term of the degree $d$ of $X$, of the specific
monomial $u_1^{ i_1} \cdots u_{ \ell - 1}^{ i_{ \ell - 1}} u_\ell^{
i_\ell}$. We write it as $\Omega_K^{ \ell - 1} u_\ell^{ i_\ell}$,
where $\Omega_K^{ \ell - 1} := u_1^{ i_1} \cdots u_{ \ell - 1}^{ i_{
\ell- 1}}$ is a $(K, K)$-form living on $X_{ \ell - 1}$ with $K +
i_\ell = n + \ell ( n-1)$.

If $i_\ell \leqslant n - 2$, then $\Omega_K^{ \ell - 1}$ vanishes form
degree-form reasons. If $i_\ell = n - 1$, then a fiber-integration
gives $\Omega_K^{ \ell - 1} \underline{ u_\ell^{ n-1}}_{\int} =
\Omega_K^{ \ell - 1} \cdot 1 = \Omega_K^{ \ell-1}\mathcal{
C}_0^{\ell-1}$.

\begin{lemma}
\label{K-i-l}
For any $\ell$ with $1 \leqslant \ell \leqslant n$, given any
$(K, K)$-form $\Omega_K^{ \ell-1}$ at level $\ell - 1$
and any integer $i_\ell$ with
$i_\ell \geqslant n-1$ and $i_\ell + K = \dim X_\ell$, the reduction
of $\Omega_K^{ \ell - 1} \, u_\ell^{ i_\ell}$ down to level $\ell-1$
precisely reads:
\[
\small
\aligned
\Omega_K^{\ell-1}u_\ell^{i_\ell}
&
=
(-1)^{i_\ell-n+1}\,
\Omega_K^{\ell-1}\,
\left\vert
\begin{array}{cccc}
c_1^{[\ell-1]} & c_2^{[\ell-1]} & \cdots & c_{i_\ell-n+1}^{[\ell-1]}
\\
1 & c_1^{[\ell-1]} & \cdots & c_{i_\ell-n}^{[\ell-1]}
\\
\cdot & \cdot & \cdots & \cdot
\\
0 & 0 & \cdots & c_1^{[\ell-1]}
\end{array}
\right\vert
\\
&
=
(-1)^{i_\ell-n+1}\,\Omega_K^{\ell-1}\,
\mathcal{C}_{i_\ell-n+1}^{\ell-1}.
\endaligned
\]
\end{lemma}

\proof
Assume first that $i_\ell = n$ and
use~\thetag{ \ref{u-ell-n}} to get:
\[
\aligned
\Omega_K^{\ell-1}\,u_\ell^n
&
=
-
\Omega_K^{\ell-1}\,c_1^{[\ell-1]}\,
\underline{u_\ell^{n-1}}_{\int}
-
\underline{\Omega_K^{\ell-1}c_2^{[\ell-1]}}_\circ
u_\ell^{n-2}
-\cdots-
\underline{\Omega_K^{\ell-1}c_n^{[\ell-1]}}_\circ
\\
&
=
-\Omega_K^{\ell-1}\,\mathcal{C}_1^{\ell-1}.
\endaligned
\]
Reasoning by induction, assume now that the lemma holds for all
$i_\ell'$ with $n \leqslant i_\ell' \leqslant i_\ell$ for some $i_\ell
\geqslant n$. Take an arbitrary $(L, L)$-form $\Omega_L^{ \ell-1}$ on
$X_{ \ell-1}$ with $L + i_\ell + 1 = \dim X_\ell$, multiply~\thetag{
\ref{u-ell-n}} by $\Omega_L^{ \ell-1} \, u_\ell^{ i_\ell + 1 - n}$ to
get:
\[
\small
\aligned
\Omega_L^{\ell-1}\,u_\ell^{i_\ell+1}
&
=
-\Omega_L^{\ell-1}
\Big(
c_1^{[\ell-1]}\,u_\ell^{i_\ell}
+
c_2^{[\ell-1]}\,u_\ell^{i_\ell-1}
+
c_3^{[\ell-1]}\,u_\ell^{i_\ell-2}
+\cdots
\Big)
\\
&
=
(-1)^{1+i_\ell-n+1}\,\Omega_L^{\ell-1}
\Big(
c_1^{[\ell-1]}\,\mathcal{C}_{i_\ell-n+1}^{\ell-1}
-
c_2^{[\ell-1]}\,\mathcal{C}_{i_\ell-n}^{\ell-1}
+
c_3^{[\ell-1]}\,\mathcal{C}_{i_\ell-n-1}^{\ell-1}
-\cdots
\Big)
\\
&
=
(-1)^{i_\ell+1-n+1}\,\Omega_L^{\ell-1}\,
\mathcal{C}_{i_\ell+1-n+1}^{\ell-1},
\endaligned
\]
thanks to~\thetag{ \ref{recurrence-jacobi-trudy}}, 
which gives the claimed reduction for the exponent 
$i_\ell+1$. 
\endproof

Applying this lemma to the monomial $u_1^{ i_1} \cdots u_\ell^{
i_\ell} u_{ \ell+1}^{ i_{ \ell+1}}$, we thus reduce it to
\[
u_1^{i_1}\cdots u_\ell^{i_\ell}u_{\ell+1}^{i_{\ell+1}}
=
(-1)^{i_{\ell+1}-n+1}\,
u_1^{i_1}\cdots u_\ell^{i_\ell}\,
\mathcal{C}_{i_{\ell+1}-n+1}^\ell
\,.
\]
To obtain effective estimations, we will need to further reduce such a
Jacobi-Trudy determinant $\mathcal{ C}_{ i_{ \ell +1} -n+1 }^\ell$
from level $\ell$ down to level $\ell -1$. A whole program begins. In
the application we have in mind, one should think that $\Omega_K^\ell
= (-1)^{ i_{\ell+1} -n+1}\, u_1^{i_1} \cdots u_\ell^{i_\ell}$ and that
$J = i_{ \ell+1} - n+1$.

\begin{lemma}
\label{dvp-expansion}
At an arbitrary level $\ell$ with $1 \leqslant \ell \leqslant n - 1$,
consider the Jacobi-Trudy determinant $\mathcal{
C}_J^\ell$ of an arbitrary size $J \times J$ with $1 \leqslant J
\leqslant \dim X_\ell$ and furthermore, let $\Omega_K^{ \ell}$ be any
$(K, K)$-form on $X_{ \ell}$ whose degree $K$ satisfies $K + J = \dim
X_\ell = n + \ell ( n-1)$. Then the reduction of
$\Omega_K^{\ell}\mathcal{C}_J^\ell$ down to level $\ell - 1$
relies upon the following formulae:
\[
\Omega_K^{\ell}\mathcal{C}_J^\ell
=
\Omega_K^{\ell}
\big[
\mathcal{C}_J^{\ell-1}
+
\mathcal{C}_0^\ell{\sf A}_J^\ell
+
\mathcal{C}_1^\ell{\sf A}_{J-1}^\ell
+\cdots+
\mathcal{C}_{J-1}^\ell{\sf A}_1^\ell
\big],
\]
in which, for any $k$ with $1 \leqslant k \leqslant J$, one has set:
\[
{\sf A}_k^\ell
:=
{\sf X}_1^\ell\mathcal{C}_{k-1}^{\ell-1}
-
{\sf X}_2^\ell\mathcal{C}_{k-2}^{\ell-1}
+\cdots+
(-1)^{k-1}{\sf X}_k^\ell\mathcal{C}_0^{\ell-1},
\]
where the ${\sf X}$-terms here gather all the terms after $c_j^{
[\ell -1]}$ in a convenient rewriting of~\thetag{
\ref{c-j-ell}} under
the following form:
\[
c_j^{[\ell]}
=
c_j^{[\ell-1]}
+\underbrace{
\lambda_{j,1}\,c_{j-1}^{[\ell-1]}u_\ell
+
\lambda_{j,2}\,c_{j-2}^{[\ell-1]}u_\ell^2
+\cdots+
\lambda_{j,j}\, u_\ell^j}_{\overset{\text{\rm def}}=\,{\sf X}_j^\ell},
\]
with the convention that ${\sf X}_j^\ell = 0$ 
for any $j \geqslant n+1$.
\end{lemma}

\proof
Naturally, we should expand the Jacobi-Trudy determinant in question
after inserting in it the relation (\ref{c-j-ell}). 
This is based on linear
algebra considerations and we shall drop $\Omega_K^{\ell}$ in the
computations.

More precisely, let us write down the determinant $\mathcal{
C}_J^\ell$ we have to expand:
\[
\small
\aligned
\mathcal{C}_J^\ell
=
\left\vert
\begin{array}{cccc}
c_1^{[\ell]} & c_2^{[\ell]} & \cdots & c_J^{[\ell]}
\\
1 & c_1^{[\ell]} & \cdots & c_{J-1}^{[\ell]}
\\
\vdots & \vdots & \ddots & \vdots
\\
0 & 0 & \cdots & c_1^{[\ell]}
\end{array}
\right\vert
=
\left\vert
\begin{array}{cccc}
{\sf X}_1^\ell + c_1^{[\ell-1]} 
& c_2^{[\ell]} & \cdots & c_J^{[\ell]}
\\
0+1\ \ \ \ & c_1^{[\ell]} & \cdots & c_{J-1}^{[\ell]}
\\
\vdots & \vdots & \ddots & \vdots
\\
0 & 0 & \cdots & c_1^{[\ell]}
\end{array}
\right\vert
\endaligned
\]
by emphasizing the induction on $\ell$ which represents its first
column naturally as the sum of two columns. As already devised, we
expand it by linearity, getting:
\[
\aligned
\mathcal{C}_J^\ell
=
\left\vert
\begin{array}{cccc}
{\sf X}_1^\ell & c_2^{[\ell]} & \cdots & c_J^{[\ell]}
\\
0 & c_1^{[\ell]} & \cdots & c_{J-1}^{[\ell]}
\\
\vdots & \vdots & \ddots & \vdots
\\
0 & 0 & \cdots & c_1^{[\ell]}
\end{array}
\right\vert
+
\left\vert
\begin{array}{cccc}
c_1^{[\ell-1]} & c_2^{[\ell]} & \cdots & c_J^{[\ell]}
\\
1 & c_1^{[\ell]} & \cdots & c_{J-1}^{[\ell]}
\\
\vdots & \vdots & \ddots & \vdots
\\
0 & 0 & \cdots & c_1^{[\ell]}
\end{array}
\right\vert,
\endaligned
\]
and just afterwards immediately, we expand the first determinant along
its first column, while at the same time, in the second column of the
second determinant, we again emphasize the induction on $\ell$:
\[
\aligned
\mathcal{C}_J^\ell
=
{\sf X}_1^\ell\cdot\mathcal{C}_{J-1}^\ell
+
\left\vert
\begin{array}{ccccc}
c_1^{[\ell-1]} & {\sf X}_2^\ell + c_2^{[\ell-1]} 
& c_3^{[\ell]} 
& \cdots & c_J^{[\ell]}
\\
1 & {\sf X}_1^\ell + c_1^{[\ell-1]} & c_2^{[\ell]}
& \cdots & c_{J-1}^{[\ell]}
\\
0 & 0 + 1 \ \ \ & c_1^{[\ell]} 
& \cdots & c_{J-2}^{[\ell]}
\\
\vdots & \vdots & \vdots & \ddots & \vdots
\\
0 & 0 & 0 & \cdots & c_1^{[\ell]}
\end{array}
\right\vert.
\endaligned
\]
Next, we similarly expand by linearity the obtained determinant,
realizing again that its second column is a sum of two columns:
\[
\footnotesize
\aligned
\mathcal{C}_J^\ell
=
{\sf X}_1^\ell\cdot\mathcal{C}_{J-1}^\ell
&
+
\left\vert
\begin{array}{ccccc}
c_1^{[\ell-1]} & {\sf X}_2^\ell
& c_3^{[\ell]} 
& \cdots & c_J^{[\ell]}
\\
1 & {\sf X}_1^{\ell} & c_2^{[\ell]}
& \cdots & c_{J-1}^{[\ell]}
\\
0 & 0 & c_1^{[\ell]} 
& \cdots & c_{J-2}^{[\ell]}
\\
\vdots & \vdots & \vdots & \ddots & \vdots
\\
0 & 0 & 0 & \cdots & c_1^{[\ell]}
\end{array}
\right\vert
+ 
\left\vert
\begin{array}{ccccc}
c_1^{[\ell-1]} & c_2^{[\ell-1]} 
& c_3^{[\ell]} 
& \cdots & c_J^{[\ell]}
\\
1 & c_1^{[\ell-1]} & c_2^{[\ell]}
& \cdots & c_{J-1}^{[\ell]}
\\
0 & 1 & c_1^{[\ell]} 
& \cdots & c_{J-2}^{[\ell]}
\\
\vdots & \vdots & \vdots & \ddots & \vdots
\\
0 & 0 & 0 & \cdots & c_1^{[\ell]}
\end{array}
\right\vert,
\endaligned
\]
and evidently again, we must expand the first obtained 
determinant along its second column, getting:
\[
\small
\aligned
\mathcal{C}_J^\ell
=
{\sf X}_1^\ell\cdot\mathcal{C}_{J-1}^\ell
&
-
{\sf X}_2^\ell
\cdot
\left\vert
\begin{array}{cccc}
1 & c_2^{[\ell]} & \cdots & c_{J-1}^{[\ell]}
\\
0 & c_1^{[\ell]} & \cdots & c_{J-2}^{[\ell]}
\\
\vdots & \vdots & \ddots & \vdots 
\\
0 & 0 & \cdots & c_1^{[\ell]}
\end{array}
\right\vert
+
{\sf X}_1^\ell\cdot
\left\vert
\begin{array}{cccc}
c_1^{[\ell-1]} & c_3^{[\ell]} & \cdots & c_J^{[\ell]}
\\
0 & c_1^{[\ell]} & \cdots & c_{J-2}^{[\ell]}
\\
\vdots & \vdots & \ddots & \vdots 
\\
0 & 0 & \cdots & c_1^{[\ell]}
\end{array}
\right\vert
\\
& \ \ \ \ \ \ \ \
+ 
\left\vert
\begin{array}{cccccc}
c_1^{[\ell-1]} & c_2^{[\ell-1]} 
& {\sf X}_3^\ell + c_3^{[\ell-1]} 
& c_4^{[\ell]} & \cdots & c_J^{[\ell]}
\\
1 & c_1^{[\ell-1]} & {\sf X}_2^\ell + c_2^{[\ell-1]} 
& c_3^{[\ell]} & \cdots & c_{J-1}^{[\ell]}
\\
0 & 1 & {\sf X}_1^\ell + c_1^{[\ell-1]}
& c_2^{[\ell]} & \cdots & c_{J-2}^{[\ell]}
\\
0 & 0 & 0 + 1 \ \ \ \ & c_1^{[\ell]} 
& \cdots & c_{J-3}^{[\ell]}
\\
\vdots & \vdots & \vdots & \vdots & \ddots & \vdots
\\
0 & 0 & 0 & 0 & \cdots & c_1^{[\ell]} 
\end{array}
\right\vert,
\endaligned
\]
and we are supposed to iterate once again the same two processes:
\[
\small
\aligned
\mathcal{C}_J^\ell
&
=
{\sf X}_1^\ell\cdot\mathcal{C}_{J-1}^\ell
-
{\sf X}_2^\ell\cdot 1
\cdot
\mathcal{C}_{J-2}^\ell
+
{\sf X}_1^\ell\cdot
\mathcal{C}_1^{\ell-1}\cdot
\mathcal{C}_{J-2}^\ell
\\
&\qquad 
+
{\sf X}_3^\ell\cdot
\left\vert
\begin{array}{cc}
1 & c_1^{[\ell-1]}
\\
0 & 1
\end{array}
\right\vert
\cdot
\left\vert
\begin{array}{ccc}
c_1^{[\ell]} & \cdots & c_{J-3}^{[\ell]}
\\
\vdots & \ddots & \vdots
\\
0 & \cdots & c_1^{[\ell]}
\end{array}
\right\vert
\\ &\qquad
-
{\sf X}_2^\ell
\cdot\left\vert
\begin{array}{cc}
c_1^{[\ell-1]} & c_2^{[\ell-1]}
\\
0 & 1
\end{array}
\right\vert
\cdot
\left\vert
\begin{array}{ccc}
c_1^{[\ell]} & \cdots & c_{J-3}^{[\ell]}
\\
\vdots & \ddots & \vdots
\\
0 & \cdots & c_1^{[\ell]}
\end{array}
\right\vert
\\ & \qquad
+
{\sf X}_1^\ell\cdot
\left\vert
\begin{array}{cc}
c_1^{[\ell-1]} & c_2^{[\ell-1]}
\\
1 & c_1^{[\ell-1]}
\end{array}
\right\vert
\cdot
\left\vert
\begin{array}{ccc}
c_1^{[\ell]} & \cdots & c_{J-3}^{[\ell]}
\\
\vdots & \ddots & \vdots
\\
0 & \cdots & c_1^{[\ell]}
\end{array}
\right\vert
\\ &\qquad
+
\left\vert
\begin{array}{ccccccc}
c_1^{[\ell-1]} & c_2^{[\ell-1]} & c_3^{[\ell-1]} 
& {\sf X}_4^\ell + c_4^{[\ell-1]} 
& c_5^{[\ell]} & \cdots & c_J^{[\ell]}
\\
1 & c_1^{[\ell-1]} & c_2^{[\ell-1]} 
& {\sf X}_3^\ell + c_3^{[\ell-1]} 
& c_4^{[\ell]} & \cdots & c_{J-1}^{[\ell]}
\\
0 & 1 & c_1^{[\ell-1]} & {\sf X}_2^\ell + c_2^{[\ell-1]}
& c_3^{[\ell]} & \cdots & c_{J-2}^{[\ell]}
\\
0 & 0 & 1 & {\sf X}_1^\ell + c_1^{[\ell-1]} 
& c_2^{[\ell]} 
& \cdots & c_{J-3}^{[\ell]}
\\
0 & 0 & 0 & 0 + 1 \ \ \ \
& c_1^{[\ell]} 
& \cdots & c_{J-4}^{[\ell]}
\\
\vdots & \vdots & \vdots & \vdots & \vdots & \ddots & \vdots
\\
0 & 0 & 0 & 0 & 0 & \cdots & c_1^{[\ell]} 
\end{array}
\right\vert.
\endaligned
\]
At this point where things start to become clearer, we make the
following general observation. Consider the determinant
that one obtains after a finite number of steps:
\[
\aligned
\left\vert
\begin{array}{cccccccc}
c_1^{[\ell-1]} & c_2^{[\ell-1]} & \cdots & c_{k-1}^{[\ell-1]}
& {\sf X}_k^\ell + c_k^{[\ell-1]} & c_{k+1}^{[\ell]} 
& \cdots & c_J^{[\ell]}
\\
1 & c_1^{[\ell-1]} & \cdots & c_{k-2}^{[\ell-1]} 
& {\sf X}_{k-1}^\ell + c_{k-1}^{[\ell-1]} 
& c_k^{[\ell]} & \cdots & c_{J-1}^{[\ell]}
\\
\vdots & \vdots & \ddots & \vdots & \vdots & \vdots 
& \ddots & \vdots
\\
0 & 0 & \cdots & c_1^{[\ell-1]}
& {\sf X}_2^\ell + c_2^{[\ell-1]} & c_3^{[\ell]} 
& \cdots & c_{J-k+2}^{[\ell]}
\\
0 & 0 & \cdots & 1 
& {\sf X}_1^\ell + c_1^{[\ell-1]} & c_2^{[\ell]} 
& \cdots & c_{J-k+1}^{[\ell]}
\\
0 & 0 & \cdots & 0 & 0 + 1 \ \ \ \ \
& c_1^{[\ell]} & \cdots & c_{J-k}^{[\ell]}
\\
\vdots & \vdots & \ddots & \vdots & \vdots & \vdots 
& \ddots & \vdots
\\
0 & 0 & \cdots & 0 & 0 & 0 & \cdots & c_1^{[\ell]}
\end{array}
\right\vert,
\endaligned
\]
where the central-looking column is the $k$-th one, for some $k$ with
$1 \leqslant k \leqslant J$. Write this determinant as a sum of two
determinants by linearity, and expand the first obtained determinant,
let us call it $\Delta_k$, along its $k$-th column in which are
present all the ${\sf X}_k^\ell$'s. We thus get that the first
determinant is equal to:
\[
\small
\aligned
\Delta_k
&
:=
(-1)^{k+1}\,{\sf X}_k^\ell\cdot
\left\vert
\begin{array}{ccc}
1 & \cdots & c_{k-2}^{[\ell-1]}
\\
\vdots & \ddots & \vdots
\\
0 & \cdots & 1
\end{array}
\right\vert
\cdot 
\mathcal{C}_{J-k}^\ell
\\ &\qquad
+
(-1)^{k+2}\,{\sf X}_{k-1}^\ell
\cdot
\left\vert
\begin{array}{cccc}
c_1^{[\ell-1]} & * & \cdots & *
\\
0 & 1 & \cdots & c_{k-3}^{[\ell-1]}
\\
\vdots & \vdots & \ddots & \vdots 
\\
0 & 0 & \cdots & 1
\end{array}
\right\vert
\cdot
\mathcal{C}_{J-k}^\ell
\\ &\qquad
+
(-1)^{k+3}\,{\sf X}_{k-2}^\ell\cdot
\left\vert
\begin{array}{ccccc}
c_1^{[\ell-1]} & c_2^{[\ell-1]} & * & \cdots & *
\\ 
1 & c_1^{[\ell-1]} & * & \cdots & *
\\
0 & 0 & 1 & \cdots & c_{k-4}^{[\ell-1]}
\\
\vdots & \vdots & \vdots & \ddots & \vdots
\\
0 & 0 & 0 & \cdots & 1
\end{array}
\right\vert
\cdot
\mathcal{C}_{J-k}^\ell
\\ &\qquad
+\cdots+
(-1)^{k+k}\,{\sf X}_1^\ell\cdot
\left\vert
\begin{array}{ccc}
c_1^{[\ell-1]} & \cdots & c_{k-1}^{[\ell-1]}
\\
\vdots & \ddots & \vdots
\\
0 & \cdots & c_1^{[\ell-1]}
\end{array}
\right\vert
\cdot
\mathcal{C}_{J-k}^\ell,
\endaligned
\]
while the second determinant is of the same kind as the one we started
with, except that the ${\sf X}$'s are now located in the $(k+1)$-th
column. Thus after mild simplifications, 
what we called the first determinant
equals:
\[
\small
\aligned
\Delta_k
&
=
(-1)^{k+1}\,{\sf X}_k^\ell\cdot 1 
\cdot\mathcal{C}_{J-k}^\ell
+
(-1)^{k+2}\,{\sf X}_{k-1}^\ell\cdot\mathcal{ C}_1^{\ell-1}
\cdot\mathcal{C}_{J-k}^\ell
+
\\
&\ \ \ \ \
+
(-1)^{k+3}\,{\sf X}_{k-2}^\ell\cdot
\mathcal{C}_2^{\ell-1}
\cdot\mathcal{C}_{J-k}^\ell
+\cdots+
{\sf X}_1^\ell\cdot
\mathcal{C}_{k-1}^{\ell-1}
\cdot\mathcal{C}_{J-k}^\ell
\\
&
=
{\sf A}_k^\ell\mathcal{C}_{J-k}^\ell.
\endaligned
\]
In conclusion, the initial Jacobi-Trudy determinant $\mathcal{
C}_J^\ell$ we started with now equals:
\[
\mathcal{C}_J^\ell
=
\Delta_1+\cdots+\Delta_k+\cdots+\Delta_J
+
\left\vert
\begin{array}{ccc}
c_1^{[\ell-1]} & \cdots & c_J^{[\ell-1]}
\\
\vdots & \ddots & \vdots 
\\
0 & \cdots & c_1^{[\ell-1]}
\end{array}
\right\vert,
\]
where the last written determinant, equal to $\mathcal{ C}_J^{\ell
-1}$ and living at the $(\ell-1)$-th level, is the remainder
determinant after all ${\sf X}$-terms are removed by
expansion. Summing the $\Delta_k = {\sf A}_k^\ell \, \mathcal{ C}_{
J-k}^\ell$, we obtain the formula announced in the lemma.
\endproof

As $J$ varies, the formulae given by this lemma:
\[
\mathcal{C}_J^\ell
=
\mathcal{C}_J^{\ell-1}
+
\mathcal{C}_0^\ell{\sf A}_J^\ell
+
\mathcal{C}_1^\ell{\sf A}_{J-1}^\ell
+\cdots+
\mathcal{C}_{J-1}^\ell{\sf A}_1^\ell,
\]
are still imperfect, for their right-hand sides still involve
Jacobi-Trudy determinants at the level $\ell$.
So necessarily, we must perform further reductions. 

\begin{lemma}
\label{C-A}
For any $J$ with $0 \leqslant J \leqslant \dim X_\ell$ and any $\ell$
with $1 \leqslant \ell \leqslant n$, one has:
\[
\mathcal{C}_J^\ell
=
\sum_{j=0}^J\,
\mathcal{C}_{J-j}^{\ell-1}
\bigg(
\sum_{\nu=1}^j\,
\sum_{k_1+\cdots+k_\nu=j\atop
k_1,\dots,k_\nu\geqslant 1}\,
{\sf A}_{k_1}^\ell\cdots{\sf A}_{k_\nu}^\ell
\bigg),
\]
with the convention that for $j = 0$, the empty sum in parentheses
equals $1$.
\end{lemma}

\proof
First, for $J = 0$, recall that by convention $\mathcal{ C }_0^\ell =
\mathcal{ C }_0^{ \ell-1} = 1$. Next, for $J = 1$, we start from the
formula of the preceding lemma and we perform an evident computation:
\[
\mathcal{C}_1^\ell
=
\mathcal{C}_1^{\ell-1}
+
\mathcal{C}_0^\ell{\sf A}_1^\ell
=
\mathcal{C}_1^{\ell-1}\Sigma_0^\ell({\sf A})
+
\mathcal{C}_0^{\ell-1}\Sigma_1^\ell({\sf A}),
\]
if, generally speaking, we denote for convenient abbreviation:
\begin{equation}
\label{sigma-j-ell}
\Sigma_j^\ell({\sf A})
:=
\sum_{\nu=1}^j\,
\sum_{k_1+\cdots+k_\nu=j\atop
k_1,\dots,k_\nu\geqslant 1}\,
{\sf A}_{k_1}^\ell\cdots{\sf A}_{k_\nu}^\ell,
\end{equation}
with of course $\Sigma_0^\ell ({\sf A}) = 1$. These $\Sigma_j^\ell (
{\sf A})$ satisfy useful induction formulae:
\begin{equation}
\label{sigma-induction}
\small
\aligned
\Sigma_j^\ell({\sf A})
&
=
{\sf A}_j^\ell
+
\sum_{\nu=2}^j\,
\sum_{k_1+k_2+\cdots+k_\nu=j\atop
k_1,k_2,\dots,k_\nu\geqslant 1}\,
{\sf A}_{k_1}^\ell{\sf A}_{k_2}^\ell\cdots
{\sf A}_{k_\nu}^\ell
\\
&
=
{\sf A}_j^\ell
+
\sum_{\nu=2}^j\,
\bigg(
{\sf A}_1^\ell\!\!
\sum_{k_2+\cdots+k_\nu=j-1\atop
k_2,\dots,k_\nu\geqslant 1}\,
{\sf A}_{k_2}^\ell\cdots{\sf A}_{k_\nu}^\ell
+
{\sf A}_2^\ell\!\!
\sum_{k_2,\dots,k_\nu=j-2\atop
k_2,\dots,k_\nu\geqslant 1}\,
{\sf A}_{k_2}^\ell\cdots{\sf A}_{k_\nu}^\ell
+
\\
&
\ \ \ \ \ \ \ \ \ \ \ \ \ \ \ \ \ \ \ \ \ \ \ \ \ \ 
+\cdots+
{\sf A}_{j-1}^\ell\!\!
\sum_{k_2+\cdots+k_\nu=1\atop
k_1,\dots,k_\nu\geqslant 1}\,
{\sf A}_{k_2}^\ell\cdots{\sf A}_{k_\nu}^\ell
\bigg)
\\
&
=
{\sf A}_j^\ell
+
{\sf A}_1^\ell\,\sum_{\nu=2}^{j-1}\,
\sum_{k_2+\cdots+k_\nu=j-1\atop
k_2,\dots,k_\nu\geqslant 1}\!\!
{\sf A}_{k_2}^\ell\cdots{\sf A}_{k_\nu}^\ell
+
{\sf A}_2^\ell\,\sum_{\nu=2}^{j-2}\,
\sum_{k_2+\cdots+k_\nu=j-2\atop
k_2,\dots,k_\nu\geqslant 1}\!\!
{\sf A}_{k_2}^\ell\cdots{\sf A}_{k_\nu}^\ell
+
\\
&
\ \ \ \ \ \ \ \ \ \ \ \ \ \ \ \ \ \ \ \ \ \ \ \ \ \
+\cdots+
{\sf A}_{j-1}^\ell\,\sum_{\nu=2}^2\,
\sum_{k_2=1\atop
k_2\geqslant 1}\,{\sf A}_{k_2}^\ell
\\
&
=
{\sf A}_j^\ell\Sigma_0^\ell({\sf A})
+
{\sf A}_1^\ell\,\Sigma_{j-1}^\ell({\sf A})
+
{\sf A}_2^\ell\Sigma_{j-2}^\ell({\sf A})
+\cdots+
{\sf A}_{j-1}^\ell\Sigma_1^\ell({\sf A}). 
\endaligned
\end{equation}
Next, for $J = 2$, starting again from the 
known (imperfect) formula
and using what has just been seen:
\[
\aligned
\mathcal{C}_2^\ell
&
=
\mathcal{C}_2^{\ell-1}
+
\mathcal{C}_0^\ell{\sf A}_2^\ell
+
\mathcal{C}_1^\ell{\sf A}_1^\ell
\\
&
=
\mathcal{C}_2^{\ell-1}
+
\mathcal{C}_0^{\ell-1}{\sf A}_2^\ell
+
\big[
\mathcal{C}_1^{\ell-1}\Sigma_0^\ell({\sf A})
+
\mathcal{C}_0^{\ell-1}\Sigma_1^\ell({\sf A})
\big]{\sf A}_1^\ell
\\
&
=
\mathcal{C}_2^{\ell-1}\Sigma_0^\ell({\sf A})
+
\mathcal{C}_1^{\ell-1}
\big[
\Sigma_0^\ell({\sf A}){\sf A}_1^\ell
\big]
+
\mathcal{C}_0^{\ell-1}
\big[
\Sigma_1^\ell({\sf A}){\sf A}_1^\ell
+
{\sf A}_2^\ell
\big]
\\
&
=
\mathcal{C}_2^{\ell-1}\Sigma_0^\ell({\sf A})
+
\mathcal{C}_1^{\ell-1}\Sigma_1^\ell({\sf A})
+
\mathcal{C}_0^{\ell-1}\Sigma_2^\ell({\sf A}).
\endaligned
\]
Suppose now by induction that we have already proved that:
\[
\mathcal{C}_{J'}^\ell
=
\mathcal{C}_{J'}^{\ell-1}\Sigma_0^\ell({\sf A})
+
\mathcal{C}_{J'-1}^{\ell-1}\Sigma_1^\ell({\sf A})
+
\mathcal{C}_{J'-2}^{\ell-1}\Sigma_2^\ell({\sf A})
+\cdots+
\mathcal{C}_0^{\ell-1}\Sigma_J^\ell({\sf A}),
\]
for all $J'$ with $0 \leqslant J' 
\leqslant J$, for some $J \geqslant 2$. 
Then we apply the known general (imperfect) 
formula with $J$
replaced by $J+1$ in it, and afterwards, we use the induction
hypothesis, which gives:
\[
\small
\aligned
\mathcal{C}_{J+1}^\ell
&
=
\mathcal{C}_{J+1}^{\ell-1}
+
\mathcal{C}_0^\ell{\sf A}_{J+1}^\ell
+
\mathcal{C}_1^\ell{\sf A}_J^\ell
+\cdots+
\mathcal{C}_{J-1}^\ell{\sf A}_2^\ell
+
\mathcal{C}_J^\ell{\sf A}_1^\ell
\\
&
=
\mathcal{C}_{J+1}^{\ell-1}\Sigma_0^\ell({\sf A})
+
\\
&\ \ \ \ \
+
\big[\mathcal{C}_0^{\ell-1}\Sigma_0^\ell({\sf A})\big]
{\sf A}_{J+1}^\ell
+
\\
&\ \ \ \ \
+
\big[
\mathcal{C}_1^{\ell-1}\Sigma_0^\ell({\sf A})
+
\mathcal{C}_0^{\ell-1}\Sigma_1^\ell({\sf A})
\big]{\sf A}_J^\ell
+
\\
&\ \ \ \ \
+\cdots\cdots\cdots\cdots\cdots\cdots\cdots\cdots\cdots
\cdots\cdots\cdots\cdots
+
\\
&\ \ \ \ \
+
\big[
\mathcal{C}_{J-1}^{\ell-1}\Sigma_0^\ell({\sf A})
+
\mathcal{C}_{J-2}^{\ell-1}\Sigma_1^\ell({\sf A})
+
\mathcal{C}_{J-3}^{\ell-1}\Sigma_2^\ell({\sf A})
+\cdots+
\mathcal{C}_0^{\ell-1}\Sigma_{J-1}^\ell({\sf A})
\big]{\sf A}_2^\ell
+
\\
&\ \ \ \ \
+
\big[
\mathcal{C}_J^{\ell-1}\Sigma_0^\ell({\sf A})
+
\mathcal{C}_{J-1}^{\ell-1}\Sigma_1^\ell({\sf A})
+
\mathcal{C}_{J-2}^{\ell-1}\Sigma_2^\ell({\sf A})
+\cdots+
\mathcal{C}_1^{\ell-1}\Sigma_{J-1}^\ell({\sf A})
+
\mathcal{C}_0^{\ell-1}\Sigma_J^\ell({\sf A})
\big]{\sf A}_1^\ell.
\endaligned
\] 
A necessary and natural reorganization then gives:
\[
\small
\aligned
\mathcal{C}_{J+1}^\ell
&
=
\mathcal{C}_{J+1}^{\ell-1}
\big[
\Sigma_0({\sf A})
\big]
+
\\
&\ \ \ \ \
+
\mathcal{C}_J^{\ell-1}
\big[
\Sigma_0^\ell({\sf A}){\sf A}_1^\ell
\big]
+
\\
&\ \ \ \ \
+
\mathcal{C}_{J-1}^{\ell-1}
\big[
\Sigma_1^\ell({\sf A}){\sf A}_1^\ell
+
\Sigma_0^\ell({\sf A}){\sf A}_2^\ell
\big]
+
\\
&\ \ \ \ \
+
\mathcal{C}_{J-2}^{\ell-1}
\big[
\Sigma_2^\ell({\sf A}){\sf A}_1^\ell
+
\Sigma_1^\ell({\sf A}){\sf A}_2^\ell
+
\Sigma_0^\ell({\sf A}){\sf A}_3^\ell
\big]
+
\\
&\ \ \ \ \ 
+
\cdots\cdots\cdots\cdots\cdots\cdots\cdots\cdots\cdots
\cdots\cdots\cdots\cdots\cdots\cdots\cdots
+
\\
&\ \ \ \ \
+
\mathcal{C}_0^{\ell-1}
\big[
\Sigma_J^\ell({\sf A}){\sf A}_1^\ell
+
\Sigma_{J-1}^\ell({\sf A}){\sf A}_2^\ell
+
\Sigma_{J-2}^\ell({\sf A}){\sf A}_3^\ell
+\cdots+
\Sigma_0^\ell({\sf A}){\sf A}_{J+1}^\ell
\big]
\\
&
=
\mathcal{C}_{J+1}^{\ell-1}\Sigma_0^\ell({\sf A})
+
\mathcal{C}_J^{\ell-1}\Sigma_1^\ell({\sf A})
+
\mathcal{C}_{J-1}^{\ell-1}\Sigma_2^\ell({\sf A})
+
\mathcal{C}_{J-2}^{\ell-1}\Sigma_3^\ell({\sf A})
+\cdots+
\mathcal{C}_0^{\ell-1}\Sigma_{J+1}^\ell({\sf A}),
\endaligned
\]
where at the end, one applies the formulae~\thetag{
\ref{sigma-induction}} just seen. Notice 
{\em passim} that the number of terms in
$\Sigma_j^\ell ( {\sf A})$ is equal to $2^{ j-1}$ for all $j \geqslant 1$.
\endproof

\subsection{Upper reduction operator}
The reduction process, after several elimination computations
involving (\ref{c-j-ell}) and (\ref{u-ell-n}) and at the end
(\ref{c-d}), transforms a general monomial of the form $h^l u_1^{ i_1}
\cdots u_n^{ i_n}$ with $l + i_1 + \cdots + i_n=n^2$ into a polynomial
$\mathcal{ R} \big( h^l u_1^{ i_1} \cdots u_n^{ i_n} \big)$ of degree
$\leqslant n + 1$ in $d$, where the symbol ``$\mathcal{ R}$'' stands
for ``{\sl reduction}''.

From now on, complete explicit algebraic computations will not be
conducted anymore, and instead, to tame their complexity,
{\em inequalities} will be dealt with.

For our majoration purposes, we now introduce an important {\sl upper
reduction operator} $\mathcal{ R}^+$ which by definition, at each
computational step of the reduction process, while going down in the
Demailly's tower, always replaces any incoming sign ``$-$'' by a sign
``$+$''. Accordingly, 
for any two monomials $h^l u_1^{i_1} \cdots u_n^{ i_n}$ and $h^{
l'} u_1^{ i_1'} \cdots u_n^{ i_n'}$, we shall say that:
\[
\mathcal{R}^+
\big(h^lu_1^{i_1}\cdots u_n^{i_n}\big)
\leqslant
\mathcal{R}^+
\big(h^{l'}u_1^{i_1'}\cdots u_n^{i_n'}\big),
\]
and write more briefly:
\[
h^lu_1^{i_1}\cdots u_n^{i_n}
\leqslant_{ \mathcal{ R}^+}
h^{l'}u_1^{i_1'}\cdots u_n^{i_n'},
\]
if the corresponding two (upper) reduced polynomials $\sum_{ k = 0}^{
n+1} \, {\sf p}_k \cdot d^k$ and $\sum_{ k = 0}^{ n+1} \, {\sf p}_k'
\cdot d^k$ have all their coefficients satisfying:
\[
\big(
0
\leqslant\!\big)\,
{\sf p}_k 
\leqslant
{\sf p}_k'
\ \ \ \ \ 
\text{\rm for every}
\ \
k=0,1,\dots,n+1.
\]
Then obviously the absolute values of the coefficients of the
reduction are smaller than the (nonnegative)
coefficients of the upper reduction:
\[
\big\vert
{\sf coeff}_{d^k}
\big[h^lu_1^{i_1}\cdots u_n^{i_n}\big]
\big\vert
\leqslant
{\sf coeff}_{d^k}
\big[
\mathcal{R}^+
\big(h^lu_1^{i_1}\cdots u_n^{i_n}\big)
\big]
\,.
\]

To obtain the desired bound
$n^{4n^3}2^{n^4}$ we need to handle the Jacobi-Trudy
determinants seen above. The following lemma will be useful.

\begin{lemma}
For any $\lambda_1, \lambda_2, \dots, \lambda_n$ with
$n = \lambda_1 + 2 \lambda_2 + \cdots + n \lambda_n$, one
has:
\[
c_1^{\lambda_1}
\big(\mathcal{C}_2^0\big)^{\lambda_2}
\cdots
\big(\mathcal{C}_n^0\big)^{\lambda_n}
\leqslant_{ \mathcal{ R}^+}
\mathcal{C}_n^0
\,.
\]
\end{lemma}

\proof
An inspection of the determinant $\mathcal{ C}_n^0$ shows that one may
view all the pure monomials $c_1^{\lambda_1 }$, $\big( \mathcal{
C}_2^0 \big)^{ \lambda_2}$, \dots, $\big( \mathcal{ C}_k^0 \big)^{
\lambda_k}$ as diagonal subblocks of the corresponding sizes lying
inside $\mathcal{ C}_n^0$. Since the operator $\mathcal{ R}^+$ expands
the determinants and replaces all the minus signs by plus signs, it is
then clear that there are more terms in the right-hand side than there
are in the left-hand side, which completes the proof.
\endproof

The same arguments yield determinantal inequalities at
any level.

\begin{lemma}
\label{C-C-majorations}
For any two $J_1$, $J_2$ with $0 \leqslant J_1, J_2 \leqslant \dim
X_\ell$ satisfying in addition $J_1 + J_2 \leqslant \dim X_\ell$, and
for any $j_1$ with $0 \leqslant j_1 \leqslant n$ satisfying in
addition $j_1 + J_2 \leqslant \dim X_\ell$, one has the two
majorations:
\[
\footnotesize
\aligned
\mathcal{R}^+
\big(
\Omega_K^\ell\cdot
\mathcal{C}_{J_1}^\ell
\cdot 
\mathcal{C}_{J_2}^\ell
\big)
\leqslant
\mathcal{R}^+
\big(
\Omega_K^\ell\cdot
\mathcal{C}_{J_1+J_2}^\ell
\big)
\ \ \ \ \ 
\text{and}
\ \ \ \ \
\mathcal{R}^+
\big(
\Omega_K^\ell\cdot
c_{j_1}^{[\ell]}
\cdot
\mathcal{C}_{J_2}^\ell
\big)
\leqslant
\mathcal{R}^+
\big(
\Omega_K^\ell\cdot
\mathcal{C}_{j_1+J_2}^\ell
\big),
\endaligned
\]
where $\Omega_K^\ell$ is any $(K, K)$-form living on $X_\ell$
completing to $\dim X_\ell$ the degree, namely with $K + J_1 + J_2$
and with $K + j_1 + J_2$ both equal to $\dim X_\ell$.
\end{lemma}

If $J_1 + J_2 < 0$ or if $J_1 + J_2 > \dim X_\ell$, and if $j_1 + J_2
< 0$ or if $j_1 + J_2 > \dim X_\ell$, the two sides vanish in both
inequalities, which hence hold without restriction.

\begin{lemma}
These coefficients $\lambda_{ j, j - k} =
\frac{(n-k)!}{(j-k)!\,(n-j)!} - \frac{(n-k)!}{(j-k-1)!(n-j+1)!}$
appearing in~\thetag{ \ref{c-j-ell}} 
satisfy the uniform
majoration:
\[
\big\vert\lambda_{j,j-k} 
\big\vert \leqslant 2^n
=:
\lambda
\]
expressed in terms of the dimension $n$ only.
\end{lemma}

\proof
Indeed, the absolute value of the difference $\lambda_{ j, j - k} =
\lambda_{ j, j-k} ' - \lambda_{ j, j-k} ''$ of two nonnegative integers is
less than the largest one, and we majorate any appearing binomial
coefficient $\frac{ n'!}{ i' ! \, (n' - i')!}$ or $\frac{ n''!}{ i''! \,
(n'' - i'')!}$ with $n' \leqslant n$ and $n'' \leqslant n$ plainly by
$2^n$.
\endproof

In the subsequent majorations, while applying the upper majoration operator
$\mathcal{ R}^+$, we shall also replace any incoming $\lambda_{ j,
j-k}$ by this majorant $\lambda = 2^n$. As a result, we define a
generalized upper majoration operator ``$\mathcal{ R}_\lambda^+$''
which both replaces any minus sign by a plus sign and any $\lambda_{
j, j- k}$ by $\lambda = 2^n$.

Also, when executing inequalities, we shall sometimes not write the
left differential form $\Omega_K^{\ell}$ which completes to $\dim
X_\ell$ the total degree of the considered forms, for one knows well now
that forms to be reduced always have degree equal to the dimension of
the level on which they sit, unless
they vanish identically for degree-form reasons. 

\begin{lemma}
\label{A-R-C}
For all $k = 1, 2, \dots, n$, one has the $\mathcal{ R}_\lambda^+$
majorations:
\[
{\sf A}_k^\ell
\,\,\leqslant_{\mathcal{R}_\lambda^+}\,\,
k\lambda
\big(
\mathcal{C}_{k-1}^{\ell-1}u_\ell
+
\mathcal{C}_{k-2}^{\ell-1}u_\ell^2
+\cdots+
u_\ell^k
\big).
\]
\end{lemma}

\proof
Starting from the evident majoration of
the ${\sf X}_j^\ell$ that were defined at the
end of Lemma~\ref{dvp-expansion}: 
\[
{\sf X}_j^\ell
\,\,\leqslant_{\mathcal{R}_\lambda^+}\,\,
\lambda
\big(
c_{j-1}^{[\ell-1]}u_\ell
+
c_{j-2}^{[\ell-1]}u_\ell^2
+\cdots+
u_\ell^j
\big),
\]
we may perform majorations of an arbitrary ${\sf A}_k^\ell$
also defined there:
\[
\small
\aligned
{\sf A}_k^\ell
&
=
{\sf X}_1^\ell\mathcal{C}_{k-1}^{\ell-1}
-
{\sf X}_2^\ell\mathcal{C}_{k-2}^{\ell-1}
+
{\sf X}_3^\ell\mathcal{C}_{k-3}^{\ell-1}
-\cdots+
(-1)^{k-1}{\sf X}_k^\ell\mathcal{C}_0^{\ell-1}
\\
&
\leqslant_{\mathcal{R}_\lambda^+}\,\,
\big[\lambda u_\ell\big]\mathcal{C}_{k-1}^{\ell-1}
+
\big[\lambda\big(c_1^{[\ell-1]}u_\ell+u_\ell^2\big)\big]
\mathcal{C}_{k-2}^{\ell-1}
+
\big[\lambda\big(c_2^{[\ell-1]}u_\ell
+c_1^{[\ell-1]}u_\ell^2+u_\ell^3\big)\big]
\mathcal{C}_{k-3}^{\ell-1}
+
\\
&\ \ \ \ \ \ \ \ \ \ \ \
+\cdots+
\big[\lambda\big(c_{k-1}^{[\ell-1]}u_\ell
+\cdots+
c_1^{[\ell-1]}u_\ell^{k-1}+u_\ell^k\big)\big]
\mathcal{C}_0^{\ell-1}
\\
&
=
\lambda
\Big(
u_\ell\big[
\mathcal{C}_{k-1}^{\ell-1}
+
c_1^{[\ell-1]}\mathcal{C}_{k-2}^{\ell-1}
+
c_2^{[\ell-1]}\mathcal{C}_{k-3}^{\ell-1}
+\cdots+
c_{k-1}^{[\ell-1]}\mathcal{C}_0^{\ell-1}
\big]+
\\
&\ \ \ \ \
+
u_\ell^2\big[
\ \ \ \ \ \ \ \ \ \ \ \ \ \ \ \ \ \ \ \ \ \
\mathcal{C}_{k-2}^{\ell-1}
+
c_1^{[\ell-1]}\mathcal{C}_{k-3}^{\ell-1}
+\cdots+
c_{k-2}^{[\ell-1]}\mathcal{C}_0^{\ell-1}
\big]+
\\
&\ \ \ \ \ 
+
u_\ell^3\big[
\ \ \ \ \ \ \ \ \ \ \ \ \ \ \ \ \ \ \ \ \ \
\ \ \ \ \ \ \ \ \ \ \ \ \ \ \ \ \ \ \ \ \ \
\mathcal{C}_{k-3}^{\ell-1}
+\cdots+
c_{k-3}^{[\ell-1]}\mathcal{C}_0^{\ell-1}
\big]+
\\
&\ \ \ \ \
+\!
\cdots\cdots\cdots\cdots\cdots\cdots\cdots\cdots\cdots
\cdots\cdots\cdots\cdots\cdots\cdots\cdots\cdots\cdot\!
+
\\
&\ \ \ \ \
+
u_\ell^k\big[
\ \ \ \ \ \ \ \ \ \ \ \ \ \ \ \ \ \ \ \ \ \
\ \ \ \ \ \ \ \ \ \ \ \ \ \ \ \ \ \ \ \ \ \
\ \ \ \ \ \ \ \ \ \ \ \ \ \ \ \ \ \ \ \ \ \
\ \ \ \ \ \ \ \ \ \
\mathcal{C}_0^{\ell-1}
\big]
\Big).
\endaligned
\]
Now, we use the majoration of an arbitrary product of a Jacobi-Trudy
determinant by a Chern class that was provided in advance by
Lemma~\ref{C-C-majorations} to obtain:
\[
\aligned
{\sf A}_k^\ell
&
\,\,\leqslant_{\mathcal{R}_\lambda^+}\,\,
\lambda
\Big(
u_\ell\big[k\cdot\mathcal{C}_{k-1}^{\ell-1}\big]
+
u_\ell^2\big[(k-1)\cdot\mathcal{C}_{k-2}^{\ell-1}\big]
+\cdots+
u_\ell^k\big[\mathcal{C}_0^{\ell-1}\big]
\Big)
\\
&
\,\,\leqslant_{\mathcal{R}_\lambda^+}\,\,
k\lambda
\big(
\mathcal{C}_{k-1}^{\ell-1}u_\ell
+
\mathcal{C}_{k-2}^{\ell-1}u_\ell^2
+\cdots+
u_\ell^k
\big),
\endaligned
\]
as was to be proved.
\endproof

We now have to majorate conveniently the ${\sf A}$-polynomials
$\Sigma_j^\ell ({\sf A})$ defined by~\thetag{ \ref{sigma-j-ell}}
in terms of Jacobi-Trudy determinants living at the inferior level
$\ell - 1$, and in terms of $u_\ell$, too. For this purpose, let us
define what will play the role of a convenient majorant:
\[
\Theta_k^\ell
:=
\mathcal{C}_{k-1}^{\ell-1}u_\ell
+
\mathcal{C}_{k-2}^{\ell-1}u_\ell^2
+\cdots+
\mathcal{C}_1^{\ell-1}u_\ell^{k-1}
+
u_\ell^k,
\]
and let us keep in mind that the lemma just proved provided the
majorations ${\sf A}_k^\ell \,\, \leqslant_{ \mathcal{ R}_\lambda^+}
\, k\lambda \, \Theta_k^\ell$. To majorate products of ${\sf
A}_k^\ell$'s, we majorate products of $\Theta_k^\ell$'s.

\begin{lemma}
\label{Theta-k}
For any $k_1, k_2, \dots, k_\nu$ with $k_1, k_2, \dots, k_\nu
\geqslant 1$ whose sum $k_1 + k_2 + \cdots + k_\nu = j$ equals $j$,
one has the majoration:
\[
\Theta_{k_1}^\ell\Theta_{k_2}^\ell\cdots
\Theta_{k_\nu}^\ell
\,\,\leqslant_{\mathcal{R}_\lambda^+}\,\,
k_1k_2\cdots k_\nu\,
\Theta_{k_1+k_2+\cdots+k_\nu}^\ell.
\]
\end{lemma}

\proof
In greater length, the considered product writes:
\[
\big(\mathcal{C}_{k_1-1}^{\ell-1}u_\ell
+\cdots+
u_\ell^{k_1}\big)
\big(\mathcal{C}_{k_2-1}^{\ell-1}u_\ell
+\cdots+
u_\ell^{k_2}\big)
\cdots
\big(\mathcal{C}_{k_\nu-1}^{\ell-1}u_\ell
+\cdots+
u_\ell^{k_\nu}\big),
\]
and the total number of terms, after expansion, is hence clearly
$\leqslant k_1 k_2 \cdots k_\nu$. Using the already known inequality
$\mathcal{ C}_{ J_1}^{ \ell-1} \cdot \mathcal{ C}_{ J_2}^{ \ell - 1}
\,\, \leqslant_{\mathcal{ R}_\lambda^+} \, \, \mathcal{ C}_{ J_1 +
J_2}^{ \ell - 1}$, we may majorate as follows any monomial appearing
after expansion:
\[
\mathcal{C}_{k_1'}^{\ell-1}\mathcal{C}_{k_2'}^{\ell-1}
\cdots
\mathcal{C}_{k_\nu'}^{\ell-1}\,
u_\ell^{k''}
\,\,\leqslant_{\mathcal{R}_\lambda^+}\,\,
\mathcal{C}_{k_1'+\cdots+k_\nu'}^{\ell-1}\,
u_\ell^{k''},
\]
where $k_1' + k_2 ' + \cdots + k_\nu' + k'' = k_1 + k_2 + \cdots +
k_\nu = j$ of course, which completes the proof.
\endproof

At last, we can state and prove the main useful
majoration proposition which will
enable us to achieve the proof of
Theorem~\ref{D-k-n}, {\em cf.} the program
launched just before Lemma~5.2.

\begin{proposition}\label{J-T}
At any level $\ell$ with $1 \leqslant \ell \leqslant n-1$, consider
the Jacobi-Trudy determinant $\mathcal{ C}_J^\ell$ of an arbitrary
size $J \times J$ with $1 \leqslant J \leqslant \dim X_\ell$ and
furthermore, let $\Omega_K^{ \ell}$ be any $(K, K)$-form on $X_{ \ell
}$ the degree $K$ of which satisfies $K + J = \dim X_\ell = n + \ell (
n-1)$. Then the upper reduction $\mathcal{ R}_\lambda^+ ( \bullet)$ of
$\Omega_K^{ \ell} \mathcal{ C}_J^\ell$ in which any incoming
$\lambda_{ j, j- k}$ is replaced by $\lambda = 2^n \geqslant \big\vert
\lambda_{ j, j - k} \big\vert$ enjoys the following majoration in the
right-hand side of which, notably, all the appearing Jacobi-Trudy
determinants live at level $\ell - 1${\rm :}
\[
\
\small
\aligned
\Omega_K^{\ell}\mathcal{C}_J^\ell
\,\,\leqslant_{\mathcal{R}_\lambda^+}\,\,
J\cdot 2^J\cdot J^{2J}\cdot 2^{nJ}\cdot
\Omega_K^{\ell}\,
\Big[
\mathcal{C}_J^{\ell-1}
+
\mathcal{C}_{J-1}^{\ell-1}u_\ell
+\cdots+
\mathcal{C}_1^{\ell-1}u_\ell^{J-1}
+
u_\ell^J
\Big]\,
\endaligned
\,.
\]
\end{proposition}

\proof
Recall that
\[
\mathcal{C}_J^\ell
=
\sum_{j=1}^J\,\mathcal{C}_{J-j}^\ell\,
\Sigma_j^\ell({\sf A})
=
\sum_{j=0}^J\,
\mathcal{C}_{J-j}^{\ell-1}
\bigg(
\sum_{\nu=1}^j\,
\sum_{k_1+\cdots+k_\nu=j\atop
k_1,\dots,k_\nu\geqslant 1}\,
{\sf A}_{k_1}^\ell\cdots{\sf A}_{k_\nu}^\ell
\bigg).
\]
Using the last two lemmas, we deduce that for any $k_1, \dots, k_\nu
\geqslant 1$ with $k_1 + \cdots + k_\nu$ the sum of which $k_1 +
\cdots + k_\nu$ equals $j$, we have the majoration:
\[
\aligned
{\sf A}_{k_1}^\ell\cdots
{\sf A}_{k_\nu}^\ell
&
\,\,\leqslant_{\mathcal{R}_\lambda^+}\,\,
k_1\cdots k_\nu\,
\lambda^\nu\,\Theta_{k_1}^\ell\cdots\Theta_{k_\nu}^\ell
\ \ \ \ \ \ \ \ \ \ \ \ \,
\explain{Lemma~\ref{A-R-C}}
\\
&
\,\,\leqslant_{\mathcal{R}_\lambda^+}\,\,
\big(k_1\cdots k_\nu\big)^2\,\lambda^\nu\,
\Theta_{k_1+\cdots+k_\nu}^\ell
\ \ \ \ \ \ \ \ \ \
\explain{Lemma~\ref{Theta-k}}
\\
&
\,\,\leqslant_{\mathcal{R}_\lambda^+}\,\,
j^{2j}\,\lambda^j\,\Theta_j^\ell. 
\endaligned
\]
Since there are $2^{ j-1} \leqslant 2^j$ terms in the sum $\sum_{ \nu
= 1}^j \, \sum_{ k_1 + \cdots + k_\nu = j \atop k_1, \dots, k_\nu
\geqslant 1}$, we receive the useful majoration:
\[
\aligned
\Sigma_j^\ell({\sf A})
&
=
\sum_{\nu=1}^j\,\sum_{k_1+\cdots+k_\nu=j\atop
k_1,\dots,k_\nu\geqslant 1}\,
{\sf A}_{k_1}^\ell\cdots{\sf A}_{k_\nu}^\ell
\\
&
\leqslant_{\mathcal{R}_\lambda^+}\,\,
2^j\,j^{2j}\,\lambda^j\,\Theta_j^\ell.
\endaligned
\]
In conclusion, 
starting from Lemma~\ref{C-A} and using
Lemma~\ref{C-C-majorations}, 
we may lastly perform the following (not optimal) majoration:
\[
\aligned
\mathcal{C}_J^\ell
&
=
\mathcal{C}_J^{\ell-1}
+
\mathcal{C}_{J-1}^{\ell-1}\Sigma_1^\ell({\sf A})
+
\mathcal{C}_{J-2}^{\ell-1}\Sigma_2^\ell({\sf A})
+\cdots+
\mathcal{C}_{J-j}^{\ell-1}\Sigma_j^\ell({\sf A})
+\cdots+
\mathcal{C}_0^{\ell-1}\Sigma_J^\ell({\sf A})
\\
&
\leqslant_{\mathcal{R}_\lambda^+}\,\,
\mathcal{C}_J^{\ell-1}
+
\mathcal{C}_{J-1}^{\ell-1}\,
2^11^{2}\lambda^1\big[u_\ell\big]
+
\mathcal{C}_{J-2}^{\ell-1}2^22^4\lambda^2
\big[\mathcal{C}_1^{\ell-1}u_\ell+u_\ell^2\big]
\\
&\qquad 
+\cdots+
\mathcal{C}_{J-j}^{\ell-1}\,
2^jj^{2j}\lambda^j
\big[
\mathcal{C}_{j-1}^{\ell-1}u_\ell+\cdots+u_\ell^j
\big]
\\
&\qquad +\cdots
+
\mathcal{C}_0^{\ell-1}\,2^JJ^{2J}\lambda^J
\big[
\mathcal{C}_{J-1}^{\ell-1}u_\ell
+\cdots+
u_\ell^J
\big]
\\
&
\leqslant_{\mathcal{R}_\lambda^+}\,\,
2^11^2\lambda^1\big[
\mathcal{C}_J^{\ell-1}
+
\mathcal{C}_{J-1}^{\ell-1}u_\ell
\big]
+
2^22^4\lambda^2
\big[\mathcal{C}_{J-1}^{\ell-1}u_\ell
+
\mathcal{C}_{J-2}^{\ell-1}u_\ell^2\big]
\\
&\qquad +\cdots
+
2^jj^{2j}\lambda^j
\big[
\mathcal{C}_{J-1}^{\ell-1}u_\ell+\cdots+
\mathcal{C}_{J-j}^{\ell-1}u_\ell^j
\big]
\\
&\qquad +\cdots
+
2^JJ^{2J}\lambda^J
\big[
\mathcal{C}_{J-1}^{\ell-1}u_\ell+\cdots+u_\ell^J
\big]
\\
&
\leqslant_{\mathcal{R}_\lambda^+}\,\,
J\cdot 2^J\cdot J^{2J}\cdot \lambda^J\,
\Big[
\mathcal{C}_J^{\ell-1}
+
\mathcal{C}_{J-1}^{\ell-1}u_\ell
+
\mathcal{C}_{J-2}^{\ell-1}u_\ell^2
+\cdots+
\mathcal{C}_1^{\ell-1}u_\ell^{J-1}
+
u_\ell^J
\Big],
\endaligned
\]
where the introduction of supplementary terms in the brackets aims at
producing a uniform right-hand side. 
\endproof

\subsection{Proof of Theorem~\ref{D-k-n}} The vanishing of the
$d^0$-coefficient comes from the fact that after reduction to the
ground level $\ell = 0$, one gets a sum of homogeneous monomials of
the form $h^l c_1^{ \lambda_1} c_2^{ \lambda_2} \cdots c_n^{
\lambda_n}$ with $l + \lambda_1 + 2 \lambda_2 + \cdots + n \lambda_n =
n$, and then after expressing each $c_k$ in terms of $d$
through~\thetag{ \ref{c-d}}, one always has the power $h^n = d$ of $h$
in factor.

Notice that the integer $J$ of the Proposition~\ref{J-T} will always be
less than or equal to $\dim X_{ n-1} = n^2 - n+1$. To simplify the
computations and to receive at the end as simple majorants as
possible, we shall apply the following elementary majoration, 
using $J \leqslant n^2 - n+1$:
\[
\aligned
J\cdot 2^J\cdot J^{2J}\cdot 2^{nJ}
&
=
2^{(n+1)J}\cdot J^{2J+1}
\\
&
\leqslant
2^{n^3+1}\,(n^2-n+1)^{2n^2-2n+3}
\\
&
\leqslant
2^{n^3}\big(n^2\big)^{2n^2},
\endaligned
\]
because $2\, ( n^2 - n+1)^{ 2n^2 - 2n +3} \leqslant 2 \, (n^2)^{ 2n^2
- 2n + 3} \leqslant (n^2)^{ 2n^2}$ for any $n \geqslant 2$ (an
assumption of Theorem~\ref{D-k-n}). Let us temporarily denote this
bound by:
\[
{\sf N}
:= 
2^{n^3}\,n^{4n^2}. 
\]
As expected, we can now perform a uniform upper majoration of an
arbitrary monomial $u_1^{ i_1} \cdots u_n^{ i_n}$ with $i_1 + \cdots +
i_n = n^2$ down to level $\ell = 0$ as follows:
\[
\small
\aligned
u_1^{i_1}\cdots 
&
u_{n-1}^{i_{n-1}}u_n^{i_n}
=
u_1^{i_1}\cdots u_{n-1}^{i_{n-1}}\mathcal{C}_{i_n-n+1}^{n-1}
\\
&
\leqslant_{\mathcal{R}_\lambda^+}\,\,
{\sf N}\cdot
u_1^{i_1}\cdots u_{n-2}^{i_{n-2}}
u_{n-1}^{i_{n-1}}
\big[
\mathcal{C}_{i_n-n+1}^{n-2}
+
\mathcal{C}_{i_n-n}^{n-2}u_{n-1}
\\ 
& \qquad\qquad
+\cdots+
\mathcal{C}_1^{n-2}u_{n-1}^{i_n-n}
+
u_{n-1}^{i_n-n+1}
\big]
\ \ \ \ \ \ \ \ \ \ \ 
\explain{Proposition~\ref{J-T}}
\\
&
\leqslant_{\mathcal{R}_\lambda^+}\,\,
{\sf N}\cdot
u_1^{i_1}\cdots u_{n-2}^{i_{n-2}}
\big[
\underline{{\mathcal{C}_{i_n-n+1}^{n-2}u_{n-1}^{i_{n-1}}} 
+\cdots}_\circ
\\ 
& 
\qquad\qquad
+\mathcal{C}_{i_{n-1}+i_n-2n+2}^{n-2}\,
\underline{u_{n-1}^{n-1}}_{\int}
+\cdots+
u_{n-1}^{i_{n-1}+i_n-n+1}
\big]
\\
&
\leqslant_{\mathcal{R}_\lambda^+}\,\,
{\sf N}\cdot
u_1^{i_1}\cdots u_{n-2}^{i_{n-2}}
\big[
\mathcal{C}_{i_{n-1}+i_n-2n+2}^{n-2}\,
+
\mathcal{C}_{i_{n-1}+i_n-2n+1}^{n-2}\,u_{n-1}^n
\\ 
& \qquad\qquad+\cdots+
u_{n-1}^{i_{n-1}+i_n-n+1}
\big]
\\
&
\leqslant_{\mathcal{R}_\lambda^+}\,\,
{\sf N}\cdot
u_1^{i_1}\cdots u_{n-2}^{i_{n-2}}
\big[
\mathcal{C}_{i_{n-1}+i_n-2n+2}^{n-2}\,
+
\mathcal{C}_{i_{n-1}+i_n-2n+1}^{n-2}\,
\mathcal{C}_1^{n-2}
\\ 
&\qquad\qquad
+\cdots+
\mathcal{C}_{i_{n-1}+i_n-2n+2}^{n-2}
\big]
\ \ \ \ \ \ \ \ \ \ \ 
\explain{Lemma~\ref{K-i-l}}
\\
&
\leqslant_{\mathcal{R}_\lambda^+}\,\,
{\sf N}\,n^2\cdot
u_1^{i_1}\cdots u_{n-2}^{i_{n-2}}\,
\mathcal{C}_{i_{n-1}+i_n-2n+2}^{n-2}
\ \ \ \ \ \ \ \ \ \ \ 
\explain{Lemma~\ref{C-C-majorations}}
\\
&
\leqslant_{\mathcal{R}_\lambda^+}\,\,
\big({\sf N}\,n^2\big)^2\cdot
u_1^{i_1}\cdots u_{n-3}^{i_{n-3}}\,
\mathcal{C}_{i_{n-2}+i_{n-1}+i_n-3n+3}^{n-3}
\ \ \ \ \ \ \ \ \ \ \ 
\explain{induction}
\\
&
\leqslant_{\mathcal{R}_\lambda^+}\,\,
\big({\sf N}\,n^2\big)^3\cdot
u_1^{i_1}\cdots u_{n-4}^{i_{n-4}}\,
\mathcal{C}_{i_{n-3}+i_{n-2}+i_{n-1}+i_n-4n+4}^{n-4}
\ \ \ \ \ \ \ \ \ \ \ 
\explain{induction}.
\endaligned
\]
In the third line, we exhibit the general case where $i_{ n-1}$ can be
$< n-1$, we underline the terms vanishing for degree-form reasons and
we point out the fiber-integration of $u_{ n-1}^{ n-1}$; when $i_{
n-1} \geqslant n-1$, the underlined terms are absent. In the sixth
line, we majorate plainly by $n^2$ the number of terms inside the
brackets. (Recall that here by convention again, $\mathcal{ C}_J^\ell
= 0$ if either $J < 0$ or $J > \dim X_\ell$, so that some of the
written $\mathcal{ C}_J^\ell$ might well vanish, depending
on $i_1, \dots, i_n$.) 
A now clear induction down to level $\ell = 1$ therefore yields:
\[
\small
\aligned
u_1^{i_1}\cdots u_{n-1}^{i_{n-1}}u_n^{i_n}
&
\leqslant_{\mathcal{R}_\lambda^+}\,\,
\big({\sf N}\,n^2\big)^{n-2}\cdot
u_1^{i_1}\,
\mathcal{C}_{i_2+\cdots+i_n-(n-1)n+n-1}^1
\\
&
\leqslant_{\mathcal{R}_\lambda^+}\,\,
\big({\sf N}\,n^2\big)^{n-2}\cdot{\sf N}\cdot
\big[
\underline{\mathcal{C}_{2n-1-i_1}^0
+\cdots}_\circ+
\\
&
\ \ \ \ \ \ \ \ \ \ \ \
+
\mathcal{C}_n^0
\underline{u_1^{n-1}}_{\int}
+\cdots+
u_1^{2n-1}
\big]
\\
&
\leqslant_{\mathcal{R}_\lambda^+}\,\,
\big({\sf N}\,n^2\big)^{n-1}\,
\mathcal{C}_n^0.
\endaligned
\]
It only remains to majorate $\mathcal{ C}_n^0$. This last reduction
using only \thetag{ \ref{c-d}} without any $\lambda_{ j, j-k}$, let
us denote by $\mathcal{ R}_d^+$ the upper reduction operator
restricted to level $\ell = 0$.

\begin{lemma}
The $n \times n$ Jacobi-Trudy determinant $\mathcal{ C}_0^n$ enjoys
the majoration:
\[
\mathcal{C}_n^0
\,\,\leqslant_{\mathcal{R}_d^+}\,\,
2^{n^2+2n}\,n!\,n^n\,
\big[d^{n+1}+d^n+\cdots+d\big].
\]
\end{lemma}

\proof
The number of monomials in the universal $n \times n$ determinant
$\vert a_i^j \vert$ is $\leqslant n!$ (and is $< n!$ when some $a_i^j$
are zero). Hence:
\[
\mathcal{C}_n^0
\,\,\leqslant_{\mathcal{R}_d^+}\,\,
n!\,
\max_{\lambda_1+2\lambda_2+\cdots+n\lambda_n=n}\,
c_1^{\lambda_1}c_2^{\lambda_2}\cdots c_n^{\lambda_n}.
\]
The general binomial coefficient $\binom{n+2}{k}$ which appears
in~\thetag{ \ref{c-d}} is less than or equal to $2^{ n+2}$, so that:
\[
c_j
\,\,\leqslant_{\mathcal{R}_d^+}\,\,
2^{n+2}\,h^j\,
\big[
d^j+\cdots+d+1
\big].
\]
We majorate as follows the products of these basic polynomials in $d$:
\[
\big[d^{j_1}+\cdots+d+1\big]
\big[d^{j_2}+\cdots+d+1\big]
\,\,\leqslant_{\mathcal{R}_d^+}\,\,
j_1j_2
\big[d^{j_1+j_2}+\cdots+d+1\big],
\]
and we therefore deduce a majorant for the general homogeneous degree $n$
monomial in the ground Chern classes:
\[
\aligned
c_1^{\lambda_1}c_2^{\lambda_2}
\cdots c_n^{\lambda_n}
&
\,\,\leqslant_{\mathcal{R}_d^+}\,\,
\big(2^{n+2}\big)^{\lambda_1+\lambda_2+\cdots+\lambda_n}\,
1^{\lambda_1}2^{\lambda_2}\cdots
n^{\lambda_n}\,h^{\lambda_1+2\lambda_2+\cdots+n\lambda_n}
\\
&\qquad\qquad
\cdot\big[d^{\lambda_1+2\lambda_2+\cdots+n\lambda_n}+\cdots+d+1\big]
\\
&
\,\,\leqslant_{\mathcal{R}_d^+}\,\,
\big(2^{n+2}\big)^n\,
n^{\lambda_1+\lambda_2+\cdots+\lambda_n}\,h^n\,
\big[d^n+\cdots+d+1\big]
\\
&
\,\,\leqslant_{\mathcal{R}_d^+}\,\,
2^{n^2+2n}\,n^n\,d\,
\big[d^n+\cdots+d+1\big]
\endaligned
\]
which completes the proof.
\endproof

Applying this lemma to the last obtained inequality: 
\[
u_1^{i_1}\cdots u_n^{i_n}
\,\,\leqslant_{\mathcal{R}_\lambda^+}\,\,
(Nn^2)^{n-1}\,2^{n^2+2n}\,n!\,n^n\cdot
\big[d^{n+1}+d^n+\cdots+1],
\]
we then obtain the announced bound $n^{4n^3}2^{n^4}$ as follows:
\[
\aligned
\big\vert
{\sf coeff}_{d^k} 
\big[
u_1^{i_1}\cdots u_n^{i_n}\big]
\big\vert
&
\leqslant
\big(2^{n^3}\,n^{4n^2}\,n^2\big)^{n-1}\,
2^{n^2+2n}\,n!\,n^n
\\
&
\leqslant
2^{n^4-n^3+n^2+2n}\,
n^{4n^3-4n^2+2n-2}\,
n^n\,n^n
\\
&
\leqslant
n^{4n^3}2^{n^4}.
\endaligned
\] 
By an inspection of the final inequalities which enabled
us to descend from the top of Demailly's tower to its ground level,
one easily convinces oneself that the monomials $h^l u_1^{ i_1} \cdots
u_n^{ i_n}$ and $c_1 h^l u_1^{ j_1} \cdots u_n^{ j_n}$ satisfy
exactly the same upper bound reduction:
\[
\aligned
& h^lu_1^{i_1}\cdots u_n^{i_n}
\,\,\leqslant_{\mathcal{R}_\lambda^+}\,\,
\big(N\,n^2\big)^{n-1}\,\mathcal{C}_n^0
\quad\text{\rm and}
\\
&
c_1h^lu_1^{j_1}\cdots u_n^{j_n}
\,\,\leqslant_{\mathcal{R}_\lambda^+}\,\,
\big(N\,n^2\big)^{n-1}\,\mathcal{C}_n^0,
\endaligned
\]
since the forms $h^l$ and $c_1 h^l$ do intervene only at the very end
of the process. This completes the proof of Theorem~\ref{D-k-n}. At
the same time, the proof of Theorem~\ref{main} can
be considered as complete, as soon as
we take for granted Theorem~\ref{lowdim}, as
was already explained at the end of Section~\ref{Section-4}. 
\qed

\section{Effective bounds in dimensions 2, 3 and 4 
\\
through the invariant theory approach}
\label{Section-6}

The goal of this section is to obtain the effective bound $\deg X^4
\geqslant 3203$ of Theorem~\ref{lowdim} in dimension $n = 4$ which
insures strong algebraic degeneracy of entire curves inside a generic
projective four-fold $X^4 \subset \mathbb{ P}^5$. As was said in the
Introduction, our reasonings will be based on a complete knowledge
(\cite{ mer2008b}) of the full algebra $\bigoplus_{m \geqslant 0}
E_{4, m} T_{X^4, x_0}^*$ of germs of invariant $4$-jet differentials
at a point $x_0 \in X^4$, which, unfortunately, is still unavailable
at present times for jets of order $k \geqslant n$ in the higher
dimensions $n = 5, 6, 7, \dots$ (remind that by Theorem~1.1 in~\cite{
rou2006b} and by Theorem~1  in~\cite{ div2008a}, $H^0 \big( X^n, \,
E_{ k, m} T_{X^n}^* \big) = 0$ whenever $k \leqslant n - 1$). The so
obtained bound $\deg X^4 \geqslant 3203$ happens to be sharper than
the one $\deg X^4 \geqslant 6527$ that one would obtain using the
intersection product~\thetag{ \ref{morse-intersection-delta}}. For
completeness and in parallel, we also recall what happens in the lower
dimensions $2$ (\cite{ dem1997, deg2000}) and $3$ (\cite{ rou2006a,
rou2006b}).

\subsection{Algebras of bi-invariant $k$-jet differentials} 
Let $(x_1, \dots, x_n)$ be local coordinates centered at some point
$x_0 \in X$ and let $f = (f_1, \dots, f_n) \colon (\mathbb{ C}, 0) \to
(X, x_0)$ be a germ of holomorphic curve. For each fixed $l$-th jet
level ($1 \leqslant l \leqslant k$) over $x_0$, the constant matrices
$v = (v_i^j)_{ 1\leqslant i \leqslant n}^{ 1 \leqslant j \leqslant n}$
in ${\rm GL}_n( \mathbb{ C})$ act in a natural way on the $n$ jet
coordinates $(f_1^{ (l)}, \dots, f_n^{ (l)})$ simply by:
\[
v\cdot
f^{(l)} 
:= 
\Big(
{\textstyle{\sum_j}}v_i^jf_j^{(l)}
\Big)_{1\leqslant i\leqslant n}. 
\]
In order to know what is the precise
decomposition of ${\rm Gr}^\bullet ( E_{ k, m} T_X^* )$ as a direct
sum of Schur bundles $\Gamma^{ (\ell_1, \dots, \ell_n)} T_X^*$, the
classical representation theory of ${\rm GL}_n ( \mathbb{ C})$ tells
us that one should look at jet polynomials $Q (f', f'', \dots, f^{
(k)})$ that are not only invariant under reparametrization in the
sense of Definition~\ref{invariant-jets}, but also invariant under the
action of the full unipotent subroup ${\rm U}_n ( \mathbb{ C}) \subset
{\rm GL}_n ( \mathbb{ C})$ consisting of matrices with $1$ on the
diagonal, $0$ above the diagonal, and arbitrary complex number below
the diagonal; background information may be found in~\cite{ dem1997,
rou2006a, rou2006b, mer2008b}.  Accordingly, one may define the
algebra of {\sl bi-invariant} $k$-jet polynomials in dimension $n$:
\[
{\sf UE}_k^n
:=
\Big(
\bigoplus_{m\geqslant 0}\,
E_{k,m}T_{X^n,x_0}^*
\Big)^{{\rm U}_n(\mathbb{C})}.
\]
This algebra does not depend on the base point $x_0 \in X$. 
We shall employ the abbreviations $\Delta_{ i_1, i_2}^{ (\alpha) ,
(\beta)} := f_{ i_1}^{ (\alpha)} f_{ i_2}^{ (\beta)} - f_{ i_2}^{
(\alpha)} f_{ i_1}^{ (\beta)}$ for $2\times 2$ determinants, and
similarly $\Delta_{ i_1, i_2, i_3}^{ (\alpha), (\beta), (\gamma)}$ for
the analogous $3 \times 3$ determinants. The upper indices of all the
appearing 16 bi-invariants $f_1'$, $\Lambda^3$, $\Lambda^5$,
$\Lambda^7$, $D^6$, $D^8$, $N^{ 10}$, $W^{ 10}$, $M^8$, $E^{10}$, $L^{
12}$, $Q^{ 14}$, $R^{ 15}$, $U^{ 17}$, $V^{ 19}$ and $X^{ 21}$ below just
denote their weighted degree $m$.

\begin{theorem}
The following three algebraic descriptions hold. 

\begin{itemize}

\smallskip\item
\cite{dem1997} In dimension $2$, one has: ${\sf UE}^2_2 = \mathbb{ C}
\big[ f_1',\, \Lambda^3\big]$, where $\Lambda^3 := \Delta_{1,2}^{',
''} = f_1' f_2'' - f_2'f_1''$ is the two-dimensional Wronskian.

\smallskip\item
\cite{rou2006a} In dimension $3$, one has:
\[
{\sf UE}_3^3
=
\mathbb{C}\big[
f_1',\,\Lambda^3,\,\Lambda^5,\,D^6
\big],
\]
where $\Lambda^5 := \Delta_{ 1, 2}^{ ', '''} f_1' - 3\, \Delta_{ 1,
2}^{ ', ''} f_1''$ and where $D^6 := \Delta_{ 1, 2, 3}^{ ', '',
'''}$ is the three-dimensional Wronskian.

\item
\cite{mer2008b} In dimension $4$, one has:
\[
\aligned
{\sf UE}_4^4
=
&
\mathbb{C}
\big[
f_1',\,\Lambda^3,\,\Lambda^5,\,\Lambda^7,\,D^6,\,D^8,\,N^{10},\,
W^{10},\,M^8,\,E^{10},\,L^{12},\,
\\
&
\ \ \ \ \
Q^{14},\,R^{15},\,U^{17},\,V^{19},\,X^{21}
\big]
\Big/
\text{a certain ideal of $41$ relations},
\endaligned
\]
where:
\[
\aligned
\Lambda^7
&
=
\Delta_{1,2}^{',\,''''}\,f_1'f_1'
+
\Delta_{1,2}^{'',\,'''}\,f_1'f_1'
-
10\,\Delta_{1,2}^{',\,'''}\,f_1'f_1''
+
15\,\Delta_{1,2}^{',\,''}\,f_1''f_1'',
\\
D^8
&
=
\Delta_{1,2,3}^{',\,''',\,''''}\,f_1'
-
3\,\Delta_{1,2,3}^{',\,'',\,''''}\,f_1'',
\\
N^{10}
&
=
\Delta_{1,2,3}^{',\,''',\,''''}\,f_1'f_1'
-
3\,\Delta_{1,2,3}^{',\,'',\,''''}\,f_1'f_1''
+
4\,\Delta_{1,2,3}^{',\,'',\,'''}\,f_1'f_1'''
+
3\,\Delta_{1,2,3}^{',\,'',\,'''}\,f_1''f_1'',
\endaligned
\]
where $W^{10} = \Delta_{ 1, 2, 3, 4}^{',\, '',\, ''',\, ''''}$ is the
four-dimensional Wronskian and where the eight remaining bi-invariants
defined by:
\[
\aligned
&
M^8
:=
{\textstyle{\frac{-5\,\Lambda^5\Lambda^5+3\,\Lambda^3\Lambda^7}{
f_1'f_1'}}}
\ \ \ \ \ \ \ \
E^{10}
:=
{\textstyle{\frac{-6\,\Lambda^5\,D^6+3\,\Lambda^3\,D^8}{f_1'}}},
\ \ \ \ \ \ \ \
L^{12}
:=
{\textstyle{\frac{-\Lambda^7D^6+5\,\Lambda^3N^{10}}{f_1'}}},
\\
&
Q^{14}
:=
{\textstyle{\frac{\Lambda^7D^8-10\,\Lambda^5N^{10}}{f_1'}}},
\ \ \ \ \ \ \ \
R^{15}
:=
{\textstyle{\frac{D^8D^8-12\,D^6N^{10}}{f_1'}}},
\ \ \ \ \ \ \ \
U^{17}
:=
{\textstyle{\frac{4\,D^8E^{10}+3\,\Lambda^3R^{15}}{f_1'}}},
\\
&
V^{19}
:=
{\textstyle{\frac{8\,N^{10}E^{10}+\Lambda^5R^{15}}{f_1'}}},
\ \ \ \ \ \ \ \
X^{21}
:=
{\textstyle{\frac{4\,D^8Q^{14}-5\,\Lambda^7R^{15}}{f_1'}}}
\endaligned
\]
happen all to be {\em true polynomials} in $\mathbb{ C} \big[ f_1',
\dots, f_4^{ ''''} \big]$, and where an explicit Gr\"obner basis, with
respect to the pure lexicographic term-order $f_1 ' < \Lambda^3 <
\cdots < X^{ 21}$, for the ideal of relations that they share, is
provided in \S11 of~\cite{ mer2008b}.

\end{itemize}
\end{theorem}

\noindent
For instance, the first three relations among the 41 
written just before the theorem of \S11 in~\cite{ mer2008b} are: 
\[
\label{41-syzygies}
\aligned
0
&
\overset{1}{\equiv}
-5\,\Lambda^5\Lambda^5
+
3\,\Lambda^3\Lambda^7
-f_1'f_1'M^8,
\\
0
&
\overset{2}{\equiv}
-2\,\Lambda^5D^6+
\Lambda^3D^8
-{\textstyle{\frac{1}{3}}}f_1'\,E^{10},
\\
0
&
\overset{3}{\equiv}
-\Lambda^7D^6
+
5\,\Lambda^3N^{10}
-f_1'L^{12}.
\endaligned
\]
Although the complexity of the algebra of bi-invariants increases
dramatically as soon as $n \geqslant 4$, one finds in~\cite{ mer2008b}
a complete algorithm which generates all bi-invariants together with
all the relations that they share, this in arbitrary dimension $n
\geqslant 1$ and for arbitrary jet order $k \geqslant 1$.

\subsection{Schur bundle decompositions}
In dimension 3, there are no relations between the four basic
bi-invariants $f_1'$, $\Lambda^3$, $\Lambda^5$ and $D^6$ and we
hence clearly have: 
\[
\aligned
\big(E_{k,m}T_{X^n,x_0}^*\big)^{{\rm U}_3(\mathbb{C})}
=
{\rm Span}_{\mathbb{C}}
\Big\{
&
(f_1')^a(\Lambda^3)^b(\Lambda^5)^c(D^6)^d
\colon
\\
&\ \
a,b,c,d\in\mathbb{N},\,\,
a+3b+5c+6d=m
\Big\}.
\endaligned
\]
Then to any such general monomial $(f_1')^a (\Lambda^3)^b
(\Lambda^5)^c (D^6)^d$ having weighted degree 
$m = a + 3b + 5 c + 6d$, the
representation theory of ${\rm GL}_n ( \mathbb{ C})$ tells us that
there corresponds the Schur bundle:
\[
\Gamma^{(a+b+2c+d,\,\,b+c+d,\,\,d)}T_X^*,
\]
just because the diagonal $3 \times 3$ matrices $t = {\rm diag} (t_1,
t_2, t_3)$ act as: $t \cdot f_i^{ (\lambda)} := t_i \, f_i^{
(\lambda)}$, whence:
\[
t\cdot f_1'
=
t_1f_1',
\ \ \ \ \ \ \
t\cdot\Lambda^3
=
t_1t_2\,\Lambda^3,
\ \ \ \ \ \ \
t\cdot\Lambda^5
=
t_1t_1t_2\,\Lambda^5,
\ \ \ \ \ \ \
t\cdot D^6
=
t_1t_2t_3\,D^6,
\]
so that indeed the three exponents of the $t_i$ in:
\[
t\cdot(f_1')^a(\Lambda^3)^b(\Lambda^5)^c(D^6)^d
=
t_1^{a+b+2c+d}\,\,t_2^{b+c+d}\,\,t_3^d\,\,
(f_1')^a(\Lambda^3)^b(\Lambda^5)^c(D^6)^d
\]
indicate the three corresponding integers in $\Gamma^{ (\lambda_1,
\lambda_2, \lambda_3)} T_X^*$. The same elementary process enables
one, in dimensions 2 and 4, to immediately deduce from the preceding
statement the following important decomposition theorem for the graded
bundle ${\rm Gr}^\bullet E_{ k,m} T_{X^n}^*$ associated to $E_{ k, m}
T_{ X^n}^*$, which is valuable without assuming that $X$ is
projective.

\begin{theorem}
\label{schur-decomposition}
Let $X$ be a compact complex manifold and let $m \in \mathbb{ N}$.
\begin{itemize}
\item
\cite{dem1997} If $\dim X=2$ then
\[
\operatorname{Gr}^{\bullet}
E_{2,m}T_X^*
=
\bigoplus_{a+3b=m}
\Gamma^{(a+b,\,\,b)}T_X^*.
\]
\item
\cite{rou2006a} If $\dim X=3$ then
\[
\operatorname{Gr}^{\bullet }E_{3,m}T_X^{\ast}
=
\bigoplus_{a+3b+5c+6d=m}\Gamma^{
(a+b+2c+d,\,\,b+c+d,\,\,d)}T_X^{\ast }.
\]
\item
\cite{mer2008b} If $\dim X=4$ then
\[
\aligned
&
\operatorname{Gr}^\bullet
E_{4,m}T_X^*
=
\bigoplus_{(a,b,c,d,e,f,g,h,i,j,k,l,m',n)\in\mathbb{N}^{14}\backslash
(\square_1\cup\cdots\cup\square_{41}) 
\atop
o+3a+5b+7c+6d+8e+10f+8g+10h+12i+14j+15k+17l+19m'+21n+10p\,=\,m}\,
\\
&
{\scriptsize
\Gamma
\left(
\aligned
o+a+2b+3c+d+2e+3f+2g+2h+3i+4j+3k+3l+4m'+5n+p&
\\
a+b+c+d+e+f+2g+2h+2i+2j+2k+3l+3m'+3n+p&
\\
d+e+f+h+i+j+2k+2l+2m'+2n+p&
\\
p&
\endaligned
\right)}\,T_X^*,
\endaligned
\]
where the 41 subsets $\square_i$, $i = 1, 2, \dots, 41$, of
$\mathbb{N}^{ 14} \ni (a, b, \dots, l, m', n)$ are explicitly defined
in \S12 of~\cite{mer2008b}.
\end{itemize}
\end{theorem}

\subsection{Euler-Poincar\'e characteristic of Schur bundles}
With $X = X^n \subset \mathbb{ P}^{ n+1}$ projective as before and
with $c_j = c_j (T_X)$ for $j =1, \dots, n$ being the Chern classes of
$T_X$ as in~\thetag{ \ref{c-d}}, a general asymptotic formula for the
Euler-Poincar\'e characteristic of a Schur bundle is given in \S13
of~\cite{ mer2008b} ({\em see} also Theorem~4 in~\cite{ bru1997}), and
for $n = 4$, this formula expands as:
\[
\chi\big(
X,\,
\Gamma^{(\ell_1,\ell_2,\ell_3,\ell_4)}\,T_X^*
\big)
=
\frac{c_1^4-3\,c_1^2c_2+c_2^2
+2\,c_1c_3-c_4}{0!\,\,1!\,\,2!\,\,7!}\
\left\vert
\begin{array}{cccc}
1\, & 1\, & 1\, & 1\,
\\
\ell_1\, & \ell_2\, & \ell_3\, & \ell_4\,
\\
\ell_1^2\, & \ell_2^2\, & \ell_3^2\, & \ell_4^2\,
\\
\ell_1^7\, & \ell_2^7\, & \ell_3^7\, & \ell_4^7\,
\end{array}
\right\vert
+
\]
\begin{equation}
\label{chi-Gamma}
+
\frac{c_1^2c_2-c_2^2-c_1c_3+c_4}
{0!\,\,1!\,\,3!\,\,6!}\
\left\vert
\begin{array}{cccc}
1\, & 1\, & 1\, & 1\,
\\
\ell_1\, & \ell_2\, & \ell_3\, & \ell_4\,
\\
\ell_1^3\, & \ell_2^3\, & \ell_3^3\, & \ell_4^3\,
\\
\ell_1^6\, & \ell_2^6\, & \ell_3^6\, & \ell_4^6\,
\end{array}
\right\vert
+
\frac{-c_1c_3+c_2^2}
{0!\,\,1!\,\,4!\,\,5!}\
\left\vert
\begin{array}{cccc}
1\, & 1\, & 1\, & 1\,
\\
\ell_1\, & \ell_2\, & \ell_3\, & \ell_4\,
\\
\ell_1^4\, & \ell_2^4\, & \ell_3^4\, & \ell_4^4\,
\\
\ell_1^5\, & \ell_2^5\, & \ell_3^5\, & \ell_4^5\,
\end{array}
\right\vert
+
\end{equation}
\[
+
\frac{c_1c_3-c_4}{0!\,\,2!\,\,3!\,\,5!}\
\left\vert
\begin{array}{cccc}
1\, & 1\, & 1\, & 1\,
\\
\ell_1^2\, & \ell_2^2\, & \ell_3^2\, & \ell_4^2\,
\\
\ell_1^3\, & \ell_2^3\, & \ell_3^3\, & \ell_4^3\,
\\
\ell_1^5\, & \ell_2^5\, & \ell_3^5\, & \ell_4^5\,
\end{array}
\right\vert
+
\frac{c_4}{1!\,\,2!\,\,3!\,\,4!}\
\left\vert
\begin{array}{cccc}
\ell_1\, & \ell_2\, & \ell_3\, & \ell_4\,
\\
\ell_1^2\, & \ell_2^2\, & \ell_3^2\, & \ell_4^2\,
\\
\ell_1^3\, & \ell_2^3\, & \ell_3^3\, & \ell_4^3\,
\\
\ell_1^4\, & \ell_2^4\, & \ell_3^4\, & \ell_4^4\,
\end{array}
\right\vert
+
{\rm O}\big(\vert\ell\vert^9\big).
\]
Of course, similar expanded\,\,---\,\,though 
shorter\,\,---\,\,formulae exist also in
dimensions 2 and 3, {\em cf.} again \S13 of~\cite{ mer2008b}. 

\subsection{Riemann-Roch computations} 
Recalling that the $n$-th power $h^n = d$ of the hyperplane class $h =
c_1 \big( \mathcal{ O}_{ \mathbb{ P}^{ n+1}} (1) \big)$ equals $\deg
X$, the formulae~\thetag{ \ref{c-d}} entail that any monomial $c_1^{
\lambda_1} c_2^{ \lambda_2} \cdots c_n^{ \lambda_n}$ whose weighted
homogeneous degree $\lambda_1 + 2 \lambda_2 + \cdots + n \lambda_n$
equals $n$ is a polynomial in $\mathbb{ Z} [ d]$, as are $c_1^4$,
$c_1^2 c_2$, $c_2 c_2$, $c_1 c_3$ and $c_4$ above when $n = 4$.  Basic
additivity, {\em e.g.} in dimension 3:
\[
\chi\big(X,\,E_{3,m}T_X^*\big)
=
\chi\big(X,\,{\rm Gr}^\bullet E_{3,m}T_X^*\big)
=
\sum_{a+3b+5c+6d=m}\,
\chi\big(X,\,\Gamma^{(a+b+2c+d,\,\,b+c+d,\,\,d)}T_X^*\big)
\]
enables one to deduce, by plain numerical summation and with some
electronic assistance, the following three Euler-Poincar\'e
characteristics, depending upon $m$ and $d$ only. We notice that the
summation of the three attached remainders, {\em e.g.} of ${\rm O}
\big( \vert \ell \vert^9 \big)$ in dimension 4, only contributes up to
a lower power of $m$, {\em e.g.}  up to an ${\rm O} \big( m^{
15}\big)$ in dimension 4.

\begin{theorem}
Let $X \subset \mathbb{ P}^{ n+1}$ be a smooth hypersurface 
of degree $d$.
\begin{itemize}
\item
\cite{dem1997} For $n=2${\rm :}
\[
\chi\big(X,E_{2,m}T_X^*\big)
=
\frac{m^4}{648}\,d\,
\big(4d^2-68d+154\big)
+
{\rm O}(m^3).
\]
\item
\cite{rou2006a} For $n=3${\rm :}
\[
\aligned
\chi\big(X,\,E_{3,m}T_X^{\ast}\big)
=
\frac{m^{9}}{81648\times 10^{6}}\,d\,
\big(389d^{3}-20739d^{2}+
185559d
&
-358873\big)
\\
&
+
{\rm O}(m^{8}).
\endaligned
\]
\item
\cite{mer2008b} For $n=4${\rm :}
\[
\footnotesize
\aligned
\chi\big(X,\,E_{4,m}T_X^*\big)
=
&
\frac{m^{16}}{1313317832303894333210335641600000000000000}\,
\cdot\,d\,\cdot
\\
&\ \ \ \ \
\cdot
\big(
50048511135797034256235\,d^4
-
\\
&\ \ \ \ \ \ \ \ \ \
-
6170606622505955255988786\,d^3
-
\\
&\ \ \ \ \ \ \ \ \ \
-
928886901354141153880624704\,d
+
\\
&\ \ \ \ \ \ \ \ \ \
+
141170475250247662147363941\,d^2
+
\\
&\ \ \ \ \ \ \ \ \ \
+
1624908955061039283976041114
\big)
\\
&\ \ \ \ \ \
+
{\rm O}\big(m^{15}\big).
\endaligned
\]
\end{itemize}
\end{theorem}

\subsection{The strategy of controlling the even cohomology dimensions} 
Remember from Theorem~\ref{A} that the first step towards the
algebraic degeneracy of entire curves $f\colon \mathbb{ C} \rightarrow
X$ consists in proving the existence of nonzero global sections in
$H^0 \big( X,\, E_{k,m} T_X^*\otimes A^{-1} \big)$, for some ample
line bundle $A\to X$, {\em e.g.} $A = \mathcal{ O}_X ( 1)$, and when
$A$ does not depend on $m$, the asymptotic cohomologies, as $m \to
\infty$, of the two bundles $E_{ k, m} T_X^*$ and $E_{ k, m} T_X^*
\otimes A^{ -1}$ coincide. So a quite natural strategy, followed by
the third-named author in~\cite{ rou2006b}, consists to rewrite
the characteristic, say
in dimension four: $\chi = h^0 - h^1 + h^2 - h^4$ under the form:
\[
\aligned
h^0
&
= 
\chi+h^1-h^2+h^3-h^4
\\
&
\geqslant
\chi\ \ \ \ \ \ \ \ \
-h^2\ \ \ \ \ \ \ \ \
-h^4,
\endaligned
\]
and to control asymptotically the dimensions $h_{ k, m}^{ 2i}$ of all
the even cohomology groups $H^{ 2i}( X, \, E_{k,m} T_X^* \otimes A^{
-1})$ by some vanishing theorem or by some appropriate inequalities
which would then show that these $h_{ k, m}^{ 2i}$ grow less rapidly
than the characteristic $\chi_{ k, m}$ as $m$ tends to $\infty$.  In
dimensions 2 and 4, the controls of the top even cohomology dimensions
$h^2$ and $h^4$ are obtained thanks to a vanishing theorem due to
Demailly which generalized a theorem of Bogomolov.

\begin{theorem}[\cite{dem1997}]
Let $X$ be a projective algebraic manifold of dimension $n \geqslant 2$ and
let $L$ be a holomorphic line bundle over $X$. Assume that $K_{X}$ is big
and nef and let $\mu = (\mu_1, \dots, \mu_n) \in \mathbb{ Z}^n$ be a weight
with $\mu_1 \geqslant \dots \geqslant \mu_n$. If either $L$ is
pseudo-effective and $\vert \mu \vert = \mu_1 + \cdots + \mu_n > 0$, or $L$
is big and $\vert \mu \vert \geqslant 0$, then{\rm :}
\[
H^0
\big(
X,\,\Gamma^{(\mu_1,\dots,\mu_n)}T_X\otimes L^*
\big)
=
0.
\]
\end{theorem}

Recall that if some $\mu_i$ is negative, we may use the identity:
\[
\Gamma^{(\mu_1,\dots,\mu_n)}T_X^*
=
\Gamma^{(\mu_1+l,\dots,\mu_n+l)}T_X^* \otimes
K_X^{-l}.
\]
For instance in dimension 4, we observe that the above 
vanishing theorem implies that
\[
h^4
\big(
X,\,E_{4,m}T_{ X}^*\otimes
A^{-1}
\big)
=
0,
\] 
for all $m$ sufficiently large; indeed, Serre duality and a
division by a tensor power of $K_X$ gives:
\[
\aligned
&
h^4
\big(
X,\,\Gamma^{(\lambda_1,\lambda_2,\lambda_3,\lambda_4)}T_X^*
\otimes A^{-1}
\big)
=
h^0
\big(
X,\,\Gamma^{(\lambda_1,\lambda_2,\lambda_3,\lambda_4)}T_X
\otimes
A
\otimes
K_X
\big)
\\
&
=
h^0
\Big(
X,\,\Gamma^{(\lambda_1-\nu,\lambda_2-\nu,
\lambda_3-\nu,\lambda_4-\nu)}T_X
\otimes
{K_X}^{1-\nu}
\otimes
\mathcal{O}(A)
\Big).
\endaligned
\]
But ${K_X}^{\nu-1} \otimes A^{-1}$ is big for $\nu$ large enough and then
the above theorem applies to provide the vanishing of $h^4$ as soon as:
\[
\vert\lambda\vert 
- 
4\nu 
\geqslant 0,
\]
which is satisfied for $m$ large enough since one easily convinces
oneself that $\vert \lambda \vert \geqslant \frac{4m}{10}$ in the
dimension 4 case of Theorem~\ref{schur-decomposition}.

\smallskip

However, it has been discovered by the third-named author \cite{rou2006b}
that already in dimension three, $H^2 \big( X, \, E_{ 3, m} T_X^* \big)
\neq 0$ does not vanish in general. Fortunately, a suitable majoration
holds.

\begin{theorem}[\cite{rou2006b}]
Let $X$ be a smooth hypersurface of degree $d$ in $\mathbb{ P}^4$. Then
for $\vert \lambda \vert$ large enough:
\begin{multline*}
h^2\big(X,\Gamma^{(\lambda_{1},\lambda_{2},\lambda_{3})}T_X^*\big)
\\
\leqslant d(d+13)\frac{3(\lambda_{1}+\lambda_{2}+\lambda_{3})^3}{2}
(\lambda_{1}-\lambda_{2})(\lambda_{1}-\lambda_{3})
(\lambda_{2}-\lambda_{3})
+
{\rm O}(\vert\lambda\vert^5).
\end{multline*}
\end{theorem}

In dimension $4$ the same proof provides the new estimate:

\begin{theorem}
\label{h2-4}
Let $X$ be a smooth hypersurface of degree $d$ in 
$\mathbb{ P}^5$. Then for $\vert \lambda \vert$ large enough, we have:
\[
\small
\aligned
h^2\big(X,\,
&
\Gamma^{(\lambda_{1},\lambda_{2},\lambda_{3},\lambda_{4})}T_X^*\big)
\\
&
\leqslant
\frac{1}{80}\,d\,
(\lambda_1-\lambda_2)(\lambda_1-\lambda_3)(\lambda_1-\lambda_4)
(\lambda_2-\lambda_3)(\lambda_2-\lambda_4)(\lambda_3-\lambda_4)
\\
&
\cdot
\big(\lambda_1+\lambda_2+\lambda_3+\lambda_4\big)^2
\big[5\lambda_2 \lambda_1d^2+132\lambda_2 \lambda_1d
+132\lambda_1 \lambda_3d+5\lambda_2 \lambda_3d^2
\\
&
+132\lambda_ 2 \lambda_4 d+5\lambda_2 d^2\lambda_4
+132\lambda_1 \lambda_4 d+5\lambda_3 \lambda_4 d^2
+5\lambda_1\lambda_3d^2
\\
&
+132\lambda_3 \lambda_4d+132\lambda_2 \lambda_3 d
+1308\lambda_2\lambda_1+648\lambda_2^2+648\lambda_3^2
\\
&
+72\lambda_3^2d+648\lambda_1^2+72\lambda_1^2d
+1308\lambda_1\lambda_4+5\lambda_1d^2\lambda_4
+1308\lambda_2\lambda_4
\\
&+1308\lambda_2\lambda_3+648\lambda_4^2
+72\lambda_2^2d+1308\lambda_1\lambda_3
+72\lambda_4^2d+1308\lambda_3\lambda_4\big]
\\
&
+
{\rm O}\big(\vert\lambda\vert^9\big).
\endaligned
\]
\end{theorem}

\proof
We follow~\cite{rou2006b} pp.~335-36, summarizing the main arguments
for the convenience of the reader.  The proof is essentially, again,
an application of holomorphic Morse inequalities and the reader will
notice strong similarities with the arguments presented in section
\ref{Section-2}.

Let $Y := Fl(T_{X}^{\ast })$ be the flag manifold of $T_{X}^{\ast }$
and let $\pi \colon Fl(T_{X}^{\ast })\rightarrow X$ the natural
projection. Let $\lambda = (\lambda_1,\lambda_2, \lambda_3,
\lambda_4)$ be a weight and $\mathcal{ L}^{\lambda }$ the line bundle
on $Y$ associated to $\Gamma^{\lambda}T_{X}^*$ such that
$\Gamma^\lambda T_{X}^* = \pi_* (\mathcal{L}^\lambda )$. By a theorem
of Bott, these bundles have the same cohomology ({\em
cf.}~\cite{rou2006b} p.~327) and therefore we are reduced to control
the cohomology of a line bundle. To this aim, we write:
\begin{equation*}
\mathcal{L}^{\lambda }
=
\big(
\mathcal{L}^{\lambda }\otimes \pi ^{\ast }\mathcal{O}_{X}
(3\left| \lambda \right|)
\big)\otimes 
\big(
\pi ^{\ast }\mathcal{O}_{X}(3\left|
\lambda \right| )
\big)^{-1}
=
F\otimes G^{-1},
\end{equation*}
with $F: = \mathcal{L}^{\lambda }\otimes \pi ^{\ast
}\mathcal{O}_{X}(3\left| \lambda \right| )$ and $G := \pi ^{\ast
}\mathcal{O}_{X}(3\left| \lambda \right| ).$ We observe that
$\mathcal{L}^{\lambda }\otimes \pi ^{\ast }\mathcal{O}_{X}(3\left|
\lambda \right| )$ is positive because $T_{X}^{\ast }\otimes
\mathcal{O}_{X}(2)$ is semi-positive (\cite{rou2006b}).

Recall also (\cite{rou2006b}) that we can write 
$K_Y=(\mathcal{L}^{\sigma})^{-1}
\otimes \pi^*K_X^5$ where $\sigma=(7,5,3,1)$, so:
\[
F \otimes K_Y^{-1} 
=
\mathcal{L}^{\lambda+\sigma} 
\otimes\pi ^{\ast }\mathcal{O}_{X}(3\left| \lambda \right| ) 
\otimes
\pi^*K_X^{-5}.
\] 
Then we still have the positivity $F \otimes K_Y^{-1} >0$ for
$|\lambda|$ large enough, because similarly as above, the line bundle:
\[
\mathcal{L}^{\lambda+\sigma} 
\otimes\pi ^{\ast }\mathcal{O}_{X}
\big(2\left| \lambda + \sigma \right|\big)
\]
is semi-positive as soon as $|\lambda| > 5(d-6)+32$.

Now, we take a smooth irreducible divisor $D_1$ in the linear series
$|G|$ of the form $\pi^* (E_1)$ for some divisor in $X$. On $Y$, we
have the exact sequence:
\[
0
\longrightarrow 
\mathcal{O}_{Y}\big(F\otimes G^{-1}\big)
\longrightarrow 
\mathcal{O}_{Y}(F)
\longrightarrow 
\mathcal{O}_{D_1}(F)
\longrightarrow 
0,
\]
and therefore in the associated long exact cohomology sequence:
\[
\aligned
0
=
H^{i}\big(Y,\mathcal{O}_{Y}(F)\big)
\longrightarrow 
H^{i}\big(D_1,\mathcal{O}_{D_1}(F)\big)
&
\longrightarrow 
H^{i+1}\big(Y,\mathcal{O}_{Y}(F\otimes G^{-1})\big)
\\
&
\longrightarrow
H^{i+1}\big(Y,\mathcal{O}_{Y}(F)\big)
=
0,
\endaligned
\]
both the first and last terms vanish for any $i>0$ by an application of
the Kodaira vanishing theorem. We at once deduce:
\[
h^{i}
\big(
D_1,\,
\mathcal{O}_{D_1}(F)
\big)
=
h^{i+1}\big(Y,\mathcal{O}_{Y}(F\otimes G^{-1})\big).
\]
Next, we take a second divisor $D_2 \in |G|$ intersecting properly
$D_1$ and of the form $\pi^* ( E_2)$ too.  Using the adjunction
formula and applying a similar restriction to $D_3 := D_1\cap D_2$
(word by word, the arguments are exactly the same as
in~\cite{rou2006b}, pp.~335--336, so we do not repeat the complete
proof), we obtain:
\[
h^2
\big(
Y,\mathcal{O}_{Y}(F\otimes G^{-1})
\big)
=
h^1
\big(
D_1,\mathcal{O}_{D_1}(F)\big)
\leqslant 
h^0
\big(D_3,\, 
\mathcal{O}_{D_3}(F\otimes G^2)\big)
=
\chi\big(D_3,\,\mathcal{O}_{D_3}(F\otimes G^2)\big).
\]
Letting $E_3 := E_1 \cap E_2$, one then shows (\cite{rou2006b}, p.~336)
that the latter Euler-Poincaré characteristic equals the following
linear combination of characteristics {\em on the
base $X$}:
\[
\aligned
&
\chi\big(D_3,\,\mathcal{O}_{D_3}(F\otimes G^2)\big)
=
\chi\big(E_3,\,
\Gamma^{(\lambda_1,\lambda_2,\lambda_3,\lambda_4)}T_X^*\vert_{E_3}
\otimes
\mathcal{O}_{E_3}(9\vert\lambda\vert)
\big)
\\
&
=
\chi\big(X,\,\Gamma^\lambda T_X^*
\otimes
\mathcal{O}_X(9\vert\lambda\vert)\big)
-2\,
\chi\big(X,\,\Gamma^\lambda T_X^*
\otimes
\mathcal{O}_X(6\vert\lambda\vert)\big)
+
\chi\big(X,\,\Gamma^\lambda T_X^*
\otimes
\mathcal{O}_X(3\vert\lambda\vert)\big).
\endaligned
\]
So the $h^2$ we want to majorate is less than
or equal to this last line. 
But then by applying a general complete 
combinatorial formula due to Br\"uckmann (Theorem~4
in~\cite{ bru1997}) for the characteristic $\chi \big( X, \,
\Gamma^\lambda T_X^* \otimes \mathcal{ O}_X ( t ) \big)$ of any
twisted Schur bundle over $X$, we may terminate the proof either by
hand or with the help of a computer.
\endproof

From such controls of higher cohomology groups, we deduce the
existence of global algebraic differential equations canalizing all
entire holomorphic maps: to obtain minorations of $h^0 \geqslant \chi
- h^2$, it suffices indeed as already explained to perform summations,
according to the representations of Theorem~\ref{schur-decomposition},
of the asymptotic Euler characteristics~\thetag{ \ref{chi-Gamma}},
subtracting at the same time the majorant of $h^2$ just obtained.  At
first, we recall here what is known in dimensions $2$ and $3$.  The
twisting $(\bullet) \otimes A^{ -1}$ by the negative of a fixed ample
line bundle $A \to X$ is erased in the asymptotics.

\begin{theorem}
Let $X\subset\mathbb{P}^{n+1}$ be a smooth hypersurface of degree $d$
and let $A$ be any ample line bundle over $X$. 
\begin{itemize}
\item
\cite{dem1997} For $n=2${\rm :}
\[
h^0\big(X,\,E_{2,m}T_X^*
\otimes
A^{-1}\big)
\geqslant
\frac{m^4}{648}\,d\,\big(4d^2-68d+154\big)
+
{\rm O}(m^3);
\] 
\item
\cite{rou2006b} For $n=3${\rm :}
\[
\aligned
h^0\big(
X,\,E_{3,m}T_X^{\ast }
\otimes
A^{-1}
\big)
&
\geqslant
\frac{m^9}{408240000000}
\cdot d\cdot
\big(
1945\,d^3-103695\,d^2
\\
&
\ \ \ \ \ \ \ \ \ \ \ \ \ \ 
-
7075491\,d-105837083\big)
+
{\rm O}(m^8).
\endaligned
\]
\end{itemize}
In particular, if $d \geqslant 15$ (resp. $d \geqslant 97$) then $E_{
2, m} T_X^* \otimes A^{ -1}$ (resp. $E_{3,m} T_X^* \otimes A^{ -1}$)
admits non trivial sections for $m$ large, and every entire curve $f
\colon \mathbb{ C} \rightarrow X$ must satisfy the corresponding
algebraic differential equations.
\end{theorem}

In dimension $4$, we may therefore present the following new result.

\begin{theorem}
\label{h0-4}
Let $X$ be a smooth hypersurface of degree $d$ in $\mathbb{ P}^5$ and
let $A$ be any ample line bundle over $X$. Then:
\[
\small
\aligned
h^0(X,\,
&
E_{4,m}T_X^*\otimes A^{-1})
\\
&
\geqslant
\frac{m^{16}}{1313317832303894333210335641600000000000000}\cdot d
\\
&\qquad\cdot
\big[-867659678949860838548185438614
\\
&\qquad
-93488069360760785094059379216\,d
\\
&\qquad
-1369327265177339103292331439\,d^2
\\
&\qquad
-6170606622505955255988786\,d^3
\\
&\qquad
+50048511135797034256235\,d^4
\big]
\\
&
\qquad\qquad
+
{\rm O}\big(m^{15}\big).
\endaligned
\]
In particular, if $d \geqslant 259$ then $E_{ 4,m }T_{ X }^*\otimes
A^{-1}$ admits non trivial sections for $m$ large, and every
entire curve $f \colon \mathbb{ C} \rightarrow X$ must satisfy the
corresponding algebraic differential equations.
\end{theorem}

\subsection{Algebraic degeneracy}  
Similarly as in Theorem~\ref{existenceKX} but say in dimension 4 to
fix ideas (for the dimension 3, {\em see}~\cite{rou2007},
pp.~381--383), one tensors the invariant jet bundle $E_{ k,m} T_X^*$
by $A^{ -1} := K_X^{ - \delta m}$, one uses the standard formula:
\[
\Gamma^{(\lambda_1,\lambda_2,\lambda_3,\lambda_4)}T_X^*
\otimes
K_X^{-\delta m}
= 
\Gamma^{(\lambda_1-\delta m,\lambda_2-\delta m,
\lambda_3-\delta m,\lambda_4-\delta m)}T_X^*
\]
in order to reapply the Schur bundle decomposition of
Theorem~\ref{schur-decomposition}, one redoes all the computations of
Theorem~\ref{h2-4} and of Theorem~\ref{h0-4}, and one gets in this way
a new minorant:
\[
h^0\big(X,\,E_{4,m}T_X^*\otimes K_X^{-\delta m}\big)
\geqslant
\alpha(d,\delta)\cdot m^{16}
+
{\rm O}\big(m^{15}\big),
\]
for a certain complicated polynomial $\alpha ( d, \delta) \in \mathbb{
Q} [ d, \delta]$ which we find now superfluous to write down
explicitly, and which of course regives for $\delta = 0$ the minorant
of Theorem~\ref{h0-4}. Remind now that according to
Theorem~\ref{globgen}, in dimension $4$, the maximal pole order of a
meromorphic frame on the space of vertical $4$-jets of the universal
hypersurface parametrizing all degree $d$ hypersurfaces of $\mathbb{
P}^5$ is equal to $4^2 + 2\cdot 4 = 24$. Then following line by line
the arguments of the proof of Theorem~\ref{noneffective-main}, in
order to be able to apply sufficiently many meromorphic derivations
$L_{ W_1} \cdots L_{ W_p}$ with $p \leqslant m$ to a given nonzero jet
differential so as to deduce\,\,---\,\,reasoning again by
contradiction as in Section~\ref{Section-3}\,\,---\,\,algebraic
degeneracy of entire curves, one has to insure: that $d > \frac{ 24}{
\delta} + 6$, as is required by~\thetag{ \ref{C2}} for the general
dimension $n$; and simultaneously also: that $\alpha ( d, \delta) >0$
for all $d \geqslant d_4$ larger than a certain effective $d_4 \in
\mathbb{ N}$.  But quite similarly as in the dimension 3 case, these
two constraints happen to be {\em compatible}, and thanks to effective
computations executed independently on two digital computers by the
second and by the third named author using different codes, one
verifies in dimension $4$ that $d_4 = 3203$ works (with $\delta =
\frac{ 3197}{ 24}$), and this is how, after so many rational
calculations, one gains the new effective lower bound $\deg X
\geqslant 3203$ of Theorem \ref{lowdim}.
\qed

\section{Effective algebraic degeneracy in dimensions 5 and 6}
\label{Section-7} 

Finally, for dimensions $5$ and $6$, we simply carry out the same
strategy as in the general case, but with a choice of weight different
from $a^\ast$ introduced in Subsection~\ref{highlighting}. Our choice
specific for these two dimensions are $\mathbf{ a} = (54, 18, 6, 2,
1)$ and $\mathbf{ a} = (162, 54, 18, 6, 2, 1)$, that is to say: the
minimal choice in order to have relative nefness of the weighted
(anti)tautological line bundle $\mathcal O_{ X_n} ( \mathbf{ a})$, $n
= 5, 6$ ({\em cf.}~\cite{dem1997, div2008a}); also, we choose $\delta
= \frac{ 5^2+2 \cdot 5}{ d-5-2 }$ and $\delta = \frac{ 6^2 + 2 \cdot
6}{ d - 6 - 2}$. The bound is then obtained thanks to computer
calculations with \textsc{gp/pari}, ({\em cf.}~\cite{div2008a} for the
code). The same method, in dimension 4 (resp.~3), would have produced
$\deg X \geqslant 6527$ (resp.~$\geqslant 1019$), less sharp than
$\deg X \geqslant 3203$ (resp.~$\geqslant 593$).

In dimension $n = 5$, here are the corresponding two polynomials ${\sf
P}_{ \mathbf{ a}} ( d)$ and ${\sf P}_{ \mathbf{ a}} ' ( d)$ the length
of which confirms the incompressible complexity of the reduction
process:
\begin{equation}
\label{P-54-18-6-2-1}
\footnotesize
\aligned
{\sf P}_{54,18,6,2,1}(d)
&
=
82970555252684668951323755447424\,d^6
-
\\
&\ \ \ \ \
-
69092357692382960198316008279615424\,d^5
-
\\
&\ \ \ \ \ \ \ \
-
37591957313184629697218108831955927744\,d^4
-
\\
&\ \ \ \ \ \ \ \ \ \ \
-
2161144497516080476955607837671278699584\,d^3
-
\\
&\ \ \ \ \ \ \ \ \ \ \ \ \
-
20767931723173741117548555837243163806144\,d^2
-
\\
&\ \ \ \ \ \ \ \ \ \ \ \ \ \ \
-
23736461779038166246115958304551871056384\,d.
\endaligned
\end{equation}
and:
\begin{equation}
\label{P-54-18-6-2-1-prime}
\footnotesize
\aligned
{\sf P}_{54,18,6,2,1}'(d)
&
=
-81064936492382180549906181650347200\,d^6
-
\\
&\ \ \ \ \ \
-25619265529443874657362851013713227200
\,d^5
-
\\
&\ \ \ \ \ \ \ \ \
-1138360224016877254137407566642735778400
\,d^4
-
\\
&\ \ \ \ \ \ \ \ \ \ \ \
-2649407942988198539201176162753240634400
\,d^3
+
\\
&\ \ \ \ \ \ \ \ \ \ \ \ \ \ \ \
+
70399558265933283202949942118101580280800
\,d^2
+
\\
&\ \ \ \ \ \ \ \ \ \ \ \ \ \ \ \ \ \ \ \
+
90355953106499854530169310985578945008800
\,d.
\endaligned
\end{equation}
We believe that the sequence of weights $\mathbf{ a} = (2\cdot
3^{n-2},\dots, 6, 2, 1)$ instead of
$a^\ast$ should work in any dimension, and that
it should provide better effective estimates in all dimensions, though
we suspect the bound should remain exponential. To conclude, we
collect our three effective estimates in a comparative table

\begin{center}
\begin{tabular}{*{6}{c}}
\hline 
\rule[-3pt]{0pt}{13.5pt}
{\bf dim} 
$\mathbf{X}$ & {\bf Theorem \ref{lowdim}} & {\bf Theorem \ref{main}} 
\\ 
\hline
\rule[-3pt]{0pt}{15.5pt}
3 & 593 & $2^{3^5}$
\\
\rule[-3pt]{0pt}{10.5pt}
4 & 3203 & $2^{4^5}$ 
\\
\rule[-3pt]{0pt}{10.5pt}
5 & 35355 & $2^{5^5}$ 
\\
\rule[-3pt]{0pt}{10.5pt}
6 & 172925 & $2^{6^5}$
\\ 
\hline
\end{tabular}
\end{center} 


\vfill\end{document}